\documentclass[11pt,reqno]{amsart}
\usepackage[utf8]{inputenc}
\usepackage[T1]{fontenc}
\usepackage{lmodern}
\usepackage[a4paper, total={6in, 9in}]{geometry}
\usepackage{amsmath}
\usepackage{amssymb}
\usepackage{amsfonts}
\usepackage{mathrsfs}
\usepackage{amsthm}
\usepackage{commath}
\usepackage{hyperref}
\usepackage{doi}
\usepackage{bigints}
\usepackage{xcolor}
\usepackage{dsfont}
\usepackage{graphicx}
\usepackage{enumitem}
\usepackage{import}
\usepackage{chngcntr}
\usepackage{appendix}

\newcommand{\R}{\mathbb{R}}
\newcommand{\Z}{\mathbb{Z}}

\newcommand{\D}{\mathcal{D}_{\mu}}

\newcommand{\E}{\mathcal{E}_{\mu,p}}
\newcommand{\HD}{\mathsf{HD}}
\newcommand{\LD}{\mathsf{LD}}
\newcommand{\BCE}{\mathsf{BCE}}
\newcommand{\Stop}{\mathsf{Stop}}
\newcommand{\Tree}{\mathsf{Tree}}
\newcommand{\Top}{\mathsf{Top}}
\newcommand{\Good}{\mathsf{Good}}
\newcommand{\Next}{\mathsf{Next}}
\newcommand{\Tr}{\mathsf{Tr}}
\newcommand{\Sep}{\mathsf{Sep}}
\newcommand{\one}{\mathds{1}}
\newcommand{\MD}{\mathcal{MD}}
\newcommand{\ID}{\mathsf{ID}}
\newcommand{\bbox}{T}
\DeclareMathOperator{\supp}{supp}
\DeclareMathOperator{\graph}{graph}
\DeclareMathOperator{\lip}{Lip}
\DeclareMathOperator{\spn}{span}
\newcommand\restr[2]{{
		\left.\kern-\nulldelimiterspace 
		#1 
		\right|_{#2} 
}}

\newcommand{\Hr}[1]{\restr{\mathcal{H}^n}{#1}}

\newcommand{\Hn}{\mathcal{H}^n}

\DeclareMathOperator{\diam}{diam}
\DeclareMathOperator{\dist}{dist}

\newtheorem{theorem}{Theorem}[section]
\newtheorem*{theorem*}{Theorem}
\newtheorem{lemma}[theorem]{Lemma}

\newtheorem{prop}[theorem]{Proposition}
\theoremstyle{remark}
\newtheorem{definition}[theorem]{Definition}
\newtheorem*{definition*}{Definition}
\newtheorem{question}[theorem]{Question}
\newtheorem{remark}[theorem]{Remark}

\numberwithin{equation}{section}
\counterwithin{figure}{section}

\title[Cones, rectifiability, and SIO's]{Cones, rectifiability, and singular integral operators}

\author[D. D\k{a}browski]{Damian D\k{a}browski}
\address{Damian D\k{a}browski\newline\indent Departament de Matem\`atiques, Universitat Aut\`onoma de Barcelona; Barcelona Graduate School of Mathematics (BGSMath)\newline\indent Edifici C Facultat de Ci\`encies, 08193 Bellaterra, Barcelona, Catalonia, Spain.}
\address{Current address: Department of Mathematics and Statistics\\ University of Jyv\"askyl\"a,
	P.O. Box 35 (MaD)\\
	FI-40014 University of Jyv\"askyl\"a\\
	Finland}
\email{damian.m.dabrowski "at" jyu.fi}

\begin{document}
	\begin{abstract}
		Let $\mu$ be a Radon measure on $\R^d$. We define and study conical energies $\E(x,V,\alpha)$, which quantify the portion of $\mu$ lying in the cone with vertex $x\in\R^d$, direction $V\in G(d,d-n)$, and aperture $\alpha\in (0,1)$. We use these energies to characterize rectifiability and the big pieces of Lipschitz graphs property. Furthermore, if we assume that $\mu$ has polynomial growth, we give a sufficient condition for $L^2(\mu)$-boundedness of singular integral operators with smooth odd kernels of convolution type.
	\end{abstract}
	\maketitle

	\section{Introduction}
	Let $m<d$ be positive integers. Given an $m$-plane $V\in G(d,m)$, a point $x\in \R^d$, and $\alpha\in (0,1)$, we define
	\begin{equation*}
	K(x,V,\alpha) = \{y\in\R^d\ :\ \dist(y,V+x)<\alpha|x-y| \}.
	\end{equation*}
	That is, $K(x,V,\alpha)$ is an open cone centered at $x$, with direction $V$, and aperture $\alpha$. 
	
	Let $0<n<d$. It is well-known that if a set $E\subset \R^d$ satisfies for some $V\in G(d,d-n),\ \alpha\in (0,1),$ the condition
	\begin{equation}\label{eq:lip graph condition}
	x\in E\quad\quad\Rightarrow\quad\quad E\cap K(x,V,\alpha)=\varnothing,
	\end{equation}
	then $E$ is contained in some $n$-dimensional Lipschitz graph $\Gamma$, and $\lip(\Gamma)\le \frac{1}{\alpha}$, see e.g. \cite[Proof of Lemma 15.13]{mattila1999geometry}.
	
	To what extent can we weaken the condition \eqref{eq:lip graph condition} and still get meaningful information about the geometry of $E$? It depends on what we mean by ``meaningful information'', naturally. One could ask for the rectifiability of $E$, or if $E$ contains big pieces of Lipschitz graphs, or whether nice singular integral operators are bounded on $L^2(E)$. The aim of this paper is to answer these three questions.
	
	\subsection{Rectifiability}
	A measurable set $E\subset \R^d$ is $n$-rectifiable if there exists a countable number of Lipschitz maps $f_i:\R^n\to\R^d$ such that
	\begin{equation*}
	\Hn\big(E\setminus \bigcup_i f_i(\R^n)\big)=0,
	\end{equation*}
	where $\Hn$ denotes the $n$-dimensional Hausdorff measure.
	More generally, a Radon measure $\mu$ is said to be $n$-rectifiable if $\mu\ll\Hn$ and there exists an $n$-rectifiable set $E\subset\R^d$ such that $\mu(\R^d\setminus E)=0$.
	
	A measure-theoretic analogue of \eqref{eq:lip graph condition}, well-suited to the study of rectifiability, is that of an \emph{approximate tangent plane}. We recall the definition below.
	
	For $r>0$ we define the truncated cone
	\begin{equation*}
	K(x,V,\alpha,r) = K(x,V,\alpha)\cap B(x,r),
	\end{equation*} 
	and for $0<r<R$ we define the doubly truncated cone
	\begin{equation*}
	K(x,V,\alpha,r,R) = K(x,V,\alpha,R)\setminus K(x,V,\alpha,r) .
	\end{equation*}
	Given a Radon measure $\mu$ on $\R^d$ and $x\in\supp\mu$, the lower and upper densities of $\mu$ at $x$ are defined as
	\begin{equation*}
	\Theta_{*}^n(\mu,x) = \liminf_{r\to 0} \frac{\mu(B(x,r))}{r^n}\quad\text{and}\quad \Theta^{n,*}(\mu,x) = \limsup_{r\to 0} \frac{\mu(B(x,r))}{r^n}.
	\end{equation*}
	
	Recall that if $\mu$ is $n$-rectifiable, then $0<\Theta_{*}^n(\mu,x) = \Theta^{n,*}(\mu,x)<\infty$ for $\mu$-a.e. $x\in\supp\mu$. In that case we set $\Theta^n(\mu,x) := \Theta_{*}^n(\mu,x) = \Theta^{n,*}(\mu,x)$.
	
	\begin{definition}\label{def:approx tangent}
		We say that an $n$-plane $W\in G(d,n)$ is an \emph{approximate tangent plane} to a Radon measure $\mu$ at $x\in \supp\mu$ if $\Theta^{n,*}(\mu,x)>0$ and for every $\alpha\in (0,1)$
		\begin{equation}\label{eq:conical density}
		\lim_{r\to 0}\frac{\mu(K(x,W^{\perp},\alpha,r))}{r^n}=0.
		\end{equation}
	\end{definition}
	
	The following classical characterization of rectifiable measures holds.
	\begin{theorem}[{\cite[Theorem 9.1]{federer1947varphi}}]\label{thm:classical tangents}
	Let $\mu$ be finite Radon measure on $\R^d$ satisfying $0<\Theta^{n,*}(\mu,x)<\infty$ for $\mu$-a.e. $x\in\R^d$. Then, the following are equivalent:
	\begin{enumerate}[label=\alph*)]
		\item $\mu$ is $n$-rectifiable,
		\item for $\mu$-a.e. $x\in \R^d$ there is a unique approximate tangent plane to $\mu$ at $x$,
		\item for $\mu$-a.e. $x\in \R^d$ there is $W_x\in G(d,n)$ and $\alpha_x\in (0,1)$ such that
		\begin{equation}\label{eq:weaker conical density}
		\limsup_{r\to 0}\frac{\mu(K(x,W_x^{\perp},\alpha_x,r))}{r^n}<(\alpha_x)^n\,\varepsilon(n)\,\Theta^{n,*}(\mu,x),
		\end{equation}
	\end{enumerate}
	where $\varepsilon(n)$ is a small dimensional constant.
	\end{theorem}
	
	The results we prove in this paper are of similar nature. More precisely, we introduce and study \emph{conical energies}.
	\begin{definition}
		Suppose $\mu$ is a Radon measure on $\R^d$, and $x\in\supp\mu$. Let $V\in G(d,d-n),\ \alpha\in(0,1),\ 1\le p <\infty$ and $R>0$. We define the \emph{$(V,\alpha,p)$-conical energy of $\mu$ at $x$ up to scale $R$} as
		\begin{equation*}
		\E(x,V,\alpha,R) = \int_0^{R}\bigg(\frac{\mu(K(x,V,\alpha,r))}{r^n} \bigg)^p\ \frac{dr}{r}.
		\end{equation*}
		For $E\subset\R^d$ we set also $\mathcal{E}_{E,p}(x,V,\alpha,R)=\mathcal{E}_{\Hr{E},p}(x,V,\alpha,R)$.
	\end{definition}
	Note that the definition above depends on the dimension parameter $n$, so
	it would be more precise to say that $\E(x,V,\alpha,R)$ is the $n$-dimensional
	$(V,\alpha,p)$-conical energy. For the sake of brevity, throughout the paper we will consider $n$ to be fixed, and we will usually not point out this dependence. The same applies to other definitions.
	
	We are ready to state our first result.
	\begin{theorem}\label{thm:characterization rectif}
		Let $1\le p <\infty$.  Suppose $\mu$ is a Radon measure on $\R^d$ satisfying $\Theta^{n,*}(\mu,x)>0$ and $\Theta_*^n(\mu,x)<\infty$ for $\mu$-a.e. $x\in\R^d$. Assume that for $\mu$-a.e. $x\in\R^d$ there exists some $V_x\in G(d,d-n)$ and $\alpha_x\in (0,1)$ such that
		\begin{equation}\label{eq:pointwise energy finite sufficient}
		\E(x,V_x,\alpha_x,1)<\infty.
		\end{equation}
		Then, $\mu$ is $n$-rectifiable.
		
		Conversely, if $\mu$ is $n$-rectifiable, then for $\mu$-a.e. $x\in\R^d$ there exists $V_x\in G(d,d-n)$ such that for all $\alpha\in (0,1)$ we have
		\begin{equation}\label{eq:pointwise energy finite necessary}
		\E(x,V_x,\alpha,1)<\infty.
		\end{equation}
	\end{theorem}
	\begin{remark}
		The ``necessary'' part of \thmref{thm:characterization rectif} improves on \thmref{thm:classical tangents} in the following way. Existence of approximate tangents means that the conical density simply converges to 0, while \eqref{eq:pointwise energy finite necessary} means that the conical density satisfies a Dini-type condition, and converges to $0$ rather fast.
	\end{remark}
	\begin{remark}
		Concerning the ``sufficient'' part of \thmref{thm:characterization rectif}: clearly, condition \eqref{eq:weaker conical density} is weaker than \eqref{eq:pointwise energy finite sufficient}. However, \thmref{thm:characterization rectif} has the following advantage over \thmref{thm:classical tangents}: we only require $\Theta^{n,*}(\mu,x)>0$ and $\Theta_*^n(\mu,x)<\infty$ for our criterion to hold. In particular, we do not assume $\mu\ll\Hn$. It is not clear to the author how to show a criterion involving \eqref{eq:weaker conical density} or \eqref{eq:conical density} without assuming a priori $\mu\ll\Hn$.
	\end{remark}

	\begin{question}
		Suppose $\mu$ is a Radon measure on $\R^d$ satisfying $\Theta^{n,*}(\mu,x)>0$ and $\Theta_*^n(\mu,x)<\infty$ for $\mu$-a.e. $x\in\R^d$. Assume that for $\mu$-a.e. $x\in\R^d$ there is an approximate tangent plane to $\mu$ at $x$. Does this imply that $\mu$ is $n$-rectifiable?
	\end{question}
	
	Let us mention that in recent years many similar characterizations of rectifiable measures have been obtained. By ``similar'' we mean: the pointwise finiteness of a square function involving some flatness quantifying coefficients. The most famous coefficients of this type are $\beta$ numbers, first introduced in \cite{jones1990rectifiable} and further developed by David and Semmes \cite{david1991singular,david1993analysis}. A necessary condition for rectifiability that uses $\beta_p$ numbers was shown in \cite{tolsa2015characterization}, see \thmref{thm:beta necessary} for the precise statement. Its sufficiency (under various assumptions on densities of the measure) was proved in \cite{pajot1997conditions,azzam2015characterization,edelen2016quantitative,badger2016two}. Measures carried by rectifiable curves are studied using $\beta$ numbers in \cite{lerman2003quantifying,badger2015multiscale,badger2016two,azzam2016characterization,badger2017multiscale,martikainen2018boundedness,naples2020rectifiability}, see also the survey \cite{badger2018generalized}.
	
	Finiteness of a square function involving $\alpha$ coefficients (defined in  \cite{tolsa2008uniform}) is shown to be necessary for rectifiability in \cite{tolsa2015characterization}. The opposite implication is studied in \cite{azzam2016wasserstein,orponen2018absolute,azzam2018characterization}. In \cite{dabrowski2019necessary,dabrowski2019sufficient} rectifiable measures were characterized using $\alpha_2$ numbers, first defined in \cite{tolsa2012mass}. Square functions involving centers of mass are studied in \cite{MAYBORODA2009} and \cite{villa2019square}. Finally, \cite{tolsa2014rectifiability,tolsa2017rectifiable} are devoted to a square function involving $\Delta$ numbers, where $\Delta_{\mu}(x,r) = |\frac{\mu(B(x,r))}{r^n} - \frac{\mu(B(x,2r))}{(2r)^n}|$.
	
	For related characterizations of rectifiable measures in terms of {tangent measures}, see \cite[Chapter 16]{mattila1999geometry} and \cite[Section 5]{preiss1987}. For a study of tangent points of Jordan curves in terms of $\beta$ numbers see \cite{bishop1994}, and for a generalization of this result for lower content regular sets of arbitrary dimension see \cite{villa2018tangent}. 
	
	The behaviour of conical densities on purely unrectifiable sets is studied in \cite{csornyei2010upper} and \cite[\S 5]{kaenmaki2010upper}. In \cite{mattila1988distribution,kaenmaki2008conical,csornyei2010upper,kaenmaki2011nonsymmetric} the relation between conical densities for higher dimensional sets and their porosity is investigated.
	
	Higher order rectifiability in terms of approximate differentiability of sets is studied in \cite{santilli2019}. In \cite{del2019geometric} the authors characterize $C^{1,\alpha}$ rectifiable sets using approximate tangents paraboloids, essentially obtaining a $C^{1,\alpha}$ counterpart of \thmref{thm:classical tangents}. See also \cite{ghinassi2017sufficient} and \cite{ghinassi2020menger} for related results.
	
	We would also like to mention recent results of Badger and Naples that nicely complement \thmref{thm:characterization rectif}. In \cite[Theorem D]{naples2020rectifiability} Naples showed that a modified version of \eqref{eq:conical density} can be used to characterize pointwise doubling measures \emph{carried by Lipschitz graphs}, that is measures vanishing outside of a countable union of $n$-dimensional Lipschitz graphs. In an even more recent paper \cite{badger2020radon} the authors completely describe measures carried by $n$-dimensional Lipschitz graphs on $\R^d$. They use a Dini condition imposed on the so-called \emph{conical defect}, and their condition is closely related to \eqref{eq:pointwise energy finite sufficient}. Note the absence of densities in the assumptions (and conclusion) of their results. If one adds an assumption $\Theta_*^n(\mu,x)<\infty$ for $\mu$-a.e. $x\in\R^d$, then it actually follows from \cite{badger2020radon} that $\mu$-a.e. finiteness of their conical Dini function implies that $\mu$ is $n$-rectifiable. We would like to stress however that neither \thmref{thm:characterization rectif} implies the results from \cite{badger2020radon}, nor the other way around.
	
	
	
	
	
	\subsection{Big pieces of Lipschitz graphs}
	Before stating our next theorem, we need to recall some definitions.
	\begin{definition}
		We say that $E\subset\R^d$ is \emph{$n$-Ahlfors-David regular} (abbreviated as $n$-ADR) if there exist constants $C_0, C_1>0$ such that for all $x\in E$ and $0<r<\diam(E)$
		\begin{equation*}
		C_0\, r^n\le \Hn(E\cap B(x,r))\le C_1\, r^n.
		\end{equation*}
		Constants $C_0, C_1$ will be referred to as ADR constants of $E$.
	\end{definition}
	\begin{definition}\label{def:BPLG}
		We say that an $n$-ADR set $E\subset\R^d$ has \emph{big pieces of Lipschitz graphs} (BPLG) if there exist constants $\kappa, L>0,$ such that the following holds.
		 
		For all balls $B$ centered at $E$, $0<r(B)<\diam(E),$ there exists a Lipschitz graph $\Gamma_B$ with $\text{Lip}(\Gamma_B)\le L$, such that
		\begin{equation*}\label{eq:kappa big piece}
		\Hn(E\cap B\cap \Gamma_B)\ge \kappa\, r(B)^n.
		\end{equation*}
	\end{definition}
	Sets with BPLG were studied e.g. in \cite{david1988operateurs, david1993analysis, david1993quantitative} as one of the possible quantitative counterparts of rectifiability. Let us point out that the class of sets with BPLG is strictly smaller than the class of uniformly rectifiable sets, introduced in the seminal work of David and Semmes \cite{david1991singular, david1993analysis}. An example of a uniformly rectifiable set that does not contain BPLG is due to Hrycak, although he never wrote it down, see \cite[Appendix]{azzam2019semi}.
	
	While there are available many characterizations of uniformly rectifiable sets, the sets containing BPLG are not as well understood. David and Semmes showed in \cite{david1993quantitative} that a set contains BPLG if and only if it has \textit{big projections} and satisfies the \textit{weak geometric lemma}. We refer the reader to \cite{david1993quantitative} or \cite[\S I.1.5]{david1993analysis} for details. 
	
	Very recently, Orponen characterized the BPLG property in terms of the \textit{big projections in plenty of directions} property \cite{orponen2020plenty}, answering an old question of David and Semmes. A little before that, Martikainen and Orponen \cite{martikainen2018characterising} characterized sets with BPLG in terms of $L^2$ norms of their projections. Interestingly, the authors use the information about projections of an $n$-ADR set $E$ to draw conclusions about intersections with cones of some subset $E'\subset E$ with $\Hn(E')\approx\Hn(E)$. This in turn allows them to find a Lipschitz graph intersecting an ample portion of $E'$. We will use some of their techniques to prove a characterization of sets containing BPLG in terms of the following property.
	\begin{definition}
		Let $1\le p<\infty$. We say that a measure $\mu$ has \emph{big pieces of bounded energy for $p$}, abbreviated as BPBE($p$), if there exist constants $\alpha,\kappa, M_0>0$ such that the following holds. 
		
		For all balls $B$ centered at $\supp \mu$, $0<r(B)<\diam(\supp\mu),$ there exist a set $G_B\subset B$ with $\mu(G_B)\ge\kappa\, \mu(B)$, and a direction $V_B\in G(d,d-n)$, such that for all $x\in G_B$
		\begin{equation}\label{eq:BPLG condition in thm}
		\mathcal{E}_{\mu,p}(x,V_B,\alpha,r(B))=\int_0^{r(B)} \bigg(\frac{\mu(K(x,V_B,\alpha,r))}{r^n} \bigg)^{p}\ \frac{dr}{r}\le M_0.
		\end{equation}
	\end{definition}
	
	\begin{theorem}\label{thm:suff BPLG}
		Let $1\le p <\infty$. Suppose $E\subset\R^d$ is $n$-ADR. Then $E$ has BPLG if and only if $\restr{\Hn}{E}$ has BPBE($p$). 
	\end{theorem}
	\begin{remark}
		In particular, for $n$-ADR sets, the condition BPBE($p$) is equivalent to BPBE($q$) for all $1\le p, q<\infty$.
	\end{remark}
	\begin{remark}
		In fact, one can show that an a priori slightly weaker condition than BPBE is already sufficient for BPLG. To be more precise, in \eqref{eq:BPLG condition in thm} replace $K(x,V_B,\alpha,r)$ with $K(x,V_B,\alpha,r)\cap G_B$, so that we get
		\begin{equation}\label{eq:BPLG weaker condition}
		\int_0^{r(B)} \bigg(\frac{\Hn(K(x,V_B,\alpha,r)\cap E\cap G_B)}{r^n} \bigg)^p\ \frac{dr}{r}\le M_0.
		\end{equation}
		We show that this ``weak'' BPBE is sufficient for BPLG in \propref{prop:suff BPLG}. It is obvious that \eqref{eq:BPLG weaker condition} is also necessary for BPLG: if $E$ contains BPLG, then choosing $G_B=\Gamma_B$ as in Definition \ref{def:BPLG}, one can pick the corresponding $V_B$ and $\alpha$ so that $K(x,V_B,\alpha,r)\cap\Gamma_B =\varnothing.$
	\end{remark}
	
	It is tempting to consider also the following definition.
	\begin{definition}
		Let $1\le p<\infty$. We say that a measure $\mu$ has \emph{bounded mean energy} (BME) for $p$ if there exist constants $\alpha, M_0>0$, and for every $x\in \supp\mu$ there exists a direction $V_x\in G(d,d-n)$, such that the following holds. 
		
		For all balls $B$ centered at $\supp \mu$, $0<r(B)<\diam(\supp\mu),$ we have 
		\begin{equation*}
		\int_{B}\mathcal{E}_{\mu,p}(x,V_x,\alpha,r(B))\ d\mu(x)\\
		= \int_{B}\int_0^{r(B)} \bigg(\frac{\mu(K(x,V_x,\alpha,r))}{r^n} \bigg)^p\ \frac{dr}{r}d\mu(x) \le M_0\, \mu(B).
		\end{equation*}
	\end{definition}
	In other words we require $\mu(K(x,V_x,\alpha,r))^pr^{-np} \frac{dr}{r}d\mu(x)$ to be a Carleson measure. This condition looks quite natural due to many similar characterizations of uniform rectifiability, e.g. the geometric lemma of \cite{david1991singular, david1993analysis} or the results from \cite{tolsa2008uniform, tolsa2012mass}. 
	
	It is easy to see, using the compactness of $G(d,d-n)$ and Chebyshev's inequality, that BME for $p$ implies BPBE($p$). However, the reverse implication does not hold. In \cite{dabrowski2020two} we give an example of a set containing BPLG that does not satisfy BME. The problem is the following. In the definition above, the plane $V_x$ is fixed for every $x\in\supp\mu$ once and for all, and we do not allow it to change between different scales. This is too rigid.
	\begin{question}
		Can one modify the definition of BME, allowing the planes $V_x$ to depend on the scale $r$, so that the modified BME could be used to characterize BPLG, or uniform rectifiability?
	\end{question}
	It seems likely that every uniformly rectifiable measure would satisfy such relaxed BME (the idea would be similar to what is done in Section \ref{sec:necessary rectifiability}: use the $\beta$-numbers characterization of UR to get an upper bound for $\beta$-numbers, and then estimate the measure of cones from above by the $\beta$-numbers). It is less clear whether this relaxed BME would imply uniform rectifiability. Perhaps additional control for the oscillation of $V_{x,r}$ would be needed.
	\subsection{Boundedness of SIOs}
	We will be concerned with singular integral operators of convolution type, with odd $C^2$ kernels $k:\R^d\setminus\{0\}\to\R$ satisfying for some constant $C_k>0$
	\begin{equation}\label{eq:calderon zygmund constant}
	|\nabla^j k(x)|\le \frac{C_k}{|x|^{n+j}}\quad \text{for $x\neq 0$}\quad \text{and}\quad j\in\{0,1,2\}.
	\end{equation}
	We will denote the class of all such kernels by $\mathcal{K}^n(\R^d)$. Note that these kernels are particularly nice examples of Calder\'{o}n-Zygmund kernels (see \cite[p. 48]{tolsa2014analytic} for definition), which will let us use many tools from the Calder\'{o}n-Zygmund theory. Since the measures we work with may be non-doubling, our main reference will be \cite[Chapter 2]{tolsa2014analytic}. For the more classical theory, we refer the reader to \cite[Chapter 5]{grafakos2014classical}, \cite[Chapter 4]{grafakos2014modern}.
	
	\begin{definition}\label{def:SIO}
		Given a kernel $k\in\mathcal{K}^n(\R^d)$, a constant $\varepsilon>0$, and a (possibly complex) Radon measure $\nu$, we set
		\begin{equation*}
		T_{\varepsilon}\nu(x) = \int_{|x-y|>\varepsilon}k(y-x)\ d\nu(y),\quad x\in\R^d.
		\end{equation*}
		For a fixed positive Radon measure $\mu$ and all functions $f\in L^1_{loc}(\mu)$ we define
		\begin{equation*}
		T_{\mu,\varepsilon}f(x) = T_{\varepsilon}(f\mu)(x).
		\end{equation*}
		We say that $T_{\mu}$ is bounded in $L^2(\mu)$ if all $T_{\mu,\varepsilon}$ are bounded in $L^2(\mu)$, uniformly in $\varepsilon>0$. Let $M(\R^d)$ denote the space of all finite real Borel measures on $\R^d$. When endowed with the total variation norm $\lVert \cdot\rVert_{TV}$, this is a Banach space. We say that $T$ is bounded from $M(\R^d)$ to $L^{1,\infty}(\mu)$ if there exists a constant $C$ such that for all $\nu\in M(\R^d)$ and all $\lambda>0$
		\begin{equation*}
		\mu(\{x\in\R^d\ :\ |T_{\varepsilon}\nu(x)|>\lambda \})\le \frac{C\lVert \nu\rVert_{TV}}{\lambda},
		\end{equation*}
		uniformly in $\varepsilon>0$.
	\end{definition}
	
	The main motivation for developing the theory of quantitative rectifiability was finding necessary and/or sufficient conditions for boundedness of singular integral operators. David and Semmes showed in \cite{david1991singular} that, for an $n$-ADR set, the $L^2$ boundedness of all singular integral operators with smooth and odd kernels is equivalent to uniform rectifiability. The famous David-Semmes problem asks whether the $L^2$ boundedness of a single SIO, the Riesz transform, is already sufficient for uniform rectifiability. It was shown that the answer is affirmative for $n=1$ in \cite{mattila1996cauchy}, for $n=d-1$ in \cite{nazarov2014onthe}, and the problem is open for other $n$.
	
	In the non-ADR setting less is known. A necessary condition for the boundedness of SIOs in $L^2(\mu)$, where $\mu$ is Radon and non-atomic, is the \emph{polynomial growth condition}:
	\begin{equation}\label{eq:growth condition}
	\mu(B(x,r))\le C_1\, r^n\quad\quad\quad\text{for all $x\in\supp\mu,\ r>0,$}
	\end{equation}
	see \cite[Proposition 1.4 in Part III]{david1991wavelets}. Eiderman, Nazarov and Volberg showed in \cite{eiderman2014s} that if $\mu$ is a measure on $\R^2$, $\mathcal{H}^1(\supp\mu)<\infty$, and $\mu$ has vanishing lower $1$-density, then the Riesz transform is unbounded. Their result was generalized to SIOs associated to gradients of single layer potentials in \cite{conde-alonso2019failure}. Nazarov, Tolsa and Volberg proved in \cite{nazarov2014} that if $E\subset\R^{n+1}$ satisfies $\Hn(E)<\infty$ and the $n$-dimensional Riesz transform is bounded in $L^2(\Hr{E})$, then $E$ is $n$-rectifiable. That the same is true for gradients of single layer potentials was shown by Prat, Puliatti and Tolsa in \cite{prat2018L2}.
	
	Concerning sufficient conditions for boundedness of SIOs, in \cite{azzam2015characterization} Azzam and Tolsa estimated the Cauchy transform of a measure using its $\beta$ numbers. Their method was further developed by Girela-Sarri\'{o}n \cite{girela-sarrion2018}. He gives a sufficient condition for boundedness of singular integral operators with kernels in $\mathcal{K}^n(\R^d)$ in terms of $\beta$ numbers. We use the main lemma from \cite{girela-sarrion2018} to prove the following criterion involving $2$-conical energy.

	\begin{theorem}\label{thm:SIO theorem}
		Let $\mu$ be a Radon measure on $\R^d$ satisfying the polynomial growth condition \eqref{eq:growth condition}.	
		Suppose that $\mu$ has BPBE(2).
		Then, all singular integral operators $T_{\mu}$ with kernels $k\in\mathcal{K}^n(\R^d)$ are bounded in $L^2(\mu)$, with norm depending only on BPBE constants, the polynomial growth constant $C_1$, and the constant $C_k$ from \eqref{eq:calderon zygmund constant}.
	\end{theorem}
	
	\begin{remark}
	A similar result, with BPBE(2) condition replaced by BPBE(1) condition, has already been shown in \cite[Theorem 10.2]{chang2017analytic}. It is easy to see that for measures satisfying polynomial growth \eqref{eq:growth condition} we have
	\begin{equation*}
	\mathcal{E}_{\mu,2}(x,V,\alpha,R)\le C_1\,\mathcal{E}_{\mu,1}(x,V,\alpha,R),
	\end{equation*}
	and so BPBE(2) is a weaker assumption than BPBE(1). Moreover, in \cite{dabrowski2020two} we show that the measure constructed in \cite{joyce2000set} does not satisfy BPBE(1), but it trivially satisfies BPBE(2). Hence, \thmref{thm:SIO theorem} really does improve on \cite[Theorem 10.2]{chang2017analytic}.
	\end{remark}
	\begin{remark}
		Recall that for $n$-ADR sets the condition BPBE($p$) was equivalent to BPLG, regardless of $p$. By the remark above, it is clear that if we replace the $n$-ADR condition with polynomial growth (i.e. if we drop the lower regularity assumption), then the condition BPBE($p$) is no longer independent of $p$. In general we only have one implication: for $1\le p< q<\infty$
		\begin{equation*}
		BPBE(p)\quad\Rightarrow\quad BPBE(q).
		\end{equation*}
	\end{remark}
	\begin{remark}
		\thmref{thm:SIO theorem} is sharp in the following sense. If one tried to weaken the assumption BPBE(2) to BPBE(p) for some $p>2$, then the theorem would no longer hold. The reason is that for any $p>2$ one may construct a Cantor-like probability measure $\mu$, say on a unit square in $\R^2$, that has linear growth and such that for all $x\in\supp\mu$
		\begin{equation*}
		\int_0^{1} \left(\frac{\mu(B(x,r))}{r}\right)^p\ \frac{dr}{r} \lesssim 1,
		\end{equation*}
		(that is, a much stronger version of BPBE($p$) holds), but nevertheless, the Cauchy transform is not bounded on $L^2(\mu)$. See \cite[Chapter 4.7]{tolsa2014analytic}.
	\end{remark}

	Sadly, the implication of \thmref{thm:SIO theorem} cannot be reversed. Let $E\subset\R^2$ be the previously mentioned example of a $1$-ADR uniformly rectifiable set that does not contain BPLG. In particular, by \thmref{thm:suff BPLG} $E$ does not satisfy BPBE($p$) for any $p$. Nevertheless, by the results of David and Semmes \cite{david1991singular}, all nice singular integral operators are bounded on $L^2(E)$.
	
	\subsection{Cones and projections}
	Let us note that \cite[Theorem 10.2]{chang2017analytic} was merely a tool to prove the main result of \cite{chang2017analytic}: a lower bound on analytic capacity involving $L^2$ norms of projections. Chang and Tolsa proved also an interesting inequality showing the connection between 1-conical energy and $L^2$ norms of projections. We introduce additional notation before stating their result.
	\begin{definition}
		Suppose $V\in G(d,d-n),$ $\alpha\in (0,1)$, and $1\le p<\infty$. 
		Let $B$ be a ball. The \emph{$(V,\alpha,p)$-conical energy of $\mu$ in $B$} is
		\begin{equation*}
		\E(B,V,\alpha) = \int_{B} \int_0^{r(B)}\bigg(\frac{\mu(K(x,V,\alpha,r))}{r^n} \bigg)^p\ \frac{dr}{r} d\mu(x).
		\end{equation*}
		We define also
		\begin{equation*}
		\E(\R^d, V,\alpha) = \int_{\R^d} \int_0^{\infty}\bigg(\frac{\mu(K(x,V,\alpha,r))}{r^n} \bigg)^p\ \frac{dr}{r} d\mu(x).
		\end{equation*}
		We will often suppress the arguments $V,\alpha$, and write simply $\E(B),\ \E(\R^d)$.
	\end{definition}
	\begin{remark}
		For $p=1$ we have
		\begin{equation}\label{eq:conical 1 energy comparable to Riesz on cones}
		\int_0^{\infty}\frac{\mu(K(x,V,\alpha,r))}{r^n}\ \frac{dr}{r} = \int_{K(x,V,\alpha)}\int_{|x-y|}^{\infty}\frac{dr}{r^{n+1}}\ d\mu(y) = n^{-1}\int_{K(x,V,\alpha)} \frac{1}{|x-y|^n}\ d\mu(y),
		\end{equation}
		and so
		\begin{equation}\label{eq:E1 and conical Riesz energy}
		\mathcal{E}_{\mu,1}(\R^d,V,\alpha) = n^{-1} \int_{\R^d}\int_{K(x,V,\alpha)} \frac{1}{|x-y|^n}\ d\mu(y)d\mu(x).
		\end{equation}
		In their paper Chang and Tolsa were working with the expression from the right hand side above.
	\end{remark}
	Given $V\in G(d,m)$ we will denote by $\pi_V:\R^d\to V$ the orthogonal projection onto $V$, and by $\pi^{\perp}_V:\R^d\to V^{\perp}$ the orthogonal projection onto $V^{\perp}$. We endow $G(d,m)$ with the natural probability measure $\gamma_{d,m}$, see \cite[Chapter 3]{mattila1999geometry}, and with a metric $d(V,W) = \lVert \pi_V - \pi_W\rVert_{op}$, where $\lVert\cdot\rVert_{op}$ is the operator norm.
	We write $\pi_V\mu$ to denote the image measure of $\mu$ by the projection $\pi_V$. If $\pi_V\mu\ll\Hr{V}$, then we identify $\pi_V\mu$ with its density with respect to $\Hr{V}$, and $\lVert \pi_V\mu\rVert_{L^2(V)}$ denotes the $L^2$ norm of this density. Otherwise, we set $\lVert \pi_V\mu\rVert_{L^2(V)}=\infty$.
	\begin{prop}[{\cite[Corollary 3.11]{chang2017analytic}}]\label{prop:ChangTolsa ineq}
		Let $V_0\in G(d,n)$ and $\alpha>0$. Then, there exist constants $\lambda,C>1$ such that for any finite Borel measure $\mu$ in $\R^d$,
		\begin{equation*}
		\mathcal{E}_{\mu,1}(\R^d,V_0^{\perp},\alpha)\overset{\eqref{eq:E1 and conical Riesz energy}}{\approx} \int_{\R^d}\int_{K(x,V_0^{\perp},\alpha)} \frac{1}{|x-y|^n}\ d\mu(y)d\mu(x)\le C \int_{B(V_0,\lambda\alpha)}\lVert \pi_V\mu\rVert_{L^2(V)}^2\ d\gamma_{d,n}(V).
		\end{equation*}
	\end{prop}
	Let us note that a variant of this estimate was also proved in \cite{martikainen2018characterising}, for a measure of the form $\mu=\Hr{E}$, with $E$ a suitable set.
	
	The inequality converse to that of \propref{prop:ChangTolsa ineq} in general is not true, but it is not far off. Additional assumptions on $\mu$ are necessary, and one has to add another term to the left hand side. See \cite[Remark 3.12, Appendix A]{chang2017analytic}.
	
	In the light of results mentioned above, as well as the characterization of sets with BPLG from \cite{martikainen2018characterising}, the connection between $L^2$ norms of projections and cones is quite striking. Note that the proof of the Besicovitch-Federer projection theorem also involves careful analysis of measure in cones, see \cite[Chapter 18]{mattila1999geometry}. Exploring further the relationship between cones and projections would be very interesting.	
	\begin{question}
		Is it possible to obtain an inequality similar to that of Proposition \ref{prop:ChangTolsa ineq}, but with $\mathcal{E}_{\mu,2}$ on the left hand side, and some quantity involving $\pi_V\mu$ on the right hand side?
	\end{question}
	\subsection{Organization of the article}
	In Section \ref{sec:preliminaries} we introduce additional notation, and recall the properties of the David-Mattila lattice $\D$. In Section \ref{sec:main lemma} we state our main lemma, a corona decomposition-like result. Roughly speaking, it says that if a measure $\mu$ has polynomial growth, and for some $V\in G(d,d-n),\ \alpha\in(0,1)$ we have $\mathcal{E}_{\mu,p}(\R^d,V,\alpha)<\infty$, then we can decompose $\D$ into a family of trees such that:
	\begin{itemize}
		\item for every tree, $\mu$ is ``well-behaved'' at the scales and locations of the tree,
		\item we have a good control on the number of trees (see \eqref{eq:packing estimate}).
	\end{itemize}
	We prove the main lemma in Sections \ref{sec:construction of graph}--\ref{sec:Top}. Let us point out that in the case $p=1$ an analogous corona decomposition was already shown in \cite[Lemma 5.1]{chang2017analytic}. Our proof follows the same general strategy, but some key estimates had to be done differently (most notably the estimates in Section \ref{sec:LD estimate}).
	
	In Section \ref{sec:SIOs} we show how to use the main lemma and results from \cite{girela-sarrion2018} to get \thmref{thm:SIO theorem}. Sections \ref{sec:suff rectifiability} and \ref{sec:necessary rectifiability} are dedicated to the proof of \thmref{thm:characterization rectif}. The ``sufficient part'' follows from our main lemma, while the ``necessary part'' is deduced from the corresponding $\beta_2$ result of Tolsa \cite{tolsa2015characterization}. Finally, we prove \thmref{thm:suff BPLG} in Sections \ref{sec:suff BPLG} and \ref{sec:necessary BPLG}. To show the ``sufficient part'' we use the results from \cite{martikainen2018characterising}, whereas the ``necessary part'' follows from a simple geometric argument. 
	
	\subsection*{Acknowledgements} I would like to thank Xavier Tolsa for all his help and patience. I am also grateful to the anonymous referee for carefully reading the article, and for many helpful suggestions.
	
	I received support from the Spanish Ministry of Economy and Competitiveness, through the María de Maeztu Programme for Units of Excellence in R\&D (MDM-2014-0445), and also partial support from the Catalan Agency for Management of University and Research Grants (2017-SGR-0395), and from the Spanish Ministry of Science, Innovation and Universities (MTM-2016-77635-P).
	
	\section{Preliminaries}\label{sec:preliminaries}
	\subsection{Additional notation}	
	We will write $A\lesssim B$ if there exists some constant $C$ such that $A\le C B$. $A\approx B$ means that $A\lesssim B\lesssim A$. If the constant $C$ depends on some parameter $t$, we will write $A\lesssim_t B$. We usually omit the dependence on $n$ and $d$.
	
	$B(x,r)$ stands for the open ball $\{y\in\R^d\ :\ |y-x|<r\}$. On the other hand, if $B$ is a ball, then $r(B)$ denotes its radius.
	
	A characteristic function of a set $E\subset\R^d$ will be denoted by $\one_E$.
	
	Given a Radon measure $\mu$ and a ball $B=B(x,r)$, we set
	\begin{equation*}
	\Theta_{\mu}(B) = \Theta_{\mu}(x,r) = \frac{\mu(B)}{r^n}.
	\end{equation*}
	
	If $T$ is a singular integral operator as in Definition \ref{def:SIO}, then the associated maximal operator $T_*$ is defined as
	\begin{equation*}
	T_{*}\nu(x) = \sup_{\varepsilon>0}|T_{\varepsilon}\nu(x)|\quad\quad \text{for $\nu\in M(\R^d),\ x\in\R^d$}.
	\end{equation*}
	
	Given an $n$-plane $L$, $\pi_L$ will denote the orthogonal projection onto $L$, and $\pi^{\perp}_L$ will denote the orthogonal projection onto $L^{\perp}$. 
	
	Given two bounded sets $E,F\subset \R^d$, $\dist_H(E,F)$ will stand for the Hausdorff distance between $E$ and $F$.
	
	\subsection{David-Mattila lattice}\label{sec:david-mattila}
	In the proof of \thmref{thm:SIO theorem} we will use the lattice of ``dyadic cubes'' constructed by David and Mattila \cite{david2000removable}. Their construction depends on parameters $C_0>1$ and $A_0>5000C_0$. The parameters can be chosen in such a way that the following lemmas hold.
	
	\begin{lemma}[{\cite[Theorem 3.2, Lemma 5.28]{david2000removable}}]\label{lem:DM lattice}
		Let $\mu$ be a Radon measure on $\R^d$, $E=\supp\mu.$ There exists a sequence of partitions of $E$ into Borel subsets $Q,\ Q\in\mathcal{D}_{\mu,k},\ k\ge 0$, with the following properties:
		\begin{itemize}
			\item[(a)] For each integer $k\ge 0$, $E$ is the disjoint union of the ``cubes'' $Q$, $Q\in \mathcal{D}_{\mu,k}$, and if $k<l$, $Q\in\mathcal{D}_{\mu,l}$, and $R\in\mathcal{D}_{\mu,k}$, then either $Q\cap R = \varnothing$ or else $R\subset Q$.
			\item[(b)] The general position of the cubes $Q$ can be described as follows. For each $k\ge0$ and each cube $Q\in\mathcal{D}_{\mu,k}$, there is a ball $B(Q)=B(x_Q, r(Q)),$ such that 
			\begin{gather*}
			x_Q\in Q,\quad A_0^{-k}\le r(Q)\le C_0 A_0^{-k},\\
			E\cap B(Q)\subset Q\subset E\cap 28 B(Q) = E\cap B(x_Q,28 r(Q)),
			\end{gather*}
			and the balls $5B(Q), Q\in \mathcal{D}_{\mu,k},$ are disjoint.
			\item[(c)] Denote by $\D^{db}$ the family of doubling cubes, i.e. $Q\in\D= \bigcup_{k\ge 0}\mathcal{D}_{\mu,k}$ satisfying
			\begin{equation}\label{eq:doubling_balls}
			\mu(100 B(Q))\le C_0 \mu(B(Q)).
			\end{equation}
			Then, for any $R\in\D$ there exists a family $\{Q_i\}_{i\in I}\subset\D^{db}$ such that $Q_i\subset R$ and $\mu(R\setminus\bigcup_i Q_i)=0$.
		\end{itemize}		
	\end{lemma}
	For any $Q\in\D$ we denote by $\D(Q)$ the family of $P\in\D$ such that $P\subset Q$. Given $Q\in \mathcal{D}_{\mu,k}$ we set $J(Q) = k$ and $\ell(Q)=56 C_0 A_0^{-k}$. Note that $r(Q)\approx\ell(Q).$
	
	We define $B_Q = 28B(Q) = B(x_Q, 28\, r(Q)),$ so that
	\begin{equation*}
		E\cap \tfrac{1}{28}B_Q\subset Q\subset B_Q.
	\end{equation*}
	Note that if $Q\subset P$, then $B_Q\subset B_P$.
	 
	\begin{lemma}[{\cite[Lemma 2.4]{azzam2015characterization}}]\label{lem:density dropping for nondoubling}
		Suppose the cubes $Q\in \D,\ R\in\D,\ Q\subset R,$ are such that all the intermediate cubes $Q\subsetneq S\subsetneq R$ are non-doubling, i.e. $S\notin\D^{db}$. Then
		\begin{equation}\label{eq:density dropping}
		\Theta_{\mu}(100B(Q))\le (C_0A_0)^d A_0^{-9d(J(Q)-J(R)-1)}\Theta_{\mu}(100B(R)),
		\end{equation}
		 and
		\begin{equation*}
		\sum_{S\in\D: Q\subset S\subset R}\Theta_{\mu}(100B(S))\lesssim \Theta_{\mu}(100B(R)).
		\end{equation*}
	\end{lemma}
	Let us remark that the constant $9d$ in the exponent of \eqref{eq:density dropping} could be replaced by any other positive constant, if $C_0$ and $A_0$ are chosen suitably, see \cite[(5.30)]{david2000removable}
	\begin{lemma}[{\cite[Lemma 4.5]{chang2017analytic}}]\label{lem:doubling subcube of doubling}
		Let $R\in\D^{db}$. Then, there exists another doubling cube $Q\subsetneq R,\ Q\in\D^{db}$, such that
		\begin{equation*}
			\mu(Q)\approx\mu(R)\quad\text{and}\quad \ell(Q)\approx\ell(R).
		\end{equation*}
	\end{lemma}
	
	From now on we will treat $C_0$ and $A_0$ as absolute constants, and we will not track the dependence on them in our estimates.
	
	\section{Main lemma}\label{sec:main lemma}
	In order to formulate our main lemma we need to introduce some vocabulary.
	
	Let $\mu$ be a compactly supported Radon measure with polynomial growth \eqref{eq:growth condition}. Suppose $\D$ is the associated David-Mattila lattice, and assume that
	\begin{equation*}
	R_0 = \supp\mu \in\D
	\end{equation*}
	is the biggest cube. 
	
	Given a family of cubes $\Top\subset\D^{db}$ satisfying $R_0\in\Top$ we define the following families associated to each $R\in\Top$:
	\begin{itemize}
		\item $\Next(R)$ is the family of maximal cubes $Q\in\Top$ strictly contained in $R$,
		\item $\Tr(R)$ is the family of cubes $Q\in\D$ contained in $R$, but \emph{not} contained in any $P\in\Next(R)$.
	\end{itemize}
	Clearly, $\D = \bigcup_{R\in\Top} \Tr(R)$. Define
	\begin{equation*}
	\Good(R)=R\setminus\bigcup_{Q\in\Next(R)}Q.
	\end{equation*}

	\begin{lemma}[main lemma]\label{lem:main lemma}
		Let $\mu$ be a compactly supported Radon measure on $\R^d$. Suppose there exists $r_0>0$ such that for all $x\in\supp\mu,\ 0<r\le r_0,$ we have
		\begin{equation}\label{eq:growth condition main lemma}
		\mu(B(x,r))\le C_1 r^n.
		\end{equation}
		Assume further that for some $V\in G(d,d-n),$ $\alpha\in (0,1),$ and $1\le p<\infty$, we have $\E(\R^d,V,\alpha)<\infty.$
		
		Then, there exists a family of cubes $\Top\subset\D^{db}$, and a corresponding family of Lipschitz graphs $\{\Gamma_R\}_{R\in\Top}$, satisfying:
		\begin{itemize}
			\item[(i)] the Lipschitz constants of $\Gamma_R$ are uniformly bounded by a constant depending on $\alpha$,
			\item[(ii)] $\mu$-almost all $\Good(R)$ is contained in $\Gamma_R$,
			\item[(iii)] for all $Q\in\Tr(R)$ we have $\Theta_{\mu}(2B_Q)\lesssim\Theta_{\mu}(2B_R)$.
		\end{itemize}
		Moreover, the following packing condition holds:
		\begin{equation}\label{eq:packing estimate}
		\sum_{R\in\Top}\Theta_{\mu}(2B_R)^p\mu(R)\lesssim_{\alpha} (C_1)^p\mu(\R^d) + \E(\R^d,V,\alpha).
		\end{equation}
		The implicit constant does not depend on $r_0$.
	\end{lemma}
	We prove the lemma above in Sections \ref{sec:construction of graph}--\ref{sec:Top}. From this point on, until the end of Section \ref{sec:Top}, we assume that $\mu$ is a compactly supported Radon measure satisfying the growth condition \eqref{eq:growth condition main lemma}, and that there exist $V\in G(d,d-n),\ \alpha\in(0,1),\ 1\le p<\infty,$ such that
	\begin{equation*}
	\E(\R^d,V,\alpha)<\infty.
	\end{equation*}	
	For simplicity, in our notation we will suppress the parameters $V$ and $\alpha$. That is, we will write $\E(\R^d)=\E(\R^d,V,\alpha)$, as well as $K=K(0,V,\alpha)$, $K(x) = K(x,V,\alpha)$, and $K(x,r) = K(x,V,\alpha,r)$. Finally, given $0<r<R$, set
	\begin{equation*}
	K(x,r,R) = K(x,R)\setminus K(x,r).
	\end{equation*}
	\subsection*{Parameters} In the proof of \lemref{lem:main lemma} we will use a number of parameters. To make it easier to keep track of what depends on what, and at which point the parameters get fixed, we list them below. Recall that ``$C_1=C_1(C_2)$'' means that ``the value of $C_1$ depends on the value of $C_2$.'' 
	\begin{itemize}
		\item $A=A(p)> 1$ is the ``$\HD$'' constant, it is fixed in \lemref{lem:ID density}.
		\item $\tau=\tau(\alpha,t)$ is the ``$\LD$'' constant, it is fixed in \eqref{eq:small measure LD}.
		\item $M=M(\alpha)>1$ is the ``key estimate'' constant, it is chosen in \lemref{lem:key lemma}. 
		\item $\eta=\eta(M,t)\in (0,1)$ is the constant from the definition of $\mathcal{E}_{\mu,p}(Q)$ in \eqref{eq:defi EQ}, it is fixed in the proof of \lemref{lem:some LD estimate}.
		\item $t=t(M,\alpha)>M$ is the ``$t$-neighbour'' constant, see Section \ref{subsec:GammaR}. It is fixed just below \eqref{eq:t fixed}, but depends also on \lemref{lem:Sep* cubes are not subsets} and \lemref{lem:Q in Sep* and KP dont intersect}.
		\item $\Lambda=\Lambda(M)>2M$ is the constant from \lemref{lem:Lipschitz graph}.
		\item $\varepsilon = \varepsilon(\tau,\alpha,\eta)\in (0,1)$ is the ``$\BCE$'' constant, it is fixed in \lemref{lem:some LD estimate}.
	\end{itemize}
	
	\section{Construction of a Lipschitz graph \texorpdfstring{$\Gamma_R$}{Gamma\_R}}\label{sec:construction of graph}
	Suppose $R\in\D^{db}$. In this section we will construct a corresponding tree of cubes $\Tree(R)$, and a Lipschitz graph $\Gamma_R$ that ``approximates $\mu$ at scales and locations from $\Tree(R)$''; see \lemref{lem:Lipschitz graph}.
	\subsection{Stopping cubes} Consider constants $A\gg 1,\ 0<\varepsilon\ll\tau\ll 1$, and $0<\eta\ll 1$, which will be fixed later on. Given $Q\in\D$ we set
	\begin{equation}\label{eq:defi EQ}
	\E(Q) = \frac{1}{\mu(Q)}\int_{2B_Q}\int_{\eta r(Q)}^{\eta^{-1}r(Q)}\bigg(\frac{\mu(K(x,r))}{r^n}\bigg)^p\ \frac{dr}{r}d\mu(x).
	\end{equation}
	For any $R\in\D^{db}$ we define the following families of cubes:
	\begin{itemize}
		\item $\BCE_0(R)$, the family of big conical energy cubes, consisting of $Q\in\D(R)$ such that
		\begin{equation*}
		\sum_{Q\subset P\subset R} \E(P) > \varepsilon\,\Theta_{\mu}(2B_R)^p.
		\end{equation*}
		\item $\HD_0(R)$, the high density family, consisting of $Q\in\D^{db}(R)\setminus \BCE_0(R)$ such that
		\begin{equation*}
		\Theta_{\mu}(2B_Q) > A\,\Theta_{\mu}(2B_R).
		\end{equation*}
		\item $\LD_0(R)$, the low density family, consisting of $Q\in\D(R)\setminus \BCE_0(R)$ such that
		\begin{equation*}
		\Theta_{\mu}(2B_Q) < \tau\,\Theta_{\mu}(2B_R).
		\end{equation*}
	\end{itemize}
	We denote by $\Stop(R)$ the family of maximal (hence, disjoint) cubes from $\BCE_0(R)\cup\HD_0(R)\cup\LD_0(R)$, and we set $\BCE(R) = \BCE_0(R)\cap\Stop(R),$ $\HD(R) = \HD_0(R)\cap\Stop(R),$ $\LD(R) = \LD_0(R)\cap\Stop(R).$ 
	
	Note that the cubes in $\HD(R)$ are doubling (by the definition), while the cubes from $\LD(R)$ and $\BCE(R)$ may be non-doubling.
	
	We define $\Tree(R)$ as the family of cubes from $\D(R)$ which are not strictly contained in any cube from $\Stop(R)$ (in particular, $\Stop(R)\subset\Tree(R)$). Note that it may happen that $R\in\BCE(R)$, in which case $\Tree(R)=\{R\}$.
	
	Basic properties of cubes in $\Tree(R)$ are collected in the lemma below.
	
	\begin{lemma}\label{lem:density estimates on Tree}
		Suppose $Q\in\Tree(R)$. Then, 
		\begin{equation}\label{eq:upper density estimate on Tree}
		\Theta_{\mu}(2B_Q)\lesssim A\,\Theta_{\mu}(2B_R).
		\end{equation}
		Moreover, for $Q\in\Tree(R)\setminus\Stop(R)$
		\begin{align}
		\tau\,\Theta_{\mu}(2B_R)&\le \Theta_{\mu}(2B_Q),\label{eq:lower density estimate on Tree}\\
		\sum_{Q\subset P\subset R} \E(P) &\le \varepsilon\,\Theta_{\mu}(2B_R)^p. \label{eq:energy controlled in Tree}
		\end{align}
		Finally, for every $Q\in\Tree(R)$ there exists a doubling cube $P(Q)\in \Tree(R)\cap\D^{db}$ such that $Q\subset P(Q)$ and $\ell(P(Q))\lesssim_{A,\tau}\ell(Q)$. If $R\not\in\Stop(R)$, we have $P(Q)\in\Tree(R)\cap\D^{db}\setminus\Stop(R)$.
	\end{lemma}
	\begin{proof}
		First, note that if $R\in\Stop(R)$, then $\Tree(R)=\{R\}$ and the lemma above is trivial. Assume that $R\not\in\Stop(R)$.
		
		Inequalities \eqref{eq:lower density estimate on Tree} and \eqref{eq:energy controlled in Tree} are obvious by the definition $\LD(R)$ and $\BCE(R)$.
		
		Concerning \eqref{eq:upper density estimate on Tree}, note that for $Q\in \Tree(R)\cap\D^{db}\setminus\Stop(R)$ we have $\Theta_{\mu}(2B_Q)\le A\,\Theta_{\mu}(2B_R)$ by the high density stopping condition. In general, given $Q\in \Tree(R)$, let $P(Q)$ be the smallest doubling cube containing $Q$, other than $Q$. Since $R\in\D^{db}$ and $R\not\in\Stop(Q)$, we certainly have $P(Q)\in \Tree(R)\cap\D^{db}\setminus\Stop(R)$, and so $\Theta_{\mu}(2B_{P(Q)})\le A\,\Theta_{\mu}(2B_R)$. 
		
		Denote by $P_1, P_2,\dots, P_k$ all the intermediate cubes, so that $Q\subset P_1\subset\dots\subset P_k\subset P(Q)$. Since $P_j$ are non-doubling, we have by \lemref{lem:density dropping for nondoubling}
		\begin{multline*}
		\Theta_{\mu}(2B_Q)\lesssim\Theta_{\mu}(2B_{P_1})\lesssim\Theta_{\mu}(100B(P_1))\le(C_0 A_0)^d\, A_0^{-9d(k-1)}\Theta_{\mu}(100B(P(Q)))\\
		\lesssim\Theta_{\mu}(2B_{P(Q)})\le A\, \Theta_{\mu}(2B_R),
		\end{multline*}
		which proves \eqref{eq:upper density estimate on Tree}. 
		
		Finally, to see that $\ell(P(Q))\lesssim_{A,\tau}\ell(Q)$, note that $P_1\in\Tree(R)\setminus\Stop(R)$, and so $\tau\,\Theta_{\mu}(2B_R)\le \Theta_{\mu}(2B_{P_1})$. On the other hand, a minor modification of the computation above shows that
		\begin{equation*}
		\Theta_{\mu}(2B_{P_1})\lesssim_{C_0,A_0} A_0^{-9d(k-1)}A\, \Theta_{\mu}(2B_R).
		\end{equation*}
		It follows that $k\lesssim_{A,\tau} 1$.
	\end{proof}
	
	The following estimate of the measure of cubes in $\BCE(R)$ will be used later on in the proof of the packing estimate \eqref{eq:packing estimate}.
	\begin{lemma}
		We have
		\begin{equation}\label{eq:BCE estimate}
		\sum_{Q\in\BCE(R)}\mu(Q)\le \frac{1}{\varepsilon\,\Theta_{\mu}(2B_R)^p}\sum_{P\in\Tree(R)}\E(P)\mu(P).
		\end{equation}
	\end{lemma}
	\begin{proof}
		We use the fact that for $Q\in\BCE(R)$ we have
		\begin{equation*}
		\sum_{Q\subset P\subset R} \E(P) > \varepsilon\,\Theta_{\mu}(2B_R)^p
		\end{equation*}
		to conclude that
		\begin{multline*}
		\Theta_{\mu}(2B_R)^p\sum_{Q\in\BCE(R)}\mu(Q)\le \frac{1}{\varepsilon}\sum_{Q\in\BCE(R)}\mu(Q)\sum_{\substack{P\in\D\\ Q\subset P\subset R}}\E(P)\\
		 = \frac{1}{\varepsilon}\sum_{P\in\Tree(R)}\E(P)\sum_{\substack{Q\in\BCE(R)\\ Q\subset P}}\mu(Q)
		\le \frac{1}{\varepsilon}\sum_{P\in\Tree(R)}\E(P)\mu(P).
		\end{multline*}
	\end{proof}
	
	\subsection{Key estimate}
	We introduce some additional notation. Given $x\in\R^d$ and $\lambda>0$ set
	\begin{equation*}
	K^{\lambda}(x) = K(x,V,\lambda\alpha).
	\end{equation*}
	For $Q\in\D$, we denote
	\begin{equation*}
	K^{\lambda}_Q = \bigcup_{x\in Q} K^{\lambda}(x).
	\end{equation*}
	If $\lambda=1$, we will write $K_Q$ instead of $K^{1}_Q$.
	\begin{lemma}\label{lem:key lemma}
		There exists a constant $M=M(\alpha)>1$ such that, if $Q\in\Tree(R)$ and $P\in\D(R)$ satisfy 
		\begin{equation}
		P\cap K^{1/2}_Q\setminus MB_Q\neq\varnothing
		\end{equation}
		and 
		\begin{equation*}
		\dist(Q,P)\ge Mr(P),
		\end{equation*}
		then $P\not\in\Tree(R)$.
	\end{lemma}
	\begin{proof}
		Taking $M=M(\alpha)>1$ big enough, we can choose cubes $P',\ Q'\in\D(R)$ such that
		\begin{itemize}
			\item $P\subsetneq P'\subset R,\ P'\subset K_Q^{3/4},$ and $\ell(P')\approx \dist(P',Q)$,
			\item $Q\subsetneq Q'\subset R,\ \ell(Q')\approx M^{-1}\ell(P')$, and $\dist(P',Q')\approx \ell(P')$.
		\end{itemize}
		Moreover, if $M$ is taken big enough, we have for all $x\in 2B_{Q'}$
		\begin{equation*}
		2B_{P'}\subset K(x).
		\end{equation*}
		Thus, if $\eta$ is taken small enough (say, $\eta\ll M^{-1}$), we have
		\begin{equation}\label{eq:density P estimate}
		\bigg(\frac{\mu(2B_{P'})}{\ell(P')^{n}}\bigg)^p\mu(2B_{Q'})\lesssim_{\eta} \int_{2B_{Q'}} \int_{\eta r(Q')}^{\eta^{-1}r(Q')}\left(\frac{\mu(K(x,r))}{r^n}\right)^p\ \frac{dr}{r}d\mu(x) =\E(Q')\mu(Q').
		\end{equation}
		Since $Q\in\Tree(R)$ and $Q\subsetneq Q'$, we have $Q'\in\Tree(R)\setminus\Stop(R)$, and so
		\begin{equation*}
		\Theta_{\mu}(2B_{P'})^p\approx \bigg(\frac{\mu(2B_{P'})}{\ell(P')^{n}}\bigg)^p\overset{\eqref{eq:density P estimate}}{\lesssim_{\eta}} \frac{\mu(Q')}{\mu(2B_{Q'})}\E(Q')\le\E(Q')\overset{\eqref{eq:energy controlled in Tree}}{\le}\varepsilon\,\Theta_{\mu}(2B_R)^p.
		\end{equation*}
		It follows that, for $\varepsilon$ small enough, $P'\in\LD_0(R)$. Since $P\subsetneq P'$, we get that $P\notin\Tree(R)$.
	\end{proof}
	We set
	\begin{equation}\label{eq:def of GR}
	G_R = R\setminus\bigcup_{Q\in\Stop(R)} Q\qquad \text{and}\qquad \widetilde{G}_R = \bigcap_{k=1}^{\infty}\bigcup_{\substack{Q\in\Tree(R)\\ r(Q)\le A_0^{-k}}} 2MB_{Q}.
	\end{equation}
	Note that $G_R\subset\widetilde{G}_R$.
	
	\begin{lemma}\label{lem:xy from GR outside Kx}
		For all $x, y\in\widetilde{G}_R$ we have $y\not\in K^{1/2}(x)$. Thus, $\widetilde{G}_R$ is contained in an $n$-dimensional Lipschitz graph with Lipschitz constant depending only on $\alpha$.
	\end{lemma}
	\begin{proof}
		Proof by contradiction. Suppose that $x,y\in\widetilde{G}_R$ and $x-y\in K^{1/2}$.  Let $Q, P\in\Tree(R)$ be such that $x\in 2MB_Q,\ y\in 2MB_P$, with sidelength so small that $P\cap (K_Q^{1/2}\setminus MB_Q)\neq\varnothing$ and $\dist(Q,P)\ge Mr(P)$ (note that this can be done because $K^{1/2}$ is an open cone, and so $x'-y'\in K^{1/2}$ also for $x'\in B(x,\varepsilon')$ and $y'\in B(y,\varepsilon')$, assuming $\varepsilon'>0$ small enough). It follows by \lemref{lem:key lemma} that $P\notin\Tree(R)$, and so we reach a contradiction.
	\end{proof}

	\subsection{Construction of \texorpdfstring{$\Gamma_R$}{Gamma R}}\label{subsec:GammaR}
	The Lipschitz graph from \lemref{lem:xy from GR outside Kx} can be thought of as a first approximation of $\Gamma_R$. It contains the ``good set'' $\widetilde{G}_R$, but we would also like for $\Gamma_R$ to lie close to cubes from $\Tree(R)$. In this subsection we show how to do it.
	
	Given $t>1$, we say that cubes $Q,P\in\D$ are \emph{$t$-neighbours} if they satisfy
	\begin{equation}\label{eq:neighbour condition 1}
	t^{-1}\,r(Q)\le r(P)\le t\,r(Q)
	\end{equation}
	and
	\begin{equation}\label{eq:neighbour condition 2}
	\dist(Q,P)\le t(r(Q)+r(P)).
	\end{equation}
	If at least one of the conditions above does not hold, we say that $Q$ and $P$ are \emph{$t$-separated.} We will also say that a family of cubes is {$t$-separated} if the cubes from that family are pairwise $t$-separated.
	
	Consider a big constant $t=t(M,\alpha)>M$ which will be fixed later on. We denote by $\Sep(R)$ a maximal $t$-separated subfamily of $\Stop(R)$ (it exists by Zorn's lemma). Clearly, for every $Q\in\Stop(R)$ there exists some $P\in\Sep(R)$ which is a $t$-neighbour of $Q$.

	Furthermore, we define $\Sep^*(R)$ as the family of all cubes $Q\in\Sep(R)$ satisfying the following two conditions:
	\begin{equation}\label{eq:Sep* 1st condition}
	2MB_{Q}\cap\widetilde{G}_R=\varnothing,
	\end{equation}
	and for all $P\in\Sep(R),\ P\neq Q$, we have 
	\begin{equation}\label{eq:Sep* 2nd condition}
	2MB_P\not\subset 2MB_Q.
	\end{equation}
	\begin{lemma}\label{lem:Sep* cubes are not subsets}
		Suppose $t=t(M)$ is big enough. Then, for all $Q, P\in\Sep^*(R),\ Q\neq P$, we have $Q\not\subset1.5MB_P.$
	\end{lemma}
	\begin{proof}
		Suppose $Q\in\Sep^*(R)$, and $Q\subset 1.5MB_P$. We will show that $P\notin\Sep^*(R)$. 
		
		Firstly, if $r(Q)>t^{-1}r(P)$, then $Q\subset 1.5MB_P$ implies that $Q$ and $P$ are $t$-neighbours (for $t$ big enough), and so $P\notin\Sep^{*}(R)$. On the other hand, if $r(Q)\le t^{-1}r(P)$, then (if $t$ is big enough) $Q\subset 1.5MB_P$ implies $2MB_Q\subset 2MB_P$, contradicting \eqref{eq:Sep* 2nd condition}. 
	\end{proof}
	
	\begin{lemma}\label{lem:Sep contains Sep* cubes}
		For every $Q\in\Sep(R)$ at least one of the following is true:
		\begin{itemize}
			\item[(a)] $2MB_Q\cap\widetilde{G}_R\neq\varnothing$,
			\item[(b)] there exists $P\in \Sep^*(R)$ such that $2MB_P\subset 2MB_Q$.
		\end{itemize}
	\end{lemma}
	\begin{proof}
		If $Q\in\Sep^*(R)$, then of course (b) holds (with $P=Q$). Suppose that $Q\notin\Sep^*(R)$, and that (a) does not hold (i.e. $2MB_Q\cap\widetilde{G}_R=\varnothing$). We will find $P\in\Sep^*(R)$ such that $2MB_P\subset 2MB_Q$.
		
		Since $Q\notin\Sep^*(R)$ and \eqref{eq:Sep* 1st condition} holds, condition \eqref{eq:Sep* 2nd condition} must be false. Thus, we get a cube $Q_1\in\Sep(R)$ such that $2MB_{Q_1}\subset 2MB_Q$. If $Q_1\in\Sep^*(R)$, we get (b) with $P=Q_1$. Otherwise, we continue as follows. 
		
		Reasoning as before, $Q_1\in \Sep(R)\setminus\Sep^*(R)$ and $2MB_{Q_1}\cap\widetilde{G}_R=\varnothing$ ensures that there exists a cube $Q_2\in\Sep(R)$ such that $2MB_{Q_2}\subset 2MB_{Q_1}$. Iterating this process, we get a (perhaps infinite) sequence of cubes $Q_0 := Q,\, Q_1,\, Q_2,\,\dots$ satisfying $2MB_{Q_{j+1}}\subset 2MB_{Q_{j}}$.
		
		If the algorithm never stops, then $\bigcap_{j=0}^{\infty}2MB_{Q_j}\neq\varnothing$. But, by the definition of $\widetilde{G}_R$ \eqref{eq:def of GR} we have $\bigcap_{j=0}^{\infty}2MB_{Q_j}\subset \widetilde{G}_R,$ and so we get a contradiction with $2MB_Q\cap\widetilde{G}_R=\varnothing$. Thus, the algorithm stops at some cube $Q_m$, which means that $Q_m\in\Sep^{*}(R)$. Setting $P=Q_m$ finishes the proof.
	\end{proof}

	\begin{lemma}\label{lem:Q in Sep* and KP dont intersect}
		Suppose $t=t(M)$ is big enough. Then:
		\begin{itemize}
			\item[(a)] for all $Q,P\in\Sep^*(R),\ Q\neq P,$ we have 
			\begin{equation}\label{eq:Q and KP dont intersect}
			Q\cap K_P^{1/2} = P\cap K_Q^{1/2}=\varnothing,
			\end{equation}
			\item[(b)] for all $x\in\widetilde{G}_R$ and for all $Q\in\Sep^*(R)$ we have
			\begin{equation}
			x\notin K_Q^{1/2}\quad\text{and}\quad Q\cap K^{1/2}(x)=\varnothing.
			\end{equation}
		\end{itemize}
	\end{lemma}
	\begin{proof}[Proof of (a)] 
		Proof by contradiction. Suppose $Q\cap K_P^{1/2}\neq\varnothing$ (which by symmetry of cones implies $P\cap K_Q^{1/2}\neq\varnothing$). Without loss of generality, assume $r(Q)\le r(P)$. Since $Q$ and $P$ are $t$-separated, at least one of the conditions \eqref{eq:neighbour condition 1}, \eqref{eq:neighbour condition 2} fails, i.e.
		\begin{equation*}
		r(Q)\le t^{-1}r(P)\quad\text{or}\quad\dist(Q,P)>t(r(Q)+r(P)).
		\end{equation*}
		
		We know by \lemref{lem:Sep* cubes are not subsets} that $Q\not\subset 1.5MB_P$. It is easy to see that in either of the cases considered above, this implies $Q\cap 1.2MB_P = \varnothing$. It follows that $Q\cap(K_P^{1/2}\setminus MB_P)\neq\varnothing$ and $r(Q)\le r(P)\le M^{-1}\dist(Q,P)$. Hence, we can use \lemref{lem:key lemma} to conclude that $Q\notin\Tree(R)$. This contradicts $Q\in\Sep^*(R)$.
	\end{proof}
	\begin{proof}[Proof of (b)]
		Proof by contradiction. Suppose $x\in K_Q^{1/2}$. We have $x\notin 2MB_Q$ by \eqref{eq:Sep* 1st condition}. Since $x\in\widetilde{G}_R$, we can find an arbitrarily small cube $P\in\Tree(R)$ such that $x\in 2MB_P$. Taking $r(P)$ small enough we will have $r(P)\le M^{-1}\dist(Q,P)$ and $P\cap K_Q^{1/2}\setminus MB_Q\neq\varnothing$ (because $x\in K^{1/2}(x')\setminus 2MB_Q$ for some $x'\in Q$, and $K^{1/2}(x')$ is an open set). \lemref{lem:key lemma} yields $P\notin\Tree$, a contradiction.
	\end{proof}
	\begin{lemma}\label{lem:Lipschitz graph}
		There exists a Lipschitz graph $\Gamma_R$, with Lipschitz constant depending only on $\alpha$, such that
		\begin{equation*}
		\widetilde{G}_R\subset\Gamma_R.
		\end{equation*}
		Moreover, there exists a big constant $\Lambda=\Lambda(M,t)>1$ such that for every $Q\in\Tree(R)$ we have
		\begin{equation}\label{eq:LambdaBQ intersects Gamma}
		\Lambda B_Q\cap\Gamma_R\neq\varnothing.
		\end{equation}
	\end{lemma}
	\begin{proof}
		Recall that for each cube $Q\in\D$ we have a ``center'' denoted by $x_Q\in Q$. Set $F = \{x_Q\, :\, Q\in\Sep^*(R)\}\cup \widetilde{G}_R$. It follows by \lemref{lem:xy from GR outside Kx} and \lemref{lem:Q in Sep* and KP dont intersect} that for any $x,y\in F$ we have $x-y\notin K^{1/2}$. Thus, there exists a Lipschitz graph $\Gamma_R$, with slope depending only on $\alpha$, such that $F\subset \Gamma_R$.
		
		Concerning the second statement, it is clearly true for $Q\in\Sep^*(R)$ (even with $\Lambda=1$). For $Q\in\Sep(R)$, we have by \lemref{lem:Sep contains Sep* cubes} that either $2MB_Q\cap\widetilde{G}_R\neq\varnothing$ or there exists $P\in\Sep^*(R)$ with $2MB_P\subset 2MB_Q$. Thus, \eqref{eq:LambdaBQ intersects Gamma} holds if $\Lambda\ge 2M$.
		
		If $Q\in\Stop(R)$, there exists some $P\in\Sep(R)$ which is a $t$-neighbour of $Q$, so that for some $\Lambda=\Lambda(t,M)>1$ we have $\Lambda B_Q\supset 2MB_P,$ and $2MB_P$ intersects $\Gamma_R$. Finally, for a general $Q\in\Tree(R)$, either $Q$ contains some cube from $\Stop(R)$, or $Q\subset\widetilde{G}_R$. In any case, $\Lambda B_Q\cap\Gamma_R\neq\varnothing$.
	\end{proof}
	
	\begin{remark}
		Note that while for a general cube $Q\in\Tree(R)$ we only have $\Lambda B_Q\cap\Gamma_R\neq\varnothing$, we have a better estimate for the root $R$:
		\begin{equation}\label{eq:2BR intersects the graph}
		B_R\cap\Gamma_R\neq\varnothing.
		\end{equation}
		Indeed, \eqref{eq:2BR intersects the graph} is clear if the set $\widetilde{G}_R$ is non-empty. If $\widetilde{G}_R=\varnothing$, then $\Sep^*(R)\neq\varnothing$, so that for some $P\in\Sep^*(R)$ we have $x_P\in \Gamma_R\cap B_R$.
	\end{remark}

	\section{Small measure of cubes from \texorpdfstring{$\LD(R)$}{LD(R)}}\label{sec:LD estimate}
	
	In the proof of the packing estimate \eqref{eq:packing estimate} it will be crucial to have a bound on the measure of low density cubes.
	\begin{lemma}\label{lem:LD small measure}
		We have
		\begin{equation*}
		\sum_{Q\in\LD(R)}\mu(Q)\lesssim_{t,\alpha} \tau \mu(R).
		\end{equation*}
	\end{lemma}
	In particular, for $\tau$ small enough we have
	\begin{equation}\label{eq:small measure LD}
		\sum_{Q\in\LD(R)}\mu(Q)\le \tau^{1/2}\mu(R).
	\end{equation}
	We begin by defining some auxiliary subfamilies of $\LD(R)$. 
	\begin{lemma}\label{lem:LDSep enough}
		There exists a $t$-separated family $\LD_{\Sep}(R)\subset \LD(R)$ such that
		\begin{equation*}
		\sum_{Q\in\LD(R)}\mu(Q)\lesssim_t \sum_{Q\in\LD_{\Sep}(R)}\mu(Q).
		\end{equation*}
	\end{lemma}
	\begin{proof}
		We construct the family $\LD_{\Sep}(R)$ in the following way. Define $\LD_1(R)$ as a maximal $t$-separated subfamily of $\LD(R)$. Next, define $\LD_2(R)$ as a maximal $t$-separated subfamily of $\LD(R)\setminus\LD_1(R)$. In general, having defined $\LD_j(R)$, we define $\LD_{j+1}(R)$ to be a maximal $t$-separated subfamily of $\LD(R)\setminus(\LD_1(R)\cup\dots\cup \LD_j(R))$.
		
		We claim that there is only a bounded number of non-empty families $\LD_j(R)$, with the bound depending on $t$. Indeed, if $Q\in\LD_j(R)$, then $Q$ has at least one $t$-neighbour in each family $\LD_k(R),\ k\le j$. It follows easily from the definition of $t$-neighbours that the number of $t$-neighbours of any given cube is bounded by a constant $C(t)$. Hence, $j\le C(t)$.
		
		Set $\LD_{\Sep}(R)$ to be the family $\LD_j(R)$ maximizing $\sum_{Q\in\LD_j(R)}\mu(Q)$. Then,
		\begin{equation*}
		\sum_{Q\in\LD(R)}\mu(Q)\le C(t) \sum_{Q\in\LD_{\Sep}(R)}\mu(Q).
		\end{equation*}
	\end{proof}
	We define also a family $\LD_{\Sep}^*(R)\subset\LD_{\Sep}(R)$ in the following way: we remove from $\LD_{\Sep}(R)$ all the cubes $P$ for which there exists some $Q\in\LD_{\Sep}(R)$ such that
	\begin{equation}\label{eq:def of LDSep*}
		1.1B_Q\cap 1.1B_P\neq\varnothing\qquad \text{and}\qquad r(Q)<r(P).
	\end{equation}
	
	\begin{lemma}\label{lem:LDSep alternative}
		For each $Q\in\LD_{\Sep}(R)$ at least one of the following is true:
		\begin{itemize}
			\item[(a)] $1.2B_Q\cap\widetilde{G}_R\neq\varnothing$
			\item[(b)] There exists some $P\in\LD_{\Sep}^*(R)$ such that $1.2B_P\subset 1.2B_Q$.
		\end{itemize}
	\end{lemma}
	\begin{proof}
		Suppose $Q\in\LD_{\Sep}(R)$, and that (a) does not hold. We will find $P$ such that (b) is satisfied. 
		
		If $Q\notin\LD^*_{\Sep}(R)$, then there exists some cube $Q_1\in\LD_{\Sep}(R)$ such that
		\begin{equation}
		1.1B_Q\cap 1.1B_{Q_1}\neq\varnothing\qquad \text{and}\qquad r(Q_1)<r(Q).
		\end{equation}
		Since $Q$ and $Q_1$ are $t$-separated, and \eqref{eq:neighbour condition 2} holds, it follows that $t\,r(Q_1)<r(Q)$. Thus, $Q_1$ is tiny compared to $Q$ and we have $1.2B_{Q_1}\subset 1.2B_Q$. If $Q_1\in \LD^*_{\Sep}(R)$, we set $P=Q_1$ and we are done. Otherwise, we iterate as in \lemref{lem:Sep contains Sep* cubes}  (with $2M$ replaced by $1.2$) to find a finite sequence $Q_1,\ Q_2, \dots,\ Q_m$ satisfying $1.2B_{Q_{j+1}}\subset 1.2B_{Q_j}$, and such that $Q_m\in\LD_{\Sep}^*(R)$.
	\end{proof}
	\begin{lemma}\label{lem:some LD estimate}
		For each $Q\in\LD_{\Sep}^*(R)$ we have
		\begin{equation}\label{eq:cones are a small subset}
		\mu\Big(Q\cap\bigcup_{P\in\LD_{\Sep}^*(R)}(K_P^{1/2}\setminus MB_P)\Big)\lesssim_{\tau,\alpha,\eta} \varepsilon \mu(Q).
		\end{equation}
		In particular, if $\varepsilon$ is small enough, then for each $Q\in\LD_{\Sep}^*(R)$ we can choose a point
		\begin{equation}\label{eq:def of wQ}
		w_Q\in Q\setminus\bigcup_{P\in\LD_{\Sep}^*(R)}(K_P^{1/2}\setminus MB_P).
		\end{equation}
	\end{lemma}
	\begin{proof}
		Suppose $Q\in\LD_{\Sep}^*(R)$ and that we have $Q\cap K_P^{1/2}\setminus MB_P\neq\varnothing$ for some $P\in\LD_{\Sep}^*(R)$. Note that if we had $Mr(Q)\le \dist(Q,P)$, then the assumptions of \lemref{lem:key lemma} would be satisfied, and we would arrive at $Q\not\in\Tree(R)$, a contradiction. Thus,
		\begin{equation}\label{eq:Q and P close}
		\dist(Q,P)\le Mr(Q)< t\,r(Q).
		\end{equation}
		It follows that \eqref{eq:neighbour condition 2}  -- one of the $t$-neigbourhood conditions -- is satisfied. Since $Q$ and $P$ are $t$-separated, we necessarily have $t\,r(Q)\le r(P)$ or $t\,r(P)\le r(Q)$.
		
		If we had $t\,r(Q)\le r(P)$, then \eqref{eq:Q and P close} implies $\dist(Q,P)\le r(P)$. Hence, $1.1B_Q\cap 1.1B_P\neq\varnothing$. But this cannot be true, by the definition of $\LD_{\Sep}^*(R)$. It follows that
		\begin{equation}\label{eq:t fixed}
		t\,r(P)\le r(Q).
		\end{equation}
		
		Let $S\supset P$ be the biggest ancestor of $P$ satisfying $r(S)\le\delta r(Q)$ for some small constant $\delta=\delta(\alpha)$ which will be fixed in a few lines. If $t$ is big enough, then $S\neq P$. Thus, $r(S)\approx_{\delta} r(Q)$, and $S\in\Tree(R)\setminus\Stop(R)$. Recall that by the definition of $\LD_{\Sep}^*(R)$ we have $1.1B_Q\cap 1.1 B_P = \varnothing$. It follows that if $\delta<0.001$, then $4B_S\cap 1.05 B_Q=\varnothing$. Now, using this separation, it is not difficult to check that for $\delta=\delta(\alpha)$ small enough, for any $x\in K_P^{1/2}\cap Q$ we have
		\begin{equation*}
		2B_S\subset K(x).
		\end{equation*}		
		Observe also that, due to \eqref{eq:Q and P close} and the fact that $r(S)\le\delta r(Q)$, we have 
		\begin{equation*}
		2B_S\subset B(x,r)\qquad\text{for}\ r\in\left( \frac{\eta^{-1}}{2}r(Q),\ \eta^{-1} r(Q)\right),
		\end{equation*}
		provided that $\eta$ is small enough (say, $\eta^{-1}\gg t$).
		Putting together the two estimates above, we get that
		\begin{equation*}
		\mu(2B_S)\le\mu(K(x,r))
		\end{equation*}
		for any $x\in K_P^{1/2}\cap Q\supset Q\cap K_P^{1/2}\setminus MB_P$ and all $r\in( {\eta^{-1}}r(Q)/2,\ \eta^{-1} r(Q))$. 
		
		Integrating the above over all $x\in A$, where $A\subset Q\cap K_P^{1/2}\setminus MB_P$ is an arbitrary measurable subset, yields
		\begin{multline}\label{eq:weird estimate}
		\mu(A)\Theta_{\mu}(2B_R)^p\overset{\eqref{eq:lower density estimate on Tree}}{\le} \tau^{-1} \mu(A)\Theta_{\mu}(2B_S)^p\approx_{\tau,\alpha} \mu(A)\left(\frac{\mu(2B_S)}{r(Q)^{n}}\right)^p\\
		\lesssim_{\eta}\int_A\int_{\eta r(Q)}^{\eta^{-1}r(Q)}\left(\frac{\mu(K(x,r))}{r^n}\right)^p\frac{dr}{r}d\mu(x).
		\end{multline}
		
		Now, let $P_i$ be some ordering of cubes $P\in\LD_{\Sep}^*(R)$ satisfying $Q\cap K_P^{1/2}\setminus MB_P\neq\varnothing$. We define $A_1 = Q\cap K_{P_1}^{1/2}\setminus MB_{P_1}$, and for $i>1$ 
		\begin{equation*}
		A_i = Q\cap K_{P_i}^{1/2}\setminus \Big(MB_{P_i}\cup\bigcup_{j=1}^{i-1}A_j\Big).
		\end{equation*}
		Observe that $A_i$ are pairwise disjoint and their union is $Q\cap\bigcup_{P\in\LD_{\Sep}^*(R)}(K_P^{1/2}\setminus MB_P)$. Thus,
		\begin{multline*}
		\mu\Big(Q\cap\bigcup_{P\in\LD_{\Sep}^*(R)}K_P^{1/2}\setminus MB_P\Big)\Theta_{\mu}(2B_R)^p=\sum_{i}\mu(A_i)\Theta_{\mu}(2B_R)^p\\
		\overset{\eqref{eq:weird estimate}}{\lesssim}_{\tau,\alpha,\eta} \int_{\bigcup_i A_i}  \int_{\eta r(Q)}^{\eta^{-1}r(Q)}\left(\frac{\mu(K({x,r}))}{r^n}\right)^p\frac{dr}{r}d\mu(x)
		\le \E(Q)\mu(Q).
		\end{multline*}
		Note that since $Q\notin\BCE(R)$, we have $ \E(Q)\mu(Q)\le \varepsilon\Theta_{\mu}(2B_R)^p\mu(Q)$. So the estimate \eqref{eq:cones are a small subset} holds.
	\end{proof}
	
	\begin{lemma}\label{lem:GammaLD}
		There exists an $n$-dimensional Lipschitz graph $\Gamma_{\LD}$ passing through all the points $w_P,\ P\in\LD_{\Sep}^*(R)$. The Lipschitz constant of $\Gamma_{\LD}$ depends only on $\alpha$.
	\end{lemma}
	\begin{proof}
		It suffices to show that for any $Q, P\in\LD_{\Sep}^*(R),\ Q\neq P,$ we have 
		\begin{equation}\label{eq:wQ and wP in the cone}
		w_Q-w_P\notin K^{1/2}.
		\end{equation}
		Without loss of generality assume $r(P)\le r(Q)$. By \eqref{eq:def of wQ} we have
		\begin{equation*}
		w_Q\notin K_P^{1/2}\setminus MB_P.
		\end{equation*}
		In particular,
		\begin{equation*}
		w_Q\notin K^{1/2}(w_P)\setminus MB_P.
		\end{equation*}
		So, to prove \eqref{eq:wQ and wP in the cone}, it is enough to show that
		\begin{equation}\label{eq:wQ notin MBP}
		w_Q\notin MB_P.
		\end{equation}
		
		Assume the contrary, i.e. $w_Q\in MB_P$. Then,
		\begin{equation*}
		\dist(Q,P)\le CMr(P)\le t(r(Q)+r(P)).
		\end{equation*}
		That is, \eqref{eq:neighbour condition 2} holds. But $Q$ and $P$ are $t$-separated, and so \eqref{eq:neighbour condition 1} must fail. Hence,
		\begin{equation*}
		r(P)\le t^{-1}r(Q).
		\end{equation*}
		$Q$ and $P$ belong to $\LD_{\Sep}^*(R)$, so by \eqref{eq:def of LDSep*} we have $1.1B_Q\cap 1.1B_P=\varnothing$. Thus,
		\begin{equation*}
		\dist(w_Q, B_P)\ge 0.1r(B_Q)\ge Ct\, r(B_P)> Mr(B_P).
		\end{equation*}
		So \eqref{eq:wQ notin MBP} holds.
	\end{proof}
	We can finally finish the proof of \lemref{lem:LD small measure}.
	\begin{proof}[Proof of \lemref{lem:LD small measure}]
		By \lemref{lem:LDSep enough} it suffices to estimate the measure of cubes from $\LD_{\Sep}(R)$. Let $\mathcal{G}$ denote an arbitrary finite subfamily of $\LD_{\Sep}(R)$. We use the covering lemma \cite[Theorem 9.31]{tolsa2014analytic} to choose a subfamily $\mathcal{F}\subset\mathcal{G}$ such that
		\begin{equation*}
		\bigcup_{Q\in\mathcal{G}}1.5B_Q\subset \bigcup_{Q\in\mathcal{F}}2B_Q,
		\end{equation*}
		and the balls $\{1.5B_Q\}_{Q\in\mathcal{F}}$ are of bounded superposition. 
				
		The above and the $\LD$ stopping condition give
		\begin{equation}\label{eq:LDSep reduced to F}
		\sum_{Q\in\mathcal{G}}\mu(Q)\le\sum_{Q\in\mathcal{F}}\mu(2B_Q)\lesssim \tau\,\Theta_{\mu}(2B_R)\sum_{Q\in\mathcal{F}}r(B_Q)^n.
		\end{equation}
		Now, it follows from \lemref{lem:LDSep alternative} and \lemref{lem:GammaLD} that for each $Q\in\mathcal{G}\subset\LD_{\Sep}(R)$ there exists either $w_Q\in\Gamma_{\LD}\cap 1.2B_Q$ or $x\in \widetilde{G}_R\cap 1.2B_Q\subset \Gamma_R\cap 1.2B_Q$. Hence,
		\begin{equation*}
		\mathcal{H}^n(1.5B_Q\cap(\Gamma_{\LD}\cup\Gamma_R))\approx_{\alpha} r(B_Q)^n.
		\end{equation*}
		Now, using the bounded superposition property of $\mathcal{F}$ we get
		\begin{multline*}
		\sum_{Q\in\mathcal{F}}r(B_Q)^n\approx_{\alpha} \sum_{Q\in\mathcal{F}}\mathcal{H}^n(1.5B_Q\cap(\Gamma_{\LD}\cup\Gamma_R))\lesssim \mathcal{H}^n\big(\bigcup_{Q \in\mathcal{F}}1.5B_Q\cap(\Gamma_{\LD}\cup\Gamma_R)\big)\\
		\le \mathcal{H}^n(2B_R\cap (\Gamma_{\LD}\cup\Gamma_R))\approx_{\alpha} r(R)^n\approx \mu(2B_R)\Theta_{\mu}(2B_R)^{-1}\overset{R\in\D^{db}}{\approx} \mu(R)\Theta_{\mu}(2B_R)^{-1}.
		\end{multline*}
		Together with \eqref{eq:LDSep reduced to F}, this gives
		\begin{equation*}
		\sum_{Q\in\mathcal{G}}\mu(Q)\lesssim_{\alpha}\tau \mu(R).
		\end{equation*}
		Since $\mathcal{G}$ was an arbitrary finite subfamily of $\LD_{\Sep}(R)$, we finally arrive at
		\begin{equation*}
		\sum_{Q\in\LD_{\Sep}(R)}\mu(Q)\lesssim_{\alpha}\tau \mu(R).
		\end{equation*}
	\end{proof}
	
	\section{Top cubes and packing estimate}\label{sec:Top}
	\subsection{Definition of \texorpdfstring{$\Top$}{Top}} In order to define the $\Top$ family, we need to introduce some additional notation. Given $Q\in\D$, let $\MD(Q)$ denote the family of maximal cubes from $\D^{db}(Q)\setminus \{Q\}$. It follows from \lemref{lem:DM lattice} (c) that the cubes from $\MD(Q)$ cover $\mu$-almost all of $Q$. 
	
	Given $R\in\D^{db}$ set
	\begin{equation*}
	\Next(R) = \bigcup_{Q\in\Stop(R)}\MD(Q).
	\end{equation*}
	Since we always have $\MD(Q)\neq\{Q\}$, it is clear that $\Next(R)\neq\{R\}$.
	
	Observe that if $P\in\Next(R)$, then by \lemref{lem:density estimates on Tree} and \lemref{lem:density dropping for nondoubling} we have for all intermediate cubes $S\in\D,\ P\subset S\subset R,$
	\begin{equation}\label{eq:control on density for next cubes}
	\Theta_{\mu}(2B_S)\lesssim_A\Theta_{\mu}(2B_R).
	\end{equation}
	
	We are finally ready to define $\Top$. It is defined inductively as $\Top = \bigcup_{k\ge0} \Top_k$. First, set
	\begin{equation*}
	\Top_0 = \{R_0\},
	\end{equation*}
	where $R_0$ was defined as $\supp{\mu}$. Having defined $\Top_k$, we set
	\begin{equation*}
	\Top_{k+1} = \bigcup_{R\in\Top_k}\Next(R).
	\end{equation*}
	Note that for each $k\ge 0$ the cubes from $\Top_k$ are pairwise disjoint.
	\subsection{Definition of \texorpdfstring{$\ID$}{ID}}
	We distinguish a special type of $\Top$ cubes. We say that $R\in \Top$ is increasing density, $R\in\ID$, if
	\begin{equation*}
	\mu\bigg(\bigcup_{Q\in\HD(R)} Q \bigg)\ge \frac{1}{2}\mu(R).
	\end{equation*}
	\begin{lemma}\label{lem:ID density}
		If $A$ is big enough, then for all $R\in\ID$
		\begin{equation}\label{eq:density estimate for ID}
		\Theta_{\mu}(2B_R)^p\mu(R)\le\frac{1}{2}\sum_{Q\in\Next(R)}\Theta_{\mu}(2B_Q)^p\mu(Q).
		\end{equation}
	\end{lemma}
	\begin{proof}
		The definition of $\ID$ and the $\HD$ stopping condition imply that for any $R\in\ID$
		\begin{equation*}
		\Theta_{\mu}(2B_R)^p\mu(R)\le 2\,\Theta_{\mu}(2B_R)^p\sum_{Q\in\HD(R)}\mu(Q)\le 2A^{-p}\sum_{Q\in\HD(R)}\Theta_{\mu}(2B_Q)^p\mu(Q).
		\end{equation*}
		Note that all $Q\in\HD(R)$ are doubling, and so by \lemref{lem:doubling subcube of doubling}
		\begin{equation*}
		\Theta_{\mu}(2B_Q)^p\mu(Q)\lesssim\sum_{P\in\MD(Q)}\Theta_{\mu}(2B_P)^p\mu(P)= \sum_{\substack{P\in\Next(R)\\P\subset Q}}\Theta_{\mu}(2B_P)^p\mu(P).
		\end{equation*}
		If $A$ is taken big enough, then the estimates above yield \eqref{eq:density estimate for ID}.
	\end{proof}
	\subsection{Packing condition}
	We will now establish the packing condition \eqref{eq:packing estimate}. For $S\in\Top$ set $\Top(S) = \Top\cap\D(S)$ and $\Top_j(S) = \Top_j\cap\D(S)$. For $k\ge 0$ we also define
	\begin{align*}
	\Top_0^k(S) &= \bigcup_{0\le j\le k}\Top_j(S),\\
	\ID_0^k(S) &= \ID\cap\Top_0^k(S).
	\end{align*}
	Recall that $\mu$ satisfies the following polynomial growth condition: there exist $C_1>0$ and $r_0>0$ such that for all $x\in\supp\mu,\ 0<r\le r_0,$ we have
	\begin{equation}\label{eq:weaker growth condition}
	\mu(B(x,r))\le C_1 r^n.
	\end{equation}	
	\begin{lemma}
		For all $S\in \Top$ we have
		\begin{equation}\label{eq:packing condition}
		\sum_{R\in\Top(S)}\Theta_{\mu}(2B_R)^p\mu(R) \lesssim_{\varepsilon,\eta,\tau}(C_1)^p\mu(S) + \int_{2B_S} \int_0^{\eta^{-1}C_0r(S)}\bigg(\frac{\mu(K(x,r))}{r^n} \bigg)^p\ \frac{dr}{r} d\mu(x).
		\end{equation}
		The implicit constant does not depend on $r_0$.
	\end{lemma}
	\begin{proof}
		First, we deal with $\ID$ cubes. Note that
		\begin{multline*}
		\sum_{R\in\ID_0^k(S)}\Theta_{\mu}(2B_R)^p\mu(R)\overset{\eqref{eq:density estimate for ID}}{\le}\frac{1}{2}\sum_{R\in\ID_0^k(S)}\sum_{Q\in\Next(R)}\Theta_{\mu}(2B_Q)^p\mu(Q)\\
		\le \frac{1}{2}\sum_{Q\in\Top_0^{k+1}(S)}\Theta_{\mu}(2B_Q)^p\mu(Q),
		\end{multline*}
		where the last inequality follows from the fact that $\bigcup_{R\in\Top_0^k}\Next(R) = \Top_0^{k+1}$. 		
		Now, observe that for $Q\in\Top_{k+1}$ we have $r(Q)\le C_0A_0^{-k}r(R_0)$, and so if $k$ is big enough, then $r(2B_Q)\le r_0$. Thus, by \eqref{eq:weaker growth condition} 
		\begin{equation}\label{eq:small cubes density estimate}
		\Theta_{\mu}(2B_Q)\le C_1.
		\end{equation}
		Hence,
		\begin{multline}\label{eq:reducing to case without ID}
		\sum_{R\in\Top_0^k(S)}\Theta_{\mu}(2B_R)^p\mu(R) = \sum_{R\in\Top_0^k(S)\setminus\ID}\Theta_{\mu}(2B_R)^p\mu(R) + \sum_{R\in\ID_0^k(S)}\Theta_{\mu}(2B_R)^p\mu(R)\\
		\le \sum_{R\in\Top_0^k(S)\setminus\ID}\Theta_{\mu}(2B_R)^p\mu(R) + \frac{1}{2}\sum_{R\in\Top_0^{k+1}(S)}\Theta_{\mu}(2B_R)^p\mu(R)\\
		\le \sum_{R\in\Top_0^k(S)\setminus\ID}\Theta_{\mu}(2B_R)^p\mu(R) + \frac{1}{2}\sum_{R\in\Top_0^{k}(S)}\Theta_{\mu}(2B_R)^p\mu(R) +\frac{(C_1)^p}{2}\mu(S).
		\end{multline}
		Note that for small cubes $Q\in\Top_0^k(S)$ (i.e. satisfying $r(2B_Q)\le r_0$) we have \eqref{eq:small cubes density estimate}, while for big cubes the trivial estimate $\Theta_{\mu}(2B_Q)\le\mu(2B_S) r_0^{-n}$ holds. It follows that
		\begin{equation*}
		\sum_{R\in\Top_0^k(S)}\Theta_{\mu}(2B_R)^p\mu(R) \le (k+1)\left( (C_1)^p + \mu(2B_S)^p r_0^{-np}\right)\mu(S)<\infty,
		\end{equation*}
		and so we may deduce from \eqref{eq:reducing to case without ID} that
		\begin{equation*}
		\sum_{R\in\Top_0^k(S)}\Theta_{\mu}(2B_R)^p\mu(R) \le 2  \sum_{R\in\Top_0^k(S)\setminus\ID}\Theta_{\mu}(2B_R)^p\mu(R) + {(C_1)^p}\mu(S).
		\end{equation*}
		Letting $k\to\infty$ we arrive at
		\begin{equation}\label{eq:reduction to Top minus ID}
		\sum_{R\in\Top(S)}\Theta_{\mu}(2B_R)^p\mu(R) \le 2  \sum_{R\in\Top(S)\setminus\ID}\Theta_{\mu}(2B_R)^p\mu(R) + {(C_1)^p}\mu(S).
		\end{equation}
		Now, we need to estimate the sum from the right hand side. By the definition of $\ID$ we have for all $R\in\Top(S)\setminus \ID$
		\begin{equation*}
		\mu\bigg(R\setminus\bigcup_{Q\in\HD(R)}Q\bigg)\ge\frac{1}{2}\,\mu(R),
		\end{equation*}
		and so by \lemref{lem:DM lattice} (c) we get
		\begin{align*}
		\mu(R)&\le 2\,\mu\bigg(R\setminus\bigcup_{Q\in\Stop(R)}Q\bigg) + 2\,\mu\bigg(\bigcup_{Q\in\Stop(R)\setminus\HD(R)}Q\bigg)\\
		&= 2\,\mu\bigg(R\setminus\bigcup_{Q\in\Next(R)}Q\bigg) + 2\sum_{Q\in\LD(R)}\mu(Q) + 2\sum_{Q\in\BCE(R)}\mu(Q).
		\end{align*}
		The measure of low density cubes is small due to \eqref{eq:small measure LD}, and so for $\tau$ small enough we have
		\begin{equation*}
		\mu(R) \le  3\,\mu\bigg(R\setminus\bigcup_{Q\in\Next(R)}Q\bigg) + 3\sum_{Q\in\BCE(R)}\mu(Q).
		\end{equation*}
		Thus,
		\begin{multline}\label{eq:estimating Top minus ID}
		\sum_{R\in\Top(S)\setminus\ID}\Theta_{\mu}(2B_R)^p\mu(R) \le 3\sum_{R\in\Top(S)}\Theta_{\mu}(2B_R)^p\mu\bigg(R\setminus\bigcup_{Q\in\Next(R)}Q\bigg)\\
		 + 3 \sum_{R\in\Top(S)\setminus\ID}\Theta_{\mu}(2B_R)^p\sum_{Q\in\BCE(R)}\mu(Q).
		\end{multline}
		
		Concerning the first sum, notice that if $\mu\big(R\setminus\bigcup_{Q\in\Next(R)}Q\big)>0$, then we have arbitrarily small cubes $P$ belonging to $\Tree(R)$. In particular, by \eqref{eq:lower density estimate on Tree} and \eqref{eq:weaker growth condition}, we have $\Theta_{\mu}(2B_R)\le\tau^{-1}\Theta_{\mu}(2B_P)\le \tau^{-1}C_1$, taking $P\in\Tree(R)\setminus\Stop(R)$ small enough. Recall also that for $R\in\Top(S),$ the sets $R\setminus\bigcup_{Q\in\Next(R)}Q$ are pairwise disjoint. Hence,
		\begin{equation}\label{eq:estimate of good part}
		\sum_{R\in\Top(S)}\Theta_{\mu}(2B_R)^p\mu\bigg(R\setminus\bigcup_{Q\in\Next(R)}Q\bigg)\le (\tau^{-1}C_1)^p\mu(S).
		\end{equation}
		
		To estimate the second sum from \eqref{eq:estimating Top minus ID}, we apply \eqref{eq:BCE estimate} to get
		\begin{multline*}
		\sum_{R\in\Top(S)}\Theta_{\mu}(2B_R)^p\sum_{Q\in\BCE(R)}\mu(Q)\le \frac{1}{\varepsilon}\sum_{R\in\Top(S)}\sum_{P\in\Tree(R)}\E(P)\mu(P)\\
		\le \frac{1}{\varepsilon}\sum_{P\in\D(S)}\E(P)\mu(P)
		\end{multline*}
		By the definition of $\E(P)$, and the bounded intersection property of the balls $2B_P$ for cubes $P$ of the same generation, we have
		\begin{align*}
		\sum_{P\in\D(S)}\E(P)\mu(P) &= \sum_k \sum_{P\in\mathcal{D}_{\mu,k}(S)} \int_{2B_P}\int_{\eta r(P)}^{\eta^{-1}r(P)}\bigg(\frac{\mu(K(x,r))}{r^n}\bigg)^p\ \frac{dr}{r}\\
		&\lesssim \sum_k \int_{2B_S}\int_{\eta A_0^{-k}}^{\eta^{-1}C_0A_0^{-k}}\bigg(\frac{\mu(K(x,r))}{r^n}\bigg)^p\ \frac{dr}{r}\\
		&\lesssim_{\eta} \int_{2B_S} \int_0^{\eta^{-1}C_0r(S)}\bigg(\frac{\mu(K(x,r))}{r^n} \bigg)^p\ \frac{dr}{r} d\mu(x).
		\end{align*}
		Consequently, 
		\begin{equation*}
		\sum_{R\in\Top(S)}\Theta_{\mu}(2B_R)^p\sum_{Q\in\BCE(R)}\mu(Q)\lesssim_{\varepsilon,\eta} \int_{2B_S} \int_0^{\eta^{-1}C_0r(S)}\bigg(\frac{\mu(K(x,r))}{r^n} \bigg)^p\ \frac{dr}{r} d\mu(x).
		\end{equation*}
		Together with \eqref{eq:reduction to Top minus ID}, \eqref{eq:estimating Top minus ID}, and \eqref{eq:estimate of good part}, this gives \eqref{eq:packing condition}.
	\end{proof}

	 Let us put together all the ingredients of the proof of the main lemma.
	\begin{proof}[Proof of \lemref{lem:main lemma}]
		Let $\Top\subset\D^{db}$ be as above, and $\{\Gamma_R\}_{R\in\Top}$ be as in \lemref{lem:Lipschitz graph}. Then, properties (i) and (ii) are ensured by \lemref{lem:Lipschitz graph}. Property (iii) follows from \eqref{eq:control on density for next cubes}. We get the packing estimate \eqref{eq:packing estimate} from \eqref{eq:packing condition} by taking $S=R_0$.
	\end{proof}

	\section{Application to singular integral operators}\label{sec:SIOs}
	To prove \thmref{thm:SIO theorem}, we will use geometric characterizations of boundedness of operators from $\mathcal{K}^n(\R^d)$ shown in \cite[Sections 4, 5, 9]{girela-sarrion2018}. For $n=1,\ d=2$, a variant of this characterization valid for the Cauchy transform was already proved in \cite{tolsa2005bilipschitz}.
	
	For $Q, S\in\D$, $Q\subset S$, we set
	\begin{equation*}
	\delta_{\mu}(Q,S) = \int_{2B_S\setminus 2B_Q}\frac{1}{|y-x_Q|^n}\ d\mu(y).
	\end{equation*}
	The notation $\Good(R),\ \Tr(R),\ \Next(R)$ used below was introduced in Section \ref{sec:main lemma}.
	\begin{lemma}[{\cite{girela-sarrion2018}}]\label{lem:estimate of Tolsa and Girela}
		Let $\mu$ be a compactly supported Radon measure on $\R^d$ satisfying the growth condition \eqref{eq:growth condition}.
		Assume there exists a family of cubes $\Top\subset\D^{db}$, and a corresponding family of Lipschitz graphs $\{\Gamma_R\}_{R\in\Top}$, satisfying:
		\begin{itemize}
			\item[(i)] Lipschitz constants of $\Gamma_R$ are uniformly bounded by some absolute constant,
			\item[(ii)] $\mu$-almost all $\Good(R)$ is contained in $\Gamma_R$,
			\item[(iii)] for all $Q\in\Tr(R)$ we have $\Theta_{\mu}(2B_Q)\lesssim\Theta_{\mu}(2B_R)$.
			\item[(iv)] for all $Q\in\Next(R)$ there exists $S\in\D,\ Q\subset S,$ such that $\delta_{\mu}(Q,S)\lesssim \Theta_{\mu}(2B_R),$ and $2B_S\cap\Gamma_R\ne\varnothing$.			
		\end{itemize}
		Then, for every singular integral operator $T$ with kernel $k\in\mathcal{K}^n(\R^d)$ we have
		\begin{equation*}
		\sup_{\varepsilon>0}\ \lVert{T}_{\varepsilon}\mu\rVert_{L^2(\mu)}^2\lesssim	\sum_{R\in\Top}\Theta_{\mu}(2B_R)^2\mu(R),
		\end{equation*}
		with the implicit constant depending on $C_1$ and the constant $C_k$ from \eqref{eq:calderon zygmund constant}.
	\end{lemma}
	The result above is not explicitly stated in \cite{girela-sarrion2018}, but it is essentially \cite[Section 5, Lemma 1]{girela-sarrion2018}. The ``corona decomposition'' assumptions of \lemref{lem:estimate of Tolsa and Girela} come from \cite[Lemma D]{girela-sarrion2018}, which is treated there as a black-box. The proof of \cite[Lemma 1]{girela-sarrion2018} is concluded in \cite[Section 9]{girela-sarrion2018}, and it is evident from its last line that we may replace the $\beta$-number right hand side of \cite[Lemma 1]{girela-sarrion2018} by the sum-over-$\Top$-cubes right hand side of \lemref{lem:estimate of Tolsa and Girela}.

	We are going to use \lemref{lem:main lemma} together with \lemref{lem:Lipschitz graph} and \lemref{lem:estimate of Tolsa and Girela} to get the following.
	\begin{lemma}\label{lem:checking T1 assumptions}
		Let $\mu$ be a compactly supported Radon measure on $\R^d$ satisfying the growth condition \eqref{eq:growth condition}.
		Assume further that for some $V\in G(d,d-n),$ $\alpha\in (0,1),$ we have $\mathcal{E}_{\mu,2}(\R^d, V, \alpha)<\infty.$
		
		Then, for every singular integral operator $T$ with kernel $k\in\mathcal{K}^n(\R^d)$ we have
		\begin{equation}\label{eq:SIO estimate}
		\sup_{\varepsilon>0}\ \lVert{T}_{\varepsilon}\mu\rVert_{L^2(\mu)}^2\lesssim	\mu(\R^d) + \mathcal{E}_{\mu,2}(\R^d, V,\alpha),
		\end{equation}
		with the implicit constant depending on $C_1,\alpha$ and the constant $C_k$ from \eqref{eq:calderon zygmund constant}.
	\end{lemma}
	\begin{proof}		
		 Using \lemref{lem:main lemma} (with $p=2$), it is clear that the assumptions (i)-(iii) of \lemref{lem:estimate of Tolsa and Girela} are satisfied. We still have to check if (iv) holds. Once we do that, the packing estimate \eqref{eq:packing estimate} together with \lemref{lem:estimate of Tolsa and Girela} will ensure that \eqref{eq:SIO estimate} holds.
		
		Suppose $R\in\Top,\ Q\in\Next(R)$. We are looking for $S\in\D$ such that $\delta_{\mu}(Q,S)\lesssim \Theta_{\mu}(2B_R),$ and $2B_{S}\cap\Gamma_R\ne\varnothing$. Let $P\in\Stop(R)$ be such that $Q\subset P$. By \lemref{lem:Lipschitz graph} we have some constant $\Lambda$ such that
		\begin{equation*}
		\Lambda B_P\cap\Gamma_R\neq\varnothing.
		\end{equation*}
		Together with \eqref{eq:2BR intersects the graph}, this implies that there exists $S\in\Tree(R)$ such that $P\subset S,\ r(S)\approx_{\Lambda} r(P)$, and
		\begin{equation*}
		2B_{S}\cap\Gamma_R\neq\varnothing.
		\end{equation*}
		We split
		\begin{equation*}
		\delta_{\mu}(Q,S) = \int_{2B_{S}\setminus 2B_P}\frac{1}{|y-x_Q|^n}\ d\mu(y) + \int_{2B_P\setminus 2B_Q}\frac{1}{|y-x_Q|^n}\ d\mu(y).
		\end{equation*}		
		Concerning the first integral, for $y\in 2B_{S}\setminus 2B_P$ we have $|y-x_Q|\approx r(S)\approx_{\Lambda} r(P)$, and so
		\begin{equation*}
		\int_{2B_{S}\setminus 2B_P}\frac{1}{|y-x_Q|^n}\ d\mu(y)\lesssim \Theta_{\mu}(2B_S)\overset{\eqref{eq:upper density estimate on Tree}}{\lesssim_A}\Theta_{\mu}(2B_R).
		\end{equation*}		
		To deal with the second integral, observe that there are no doubling cubes between $Q$ and $P$. Then, it follows from \lemref{lem:density dropping for nondoubling} that
		\begin{equation*}
		\int_{2B_P\setminus 2B_Q}\frac{1}{|y-x_Q|^n}\ d\mu(y)\lesssim\Theta_{\mu}(100B(P)).
		\end{equation*}
		If $P=R$, then $P$ is doubling and we have $\Theta_{\mu}(100B(P))\lesssim\Theta_{\mu}(2B_R)$. Otherwise, the parent of $P$, denoted by $P'$, belongs to $\Tree(R)\setminus\Stop(R)$. Since $100B(P)\subset 2B_{P'}$, we get
		\begin{equation*}
		\Theta_{\mu}(100B(P))\lesssim \Theta_{\mu}(2B_{P'})\overset{\eqref{eq:upper density estimate on Tree}}{\lesssim_A}\Theta_{\mu}(2B_R).
		\end{equation*}
		Either way, we get that $\delta_{\mu}(Q,S)\lesssim_A\Theta_{\mu}(2B_R)$, and so the assumption (iv) of \lemref{lem:estimate of Tolsa and Girela} is satisfied.
	\end{proof}
	\lemref{lem:checking T1 assumptions} allows us to use the non-homogeneous $T1$ theorem of Nazarov, Treil and Volberg \cite{nazarov1997} to prove a version of \thmref{thm:SIO theorem} in the case of a fixed direction $V$, i.e. if for all $x\in\supp\mu$ we have $V_x\equiv V$.
	\begin{lemma}\label{lem:SIOs bdd if single direction}
		Let $\mu$ be a Radon measure on $\R^d$ satisfying the polynomial growth condition \eqref{eq:growth condition}. Suppose that there exist $M_0>1,\ \alpha\in (0,1),\ V\in G(d,d-n),$ such that for every ball $B$ we have
		\begin{equation}\label{eq:Energy bdd in single direction}
		\mathcal{E}_{\mu,2}(B,V,\alpha)\le M_0 \mu(B).
		\end{equation}
		Then, all singular integral operators $T_{\mu}$ with kernels in $\mathcal{K}^n(\R^d)$ are bounded in $L^2(\mu)$. The bound on the operator norm of $T_{\mu}$ depends only on $C_1,\alpha, M_0,$ and the constant $C_k$ from \eqref{eq:calderon zygmund constant}.
	\end{lemma}
	
	\begin{proof}
		 We apply \lemref{lem:checking T1 assumptions} to $\restr{\mu}{B}$, where $B$ is an arbitrary ball, and get that
		\begin{equation*}
		\sup_{\varepsilon>0}\ \lVert{T}_{\varepsilon}(\restr{\mu}{B})\rVert_{L^2(\restr{\mu}{B})}^2\lesssim_{C_1,\alpha,C_k} \mu(B)+\mathcal{E}_{\restr{\mu}{B},2}(\R^d,V,\alpha).
		\end{equation*}
		It is easy to see that, using the assumptions \eqref{eq:growth condition} and \eqref{eq:Energy bdd in single direction}, we have 
		\begin{equation*}
		\mathcal{E}_{\restr{\mu}{B},2}(\R^d,V,\alpha)\lesssim \mathcal{E}_{\mu,2}(B,V,\alpha) + C_1^2\mu(B)\le (M_0+C_1^2) \mu(B).
		\end{equation*}
		Hence, 
		\begin{equation}\label{eq:T1 assumption}
		\sup_{\varepsilon>0}\ \lVert{T}_{\varepsilon}(\restr{\mu}{B})\rVert_{L^2(\restr{\mu}{B})}^2\lesssim_{C_1,\alpha,C_k,M_0} \mu(B).
		\end{equation}
		The $L^2$ boundedness of $T_{\mu}$ follows by the non-homogeneous $T1$ theorem from \cite{nazarov1997}. The condition \eqref{eq:T1 assumption} is slightly weaker than the original assumption in \cite{nazarov1997}, but this is not a problem, see the discussion in \cite[\S 3.7.2]{tolsa2014analytic}.
	\end{proof}
	We are ready to finish the proof of \thmref{thm:SIO theorem}.
	\begin{proof}[Proof of \thmref{thm:SIO theorem}]
		
		Let $B$ be an arbitrary ball intersecting $\supp\mu$. Recall that, by the definition of BPBE($2$), there exist $M_0>1,\ \kappa>0,\ V_B\in G(d,d-n),$ and $G_B\subset B$ such that $\mu(G_B)\ge\kappa\mu(B)$ and for all $x\in G_B$
		\begin{equation*}
		\int_0^{r(B)} \left(\frac{\mu(K(x,V_B,\alpha,r))}{r^n}\right)^2\ \frac{dr}{r} \le M_0.
		\end{equation*}
		By the polynomial growth condition \eqref{eq:growth condition} we also have
		\begin{equation*}
		\int_{r(B)}^{\infty}\left(\frac{\mu(K(x,V_B,\alpha,r))}{r^n}\right)^2\ \frac{dr}{r}\le \int_{r(B)}^{\infty}\frac{\mu(B)^2}{r^{2n+1}}\ dr\lesssim \frac{\mu(B)^2}{r(B)^{2n}}\le C_1^2.
		\end{equation*}
		Hence, for all $x\in G_B$
		\begin{equation*}
		\int_0^{\infty} \left(\frac{\mu(K(x,V_B,\alpha,r))}{r^n}\right)^2\ \frac{dr}{r} \lesssim_{C_1,M_0} 1.
		\end{equation*}
		Set $\nu = \restr{\mu}{G_B}$. The estimate above implies that for all balls $B'\subset\R^d$ we have
		\begin{equation*}
		\mathcal{E}_{\nu,2}(B',V_B,\alpha)=\int_{B'}\int_0^{r(B')} \left(\frac{\nu(K(x,V_B,\alpha,r))}{r^n}\right)^2\ \frac{dr}{r}d\nu(x)\lesssim_{C_1,M_0} \nu(B').
		\end{equation*}
		Clearly, $\nu= \restr{\mu}{G_B}$ has polynomial growth, and so we may apply \lemref{lem:SIOs bdd if single direction} to conclude that all singular integral operators $T_{ \nu}$ with kernels in $\mathcal{K}^n(\R^d)$ are bounded in $L^2(\nu)$. Thus, the corresponding maximal operators $T_*$ are bounded from $M(\R^d)$ to $L^{1,\infty}(\nu)$, see \cite[Theorem 2.21]{tolsa2014analytic}.
		
		Recall that for all balls $B$ we have $\mu(G_B)\approx_{\kappa}\mu(B)$. For any fixed $T$, the operator norm of $T_{\restr{\mu}{G_B},\varepsilon}:L^2(\restr{\mu}{G_B})\to L^2(\restr{\mu}{G_B})$ is bounded uniformly in $B$ and $\varepsilon$, and so the same is true for the operator norm of $T_*:M(\R^d)\to L^{1,\infty}(\restr{\mu}{G_B})$. Hence, we may use the good lambda method \cite[Theorem 2.22]{tolsa2014analytic} to conclude that $T_{\mu}$ is bounded in $L^2(\mu)$.
	\end{proof}
	
	\section{Sufficient condition for rectifiability}\label{sec:suff rectifiability}
	The aim of this section is to prove the following sufficient condition for rectifiability.
	\begin{prop}\label{prop:suff rectif}
		Suppose $\mu$ is a Radon measure on $\R^d$ satisfying $\Theta^{n,*}(\mu,x)>0$ and  $\Theta_*^{n}(\mu,x)<\infty$ for $\mu$-a.e. $x\in\R^d$. Assume further that for $\mu$-a.e. $x\in\R^d$ there exists some $V_x\in G(d,d-n)$ and $\alpha_x\in (0,1)$ such that
		\begin{equation}\label{eq:pointwise conical energy finite}
		\int_0^{1}\bigg(\frac{\mu(K(x,V_x,\alpha_x,r))}{r^n} \bigg)^p\ \frac{dr}{r}<\infty.
		\end{equation}
		Then, $\mu$ is $n$-rectifiable.
	\end{prop}
	We reduce the proposition above to the following lemma.
	\begin{lemma}\label{lem:reduction suff rect}
		Suppose $\mu$ is a Radon measure on $B(0,1)\subset\R^d$, and assume that there exists a constant $C_*>0$ such that $\Theta_*^{n}(\mu,x)\le C_*$ and $\Theta^{n,*}(\mu,x)> 0$ for $\mu$-a.e. $x\in\R^d$. Assume further that there exist $M_0>0$, $V\in G(d,d-n)$ and $\alpha\in (0,1)$ such that for $\mu$-a.e. $x\in\R^d$
		\begin{equation}\label{eq:pointwise conical energy bounded}
		\int_0^{1}\bigg(\frac{\mu(K(x,V,\alpha,r))}{r^n} \bigg)^p\ \frac{dr}{r}\le M_0.
		\end{equation}
		Then, $\mu$ is $n$-rectifiable.
	\end{lemma}
	\begin{proof}[Proof of \propref{prop:suff rectif} using \lemref{lem:reduction suff rect}]
		To show that $\mu$ is rectifiable, it suffices to prove that for any bounded $E\subset \supp\mu$ of positive measure there exists $F\subset E,\ \mu(F)>0,$ such that $\restr{\mu}{F}$ is rectifiable. Given any such $E$ we may rescale it and translate it, so without loss of generality $E\subset B(0,1)$.
		
		Since $0<\Theta^{n,*}(\mu,x)$ and $\Theta_*^{n}(\mu,x)<\infty$ for $\mu$-a.e. $x\in E$, choosing $C_*>1$ big enough, we get that the set
		\begin{equation}\label{eq:definition of E'}
		E' = \{x\in E\ :\ \Theta^{n,*}(\mu,x)>0,\ \Theta_*^{n}(\mu,x)\le C_* \}
		\end{equation}
		has positive $\mu$-measure.
		
		Let $\{V_k\}_{k\in\mathbb{N}}$ be a countable and dense subset of $G(d,d-n)$. It is clear that for any $\alpha\in (0,1),\ V\in G(d,d-n),$ there exists $k\in\mathbb{N}$ such that $K(0,V_k,k^{-1})\subset K(0,V,\alpha)$. Set
		\begin{equation*}
		E_k = \set[3]{x\in \R^d\ :\ \int_0^{1}\bigg(\frac{\mu(K(x,V_k,k^{-1},r))}{r^n} \bigg)^p\ \frac{dr}{r}\le k }.
		\end{equation*}
		It is a simpl exercise to check that for each $k\in\mathbb{N}$ the set $E_k$ is Borel. Moreover, it follows from \eqref{eq:pointwise conical energy finite} that $\mu(\R^d\setminus \bigcup_k E_k)=0$. Pick any $k\in\mathbb{N}$ with $\mu(E'\cap E_k)>0$ and set $F=E'\cap E_k$. Using the Lebesgue differentiation theorem and \eqref{eq:definition of E'}, it is easy to see that for $\mu$-a.e. $x\in F$ we have $\Theta^{n,*}(\restr{\mu}{F},x)=\Theta^{n,*}(\mu,x)>0$ and  $\Theta_*^{n}(\restr{\mu}{F},x)=\Theta_*^{n}(\mu,x)\le C_*$. Hence, $\restr{\mu}{F}$ satisfies the assumptions of \lemref{lem:reduction suff rect}, and so it is $n$-rectifiable.
%
	\end{proof}
	\subsection{Proof of \lemref{lem:reduction suff rect} for \texorpdfstring{$\mu\ll\Hn$}{mu<<Hn}}\label{sec:special case}
	First, we will prove \lemref{lem:reduction suff rect} under the additional assumption
	$\Theta^{n,*}(\mu,x)<\infty$ for $\mu$-a.e. $x\in\R^d$ (which is equivalent to $\mu\ll\Hn$). 
	
	Using similar tricks as in the proof of \propref{prop:suff rectif}, it is easy to see that we may actually replace $\Theta^{n,*}(\mu,x)<\infty$ by a stronger condition: without loss of generality, we can assume that there exist $C_1>0$ and $r_0>0$ such that for all $x\in\supp\mu$ and all $0<r\le r_0$ we have
	\begin{equation}\label{eq:upper density bounded}
	\mu(B(x,r))\le C_1\, r^n.
	\end{equation}
%
	Then, the assumptions of \lemref{lem:main lemma} are satisfied, and we get a family of cubes $\Top\subset\D^{db}$ and an associated family of Lipschitz graphs $\Gamma_R,\ R\in\Top$. The cubes from $\Top$ satisfy the packing condition
	\begin{equation*}
	\sum_{R\in\Top}\Theta_{\mu}(2B_R)^p\mu(R)\lesssim \mu(\R^d) + \E(\R^d,V,\alpha)\le (1+M_0)\mu(B(0,1)).
	\end{equation*}
	It follows that for $\mu$-a.e. $x\in\R^d$ we have
	\begin{equation*}
	\sum_{R\in\Top:\, R\ni x}\Theta_{\mu}(2B_R)^p<\infty.
	\end{equation*}
	Fix some $x$ for which the above holds. Denote by $R_0\supset R_1\supset\dots$ the sequence of cubes from $\Top$ containing $x$. We claim that for $\mu$-a.e. $x$ this sequence is finite. 
	
	Indeed, if the sequence is infinite, we have $\Theta_{\mu}(2B_{R_i})\to 0$. On the other hand, let $i\ge 0$ and $r(R_{i+1})\le r\le r(R_i)$. Since $R_{i+1}\in\Next(R_i)$, we get from \eqref{eq:control on density for next cubes}
	\begin{equation*}
	\Theta_{\mu}(x,r)\lesssim_A \Theta_{\mu}(2B_{R_i}).
	\end{equation*}
	In consequence, 
	\begin{equation*}
	\Theta^{n,*}(\mu,x)\lesssim_A \limsup_{i\to\infty}\Theta_{\mu}(2B_{R_i}) = 0,
	\end{equation*}
	which may happen only on a set of $\mu$-measure $0$ because $\Theta^{n,*}(\mu,x)>0$ for $\mu$-a.e. $x\in\R^d$.
	
	Hence, for $\mu$-a.e. $x\in\R^d$ the sequence $\{R_i\}$ is finite. This means that if $R_{k}$ denotes the smallest $\Top$ cube containing $x$, then $x\in\Good(R_k)$. It follows that
	\begin{equation*}
	\mu\bigg(\R^d\setminus\bigcup_{R\in\Top}\Good(R)\bigg)=0.
	\end{equation*}
	By \lemref{lem:main lemma} (ii) we have $\mu(\Good(R_k)\setminus\Gamma_{R_k})=0$. Hence,
	\begin{equation*}
	\mu\bigg(\R^d\setminus\bigcup_{R\in\Top}\Gamma_R\bigg)=0,
	\end{equation*}
	and so $\mu$ is $n$-rectifiable.
	\subsection{Proof of \lemref{lem:reduction suff rect} in full generality}
	Thanks to the partial result from the preceding subsection, it is clear that to prove \lemref{lem:reduction suff rect} in full generality, it suffices to show that for $\mu$ satisfying the assumptions of \lemref{lem:reduction suff rect} we have
	\begin{equation*}
	M_n\mu(x) = \sup_{r>0}\frac{\mu(B(x,r))}{r^n}<\infty\quad \text{for $\mu$-a.e. $x\in B(0,1)$}.
	\end{equation*}
	To do that, we will use techniques from \cite[Section 5]{tolsa2017rectifiability}. 

	\begin{lemma}[{\cite[Lemma 5.1]{tolsa2017rectifiability}}]
		Let $C>2$. Suppose that $\mu$ is a Radon measure on $\R^d$, and that $\Theta_*^{n}(\mu,x)\le C_*$ for $\mu$-a.e. $x\in\R^d$. Then, for $\mu$-a.e. $x\in\R^d$ there exists a sequence of radii $r_k\to 0$ such that
		\begin{equation}
		\mu(B(x,C r_k))\le 2C^d\mu(B(x,r_k))\le 20\, C_*C^{n+d}\,r_k^n.
		\end{equation}
	\end{lemma}
	Let $\lambda<\tfrac{1}{2}$ be a small constant depending on $\alpha$, to be chosen later. By the lemma above (used with $C=\lambda ^{-1}$) and Vitali's covering theorem (see \cite[Theorem 2.8]{mattila1999geometry}), there exists a family of pairwise disjoint closed balls $B_i,\ i\in I,$ centered at $x_i\in\supp\mu\subset B(0,1)$, which cover $\mu$-almost all of $B(0,1)$, and which satisfy
	\begin{equation*}
	\mu(B_i)\le 2\lambda ^{-d}\mu(\lambda  B_i)\le 20\, C_*\lambda ^{-d}\,r(B_i)^n,
	\end{equation*}
	and 
	\begin{equation*}
	r(B_i)\le \rho
	\end{equation*}
	for some arbitrary fixed $\rho>0$. We may assume that \eqref{eq:pointwise conical energy bounded} holds for all the centers $x_i$. Choose $I_0\subset I$ a finite subfamily such that
	\begin{equation*}
	\mu(B(0,1)\setminus \bigcup_{i\in I_0} B_i)\le \varepsilon\mu(B(0,1)),
	\end{equation*}
	where $\varepsilon>0$ is some small constant. Clearly, $I_0=I_0(\rho,\varepsilon)$.
	
	For each $i\in I_0$ we consider an $n$-dimensional disk $D_i$, centered at $x_i$, parallel to $V^{\perp}\in G(d,n),$ with radius $\lambda  r(B_i)$. We define an approximating measure
	\begin{equation*}
	\nu = \sum_{i\in I_0}\frac{\mu(B_i)}{\Hn(D_i)}\Hr{D_i}.
	\end{equation*}
	Note that 
	\begin{equation}\label{eq:growth of nu}
	\nu(D_i) = \mu(B_i)\approx_{\lambda }\mu(\lambda  B_i)\lesssim_{\lambda } C_* r(B_i)^n.
	\end{equation}
	Moreover, since $I_0$ is a finite family, the definition of $\nu$ and \eqref{eq:growth of nu} imply that $\nu$ satisfies the polynomial growth condition \eqref{eq:growth condition main lemma} with  $r_0 = \min_{i\in I_0} r(B_i)/2$ and $C_1=C(\lambda ) C_*$, i.e. for $0<r<r_0$ and $x\in\supp\nu$
	\begin{equation}\label{eq:bounded growth of nu}
	\nu(B(x,r))\le C(\lambda ) C_*r^n.
	\end{equation}
	\begin{lemma}\label{lem:conical energy of nu}
		For $\lambda =\lambda (\alpha)<\tfrac{1}{2}$ small enough, we have
		\begin{equation*}
		\mathcal{E}_{\nu,p}(\R^d,V,\tfrac{1}{2}\alpha) \lesssim_{\lambda ,p} (M_0 + \mu(B(0,1))^p)\mu(B(0,1)).
		\end{equation*}
		The implicit constant does not depend on $\rho, \varepsilon$.
	\end{lemma}
	\begin{proof}
		Let $i\in I_0$ and $x\in D_i$. We will estimate the $\nu$-measure of $K(x,V,\tfrac{1}{2}\alpha,r)$. 
		
		First, note that  $\nu(K(x,V,\tfrac{1}{2}\alpha,r)) = \nu(K(x,V,\tfrac{1}{2}\alpha,r)\setminus B_i)$. Indeed, $B_i\cap \supp\nu= D_i$, and $D_i\cap K(x,V,\tfrac{1}{2}\alpha)=\varnothing$ because $D_i$ is parallel to $V^{\perp}$. Thus, $\nu(K(x,V,\tfrac{1}{2}\alpha,r)\cap B_i)=0$.
		It follows immediately that for $r\le (1-\lambda  )r(B_i)$ we have $\nu(K(x,V,\tfrac{1}{2}\alpha,r))=0$.
		
		Concerning $r>(1-\lambda  )r(B_i)$, if $\lambda =\lambda (\alpha)$ is small enough, then
		\begin{equation*}
		K(x,V,\tfrac{1}{2}\alpha,r)\setminus B_i\subset K(x_i,V,\tfrac{3}{4}\alpha,2r)\setminus B_i
		\end{equation*}
		because $x\in\lambda  B_i$. Thus, it suffices to estimate $\nu(K(x_i,V,\tfrac{3}{4}\alpha,2r)\setminus B_i)$.
		
		Suppose $r>(1-\lambda  )r(B_i)$ and $j\in I_0$ is such that $D_j\cap K(x_i,V,\tfrac{3}{4}\alpha,2r)\setminus B_i\neq\varnothing$. Since $B_i$ and $B_j$ are disjoint, we have
		\begin{equation*}
		r(B_j) + r(B_i)+\dist(B_i,B_j)\le 3r\quad\text{and}\quad \dist(D_i,D_j)\ge \frac{r(B_i)}{2} + \frac{r(B_j)}{2}.
		\end{equation*}
		It follows easily that, for $\lambda =\lambda (\alpha)$ small enough, we get $\lambda  B_j\subset K(x_i,V,\alpha,4r)$. Thus,
		\begin{multline*}
		\nu(K(x_i,V,\tfrac{3}{4}\alpha,2r))=\nu(K(x_i,V,\tfrac{3}{4}\alpha,2r)\setminus B_i)\le \sum_{j\in I_0:\lambda  B_j\subset K(x_i,V,\alpha,4r)}\nu(D_j)\\
		\overset{\eqref{eq:growth of nu}}{\approx_{\lambda }} \sum_{j\in I_0:\lambda  B_j\subset K(x_i,V,\alpha,4r)} \mu(\lambda  B_j)\le \mu( K(x_i,V,\alpha,4r)).
		\end{multline*}
		Hence,
		\begin{equation*}
		\int_0^{1/4}\bigg(\frac{\nu(K(x_i,V,\tfrac{3}{4}\alpha,2r))}{r^n} \bigg)^p\ \frac{dr}{r}\lesssim_{\lambda } \int_0^{1}\bigg(\frac{\mu(K(x_i,V,\alpha,r))}{r^n} \bigg)^p\ \frac{dr}{r}\overset{\eqref{eq:pointwise conical energy bounded}}{\le}  M_0.
		\end{equation*}
		This gives
		\begin{multline*}
		\int_{D_i}\int_0^{\infty} \bigg(\frac{\nu(K(x,V,\tfrac{1}{2}\alpha,r))}{r^n} \bigg)^p\ \frac{dr}{r}d\nu(x)\le 	\int_{D_i}\int_0^{\infty} \bigg(\frac{\nu(K(x_i,V,\tfrac{3}{4}\alpha,r))}{r^n} \bigg)^p\ \frac{dr}{r}d\nu(x)\\
		 \le C(\lambda )M_0\nu(D_i) + \int_{D_i}\int_{1/4}^{\infty} \bigg(\frac{\nu(\R^d)}{r^n} \bigg)^p\ \frac{dr}{r}d\nu(x)\\
		\lesssim_{\lambda ,p} M_0\nu(D_i) + \nu(\R^d)^p\nu(D_i)\le M_0\mu(B_i) + \mu(B(0,1))^p\mu(B_i).
		\end{multline*}
		Summing over $i\in I_0$ yields
		\begin{equation*}
		\mathcal{E}_{\nu,p}(\R^d,V,\tfrac{1}{2}\alpha) \lesssim_{\lambda ,p} (M_0 + \mu(B(0,1))^p)\mu(B(0,1)).
		\end{equation*}
	\end{proof}
	
	\begin{lemma}\label{lem:Mnnu estimate}
		For $\lambda =\lambda (\alpha)<\tfrac{1}{2}$ small enough, we have
		\begin{equation*}
		\int M_n\nu(x)^p\ d\nu(x) \lesssim_{\alpha,\lambda ,p} \big((C_*)^p+ M_0 + \mu(B(0,1))^p\big)\mu(B(0,1)).
		\end{equation*}
		The constants on the right hand side do not depend on $\rho, \varepsilon$.
	\end{lemma}
	\begin{proof}
		By \eqref{eq:bounded growth of nu} and \lemref{lem:conical energy of nu}, we may use \lemref{lem:main lemma} to get a family of cubes $\Top_{\nu}$ satisfying properties (i)-(iii) of \lemref{lem:main lemma}, and such that
		\begin{multline}\label{eq:packing estimate to the rescue}
		\sum_{R\in\Top_{\nu}} \Theta_{\nu}(2B_R)^p\nu(R)\lesssim_{\alpha,\lambda } (C_*)^p\nu(\R^d) + C(p)(M_0 + \mu(B(0,1))^p)\mu(B(0,1))\\
		\lesssim_{\alpha,\lambda ,p} \big((C_*)^p+ M_0 + \mu(B(0,1))^p\big)\mu(B(0,1)).
		\end{multline}
		
		Now, the property (iii) of \lemref{lem:main lemma} lets us estimate $M_{n}\nu(x)$. Indeed, suppose $x\in\supp\nu$, and let $r_1>0$ be such that
		\begin{equation*}
		M_n\nu(x)\le 2\frac{\nu(B(x,r_1))}{r_1^n}.
		\end{equation*}
		Since $\supp\nu\subset B(0,2)$, we have $r_1\le 4$. Let $Q\in\mathcal{D}_{\nu}$ be the smallest cube satisfying $x\in Q$ and $B(x,r_1)\cap\supp\nu\subset 2B_Q$ (such a cube exists because the largest cube $Q_0:=\supp\nu$ clearly satsfies $\supp\nu\subset 2B_{Q_0}$). Let $R\in\Top_{\nu}$ be the top cube such that $Q\in\Tr(R)$. Clearly, $\ell(Q)\approx r_1$. By \lemref{lem:main lemma} (iii), we have 
		\begin{equation*}
		\frac{\nu(B(x,r_1))}{r_1^n}\lesssim\Theta_{\nu}(2B_Q)\lesssim\Theta_{\nu}(2B_R).
		\end{equation*}
		Thus, $M_n\nu(x)^p\lesssim \sum_{R\in\Top_{\nu}}\one_{R}(x)\Theta_{\nu}(2B_R)^p$. Integrating with respect to $\nu$ and applying \eqref{eq:packing estimate to the rescue} yields the desired estimate.
	\end{proof}
	\begin{lemma}\label{lem:Mnmu estimate}
		We have 
		\begin{equation*}
		\int M_n\mu(x)^p\ d\mu(x) \lesssim_{\alpha,\lambda ,p} \big((C_*)^p+ M_0 + \mu(B(0,1))^p\big)\mu(B(0,1)).
		\end{equation*}
		In particular, $M_n\mu(x)<\infty$ for $\mu$-a.e. $x\in B(0,1)$.
	\end{lemma}
	\begin{proof}
		Denote 
		\begin{equation*}
		M_{n,\rho}\mu(x) = \sup_{r\ge \rho}\frac{\mu(B(x,r))}{r^n}.
		\end{equation*}
		Recall that $I_0=I_0(\rho,\varepsilon)$ and set
		\begin{equation*}
		E_{\varepsilon,\rho} = \supp\mu\cap\bigcup_{i\in I_0} B_i.
		\end{equation*}
		We claim that
		\begin{equation}\label{eq:Mnrho estimate}
		\int_{E_{\varepsilon,\rho}}M_{n,\rho}(\one_{E_{\varepsilon,\rho}}\mu)(x)^p\ d\mu(x)\lesssim \int M_{n,\rho}\nu(x)^p\ d\nu(x).
		\end{equation}
		Indeed, let $x,x'\in B_j,\ j\in I_0$, and $r\ge \rho$. Then, using repeatedly the fact that $r(B_i)\le \rho\le r$ for $i\in I_0$,
		\begin{multline*}
		\mu(B(x,r)\cap E_{\varepsilon,\rho})\le \mu(B(x',3r)\cap E_{\varepsilon,\rho}) \le \sum_{i\in I_0: B_i\cap B(x',3r)\neq\varnothing}\mu(B_i)\\
		= \sum_{i\in I_0: B_i\cap B(x',3r)\neq\varnothing}\nu(D_i)\le \nu(B(x',5r)).
		\end{multline*}
		Hence, for all $x\in B_j,\ j\in I_0,$
		\begin{equation*}
		M_{n,\rho}(\one_{E_{\varepsilon,\rho}}\mu)(x)\le 5^n \inf_{x'\in B_j} M_{n,\rho}\nu(x').
		\end{equation*}
		Integrating both sides of the inequality with respect to $\mu$ in $E_{\varepsilon,\rho}$ yields \eqref{eq:Mnrho estimate}.
		
		\lemref{lem:Mnnu estimate} and \eqref{eq:Mnrho estimate} give
		\begin{equation*}
			\int_{E_{\varepsilon,\rho}}M_{n,\rho}(\one_{E_{\varepsilon,\rho}}\mu)(x)^p\ d\mu(x)\le C(\alpha,\lambda ,p) \big((C_*)^p+ M_0 + \mu(B(0,1))^p\big)\mu(B(0,1)) =: K,
		\end{equation*}
		where $K$ is independent of $\rho$ and $\varepsilon$.
		
		Set $\varepsilon_k = 2^{-k}$. Observe that, for a fixed $\rho>0$, we have $\mu(\R^d\setminus\liminf_k E_{\varepsilon_k,\rho})=0$, where
		\begin{equation*}
		\liminf_k E_{\varepsilon_k,\rho} = \bigcup_{j=1}^{\infty} G_j\quad\text{and}\quad G_j = \bigcap_{k=j}^{\infty}E_{\varepsilon_k,\rho}.
		\end{equation*}
		The inclusion $G_j\subset E_{\varepsilon_j,\rho}$ gives
		\begin{equation*}
		\int_{G_j}M_{n,\rho}(\one_{G_j}\mu)(x)^p\ d\mu(x)\le \int_{E_{\varepsilon_j,\rho}}M_{n,\rho}(\one_{E_{\varepsilon_j,\rho}}\mu)(x)^p\ d\mu(x)\le K.
		\end{equation*}
		Since the sequence of sets $G_j$ is increasing, we easily get that for $\mu$-a.e. $x\in B(0,1)$
		\begin{equation*}
		\one_{G_j}(x)\,M_{n,\rho}(\one_{G_j}\mu)(x) \xrightarrow{j\to\infty} M_{n,\rho}\mu(x),
		\end{equation*}
		and the convergence is monotone. Hence, by monotone convergence theorem,
		\begin{equation*}
		\int M_{n,\rho}\mu(x)^p\ d\mu(x)\le K.
		\end{equation*}
		The estimate is uniform in $\rho$, and so once again monotone convergence gives
		\begin{equation*}
		\int M_{n}\mu(x)^p\ d\mu(x)\le K.
		\end{equation*}
	\end{proof}
	Taking into account \lemref{lem:Mnmu estimate} and Section \ref{sec:special case}, the proof of \lemref{lem:reduction suff rect} is finished.

	\section{Necessary condition for rectifiability}\label{sec:necessary rectifiability}
	In this section we will prove the following.
	\begin{prop}\label{prop:necessary condition for rect}
		Suppose $\mu$ is an $n$-rectifiable measure on $\R^d$, and $1\le p<\infty$. Then, for $\mu$-a.e. $x\in\R^d$ there exists $V_x\in G(d,d-n)$ such that for any $\alpha\in (0,1)$ we have
		\begin{equation*}
		\int_0^{1}\bigg(\frac{\mu(K(x,V_x,\alpha,r))}{r^n} \bigg)^p\ \frac{dr}{r}<\infty.
		\end{equation*}
	\end{prop}
	
	First, we recall the definition of $\beta_2$ numbers, as defined by David and Semmes \cite{david1991singular}.
	\begin{definition}
		Given a Radon measure $\mu$, $x\in\supp\mu$, $r>0,$ and an $n$-plane $L$, define
		\begin{equation*}
		\beta_{\mu,2}(x,r) = \inf_L \bigg(\frac{1}{r^n} \int_{B(x,r)}\bigg(\frac{\dist(y,L)}{r}\bigg)^2\ d\mu(y)\bigg)^{1/2},
		\end{equation*}
		 where the infimum is taken over all $n$-planes intersecting $B(x,r)$.
	\end{definition}
	Tolsa showed the following necessary condition for rectifiability in terms of $\beta_2$ numbers.
	\begin{theorem}[\cite{tolsa2015characterization}]\label{thm:beta necessary}
		Suppose $\mu$ is an $n$-rectifiable measure on $\R^d$. Then, for $\mu$-a.e. $x\in\R^d$ we have
		\begin{equation}\label{eq:beta square function}
		\int_0^1\beta_{\mu,2}(x,r)^2\ \frac{dr}{r}<\infty.
		\end{equation}
	\end{theorem}
	\begin{remark}
		When showing that rectifiable sets have approximate tangents almost everywhere one uses the so-called \emph{linear approximation} properties, see \cite[Theorems 15.11 and 15.19]{mattila1999geometry}. The theorem of Tolsa improves on the linear approximation property, and that allows us to improve on the classical approximate tangent plane result.
	\end{remark}

	For a fixed $n$-rectifiable measure $\mu$, let $L_{x,r}$ denote a plane minimizing $\beta_{\mu,2}(x,r)$ (it may be non-unique, in which case we simply choose one of the minimizers). 
	
	Recall that in Definition \ref{def:approx tangent} we defined the approximate tangent to $\mu$ to be an $n$-plane $W'_x\in G(d,n)$. Let $W_x:=x+W_x'$, whenever the approximate tangent exists and is unique (that is for $\mu$-a.e. $x$, by \thmref{thm:classical tangents}). In order to apply Tolsa's result in our setting, we need the following intuitively clear result.
	
	\begin{lemma}\label{lem:approx planes converge to approx tangent}
		Let $\mu$ be a rectifiable measure. Then for $\mu$-a.e. $x\in\supp\mu$ we have
		\begin{equation}\label{eq:Lxr converge to approximate tangent}
			\frac{\dist_H(L_{x,r}\cap B(x,r), W_x\cap B(x,r))}{r}\xrightarrow{r\to 0 } 0.
		\end{equation}
	\end{lemma}

	A relatively simple (although lengthy) proof can be found in \appref{sec:appendix}.	
	
	Before proving \propref{prop:necessary condition for rect} we need one more lemma. Recall that if $\alpha>0,\ W$ is an $n$-plane, and $0<r<R,$ then $K(x,W^{\perp},\alpha,r,R)=K(x,W^{\perp},\alpha,R)\setminus K(x,W^{\perp},\alpha,r)$.
	\begin{lemma}\label{lem:cone outside of tube}
		Let $\alpha,\varepsilon\in (0,1)$ be some constants satisfying $\eta := 1-\alpha - 3\varepsilon>0$. Let $x\in\R^d,\ r>0,$ and suppose that $W$ and $L$ are $n$-planes satisfying $x\in W$ and
		\begin{equation}\label{eq:L close to W}
		{\dist_H(L\cap B(x,r), W\cap B(x,r))}\le \varepsilon r.
		\end{equation}
		Then 
		\begin{equation*}
		K(x,W^{\perp},\alpha,r,2r)\subset B(x,2r)\setminus B_{\eta r}(L).
		\end{equation*}
	\end{lemma}
	\begin{proof}
		Suppose $y\in K(x,W^{\perp},\alpha,r,2r)$, so that $r<|x-y|<2r$ and $|x-\pi_{W}(y)|<\alpha|x-y|$. We need to show that $\dist(y, L)>\eta r$.
		
		Set $y' = \pi_{L}(y),\ x' = \pi_{L}(x)$. Then
		\begin{multline*}
		\dist(y, L) = |y-y'|\ge |x-y|-|x'-y'|-|x-x'|\\
		= |x-y| - |x'-y'| - \dist(x,L)
		\overset{\eqref{eq:L close to W}}{\ge} |x-y| - |x'-y'| - \varepsilon r.
		\end{multline*}
		Let $\tilde{\pi}_{W}$ and $\tilde{\pi}_{L}$ denote the orthogonal projections onto the $n$-planes parallel to $W$ and $L$ passing through the origin. It follows from \eqref{eq:L close to W} that $\norm{\tilde{\pi}_{W}-\tilde{\pi}_{L}}_{op}\le \varepsilon$. Thus,
		\begin{equation*}
		|x'-y'| = \abs{\tilde{\pi}_{L}(x-y)}\le \abs{\tilde{\pi}_{W}(x-y)} + \norm{\tilde{\pi}_{W}-\tilde{\pi}_{L}}_{op}|x-y|\le \abs{\tilde{\pi}_{W}(x-y)} +2\varepsilon r.
		\end{equation*}
		Hence, using the fact that $|\tilde{\pi}_{W}(x-y)| =|x-\pi_{W}(y)| <\alpha|x-y|$, we get from the two estimates above
		\begin{equation*}
		\dist(y, L) \ge |x-y| - \abs[0]{\tilde{\pi}_{W}(x-y)} - 3\varepsilon r 
		\ge (1-\alpha)|x-y| - 3\varepsilon r \ge (1-\alpha - 3\varepsilon)r = \eta r.
		\end{equation*}
	\end{proof}
	
	\begin{proof}[Proof of \propref{prop:necessary condition for rect}]
		Let $\mu$ be $n$-rectifiable. For $r>0$ and $x\in\supp\mu$ let $L_{x,r}$ be the $n$-plane minimizing $\beta_{\mu,2}(x,r)$. We know that for $\mu$-a.e. $x\in\supp\mu$ we have \eqref{eq:beta square function} and \eqref{eq:Lxr converge to approximate tangent} (in particular, the approximate tangent plane $W_x$ exists). Fix such $x$. Set $V_x = W_x^{\perp}$, let $\alpha\in (0,1)$ be arbitrary, and for $0<r<R$ set $K(r) = K(x,V_x,\alpha, r),\ K(r,R) = K(x,V_x,\alpha, r,R)$. We will show that 
		\begin{equation}\label{eq:necessary condition for rect}
		\int_0^{1}\bigg(\frac{\mu(K(r))}{r^n} \bigg)^p\ \frac{dr}{r}<\infty.
		\end{equation}
		
		Let $\varepsilon>0$ be a constant so small that $\eta := 1-\alpha-3\varepsilon>0$. Use \lemref{lem:approx planes converge to approx tangent} to find $r_0>0$ such that for $0< r\le r_0$ we have 
		\begin{equation*}
		{\dist_H(L_{x,r}\cap B(x,r), W_x\cap B(x,r))}\le \varepsilon r.
		\end{equation*}
		
		Then, it follows from \lemref{lem:cone outside of tube} that for all $0<r\le r_0$
		\begin{equation*}
		K(r,2r)\subset B(x,2r)\setminus B_{\eta r}(L_{x,r}).
		\end{equation*}
		Note that by Chebyshev's inequality
		\begin{equation*}
		\mu(B(x,2r)\setminus B_{\eta r}(L_{x,r}))\le {\eta^{-2}}\int_{B(x,2r)}\bigg(\frac{\dist(y,L_{x,r})}{r}\bigg)^2\ d\mu(y)=\eta^{-2}(2r)^n{\beta}_{\mu,2}(x,2r)^2.
		\end{equation*}
		Hence, for $0<r\le r_0$ we have
		\begin{equation*}
		\frac{\mu(K(r,2r))}{r^n}\lesssim_{\eta}{\beta}_{\mu,2}(x,2r)^2,
		\end{equation*}
		and so 
		\begin{equation}\label{eq:integral over annuli estimated with betas}
		\int_0^{r_0}\frac{\mu(K(r,2r))}{r^n}\ \frac{dr}{r}\lesssim_{\eta}\int_0^{2r_0}{\beta}_{\mu,2}(x,r)^2\ \frac{dr}{r} \overset{\eqref{eq:beta square function}}{<} \infty.
		\end{equation}
		Now, observe that for any integer $N>0$
		\begin{multline*}
		\int_{2^{-N}r_0}^{r_0}\frac{\mu(K(r))}{r^n}\ \frac{dr}{r}\lesssim(r_0)^{-n}\sum_{k=0}^{N}\mu(K(2^{-k}r_0))2^{kn}\\
		\le 2^n(r_0)^{-n}\sum_{k=0}^{N}\mu(K(2^{-k}r_0))2^{kn} - (r_0)^{-n}\sum_{k=0}^{N}\mu(K(2^{-k}r_0))2^{kn}\\
		= (r_0)^{-n}\sum_{k=0}^{N}\mu(K(2^{-k}r_0))2^{(k+1)n} - (r_0)^{-n}\sum_{k=0}^{N}\mu(K(2^{-k}r_0))2^{kn}\\
		=(r_0)^{-n}\sum_{k=1}^{N+1}\big(\mu(K(2^{-k+1}r_0))- \mu(K(2^{-k}r_0))\big)2^{kn} + \frac{\mu(K(2^{-(N+1)}r_0))}{(2^{-(N+1)}r_0)^n} - \frac{\mu(K(r_0))}{r_0^n}\\
		\lesssim \int_{0}^{r_0/2}\frac{\mu(K(r,2r))}{r^n}\ \frac{dr}{r} +  \Theta_{\mu}(x,2^{-(N+1)}r_0) +0.
		\end{multline*}
		Letting $N\to\infty$, we get from the above and \eqref{eq:integral over annuli estimated with betas} that
		\begin{equation*}
		\int_0^{r_0}\frac{\mu(K(r))}{r^n}\ \frac{dr}{r} \lesssim_{\eta}\int_0^{2r_0}{\beta}_{\mu,2}(x,r)^2\ \frac{dr}{r} + \Theta^{n,*}(\mu,x)< \infty,
		\end{equation*}
		for $\mu$-a.e. $x\in\supp\mu$, where we also used the fact that $\Theta^{n,*}(\mu,x)<\infty$ $\mu$-almost everywhere (because $\mu$ is $n$-rectifiable).
		The integral $\int_{2r_0}^1\frac{\mu(K(r))}{r^n}\ \frac{dr}{r}$ is obviously finite, and so we get that
		\begin{equation*}
		\int_0^{1}\frac{\mu(K(r))}{r^n}\ \frac{dr}{r}<\infty,
		\end{equation*}
		which is precisely \eqref{eq:necessary condition for rect} with $p=1$. To get the same with $p>1$, note that since $\Theta^{n,*}(\mu,x)<\infty$ for $\mu$-a.e. $x$, we have
		\begin{equation*}
		\int_0^{1}\bigg(\frac{\mu(K(r))}{r^n}\bigg)^p\ \frac{dr}{r}\le \int_0^{1}\frac{\mu(K(r))}{r^n}\Theta_{\mu}(x,r)^{p-1}\ \frac{dr}{r}\le \sup_{0<r<1}\Theta_{\mu}(x,r)^{p-1}\int_0^{1}\frac{\mu(K(r))}{r^n}\ \frac{dr}{r}<\infty.
		\end{equation*}		
	\end{proof}
	\section{Sufficient condition for BPLG}\label{sec:suff BPLG}
	In this section we prove the ``sufficient part'' of \thmref{thm:suff BPLG}. After a suitable translation and rescaling, it suffices to show the following:
	\begin{prop}\label{prop:suff BPLG}
		Suppose $p\ge 1,\ E\subset\R^d$ is $n$-AD-regular, and $0\in E$. Let $\alpha>0,\ M_0>1,\ \kappa>0$, and assume that there exist $F\subset E\cap B(0,1)$ and $V\in G(d,d-n)$, such that $\Hn(F)\ge \kappa$, and for all $x\in F$
		\begin{equation}\label{eq:BPLG main assumption}
		\int_0^{1} \bigg(\frac{\Hn(K(x,V,\alpha,r)\cap F)}{r^n} \bigg)^p \ \frac{dr}{r}\le M_0.
		\end{equation}
		Then there exists a Lipschitz graph $\Gamma$, with Lipschitz constant depending on $\alpha, n, d$, such that
		\begin{equation}\label{eq:Gamma for BPLG}
		\Hn(F\cap\Gamma)\gtrsim 1,
		\end{equation}
		with the implicit constant depending on $\kappa,p,M_0,\alpha,n,d$, and the $AD$-regularity constants of $E$.
	\end{prop}
	
	To prove the above we will use techniques developed in \cite{martikainen2018characterising}. 	
	Fix $V\in G(d,d-n)$. Let $\theta>0$ and $M\in\{0,1,2\dots\}$. In the language of Martikainen and Orponen, a set $E\subset\R^d$ has the \emph{$n$-dimensional $(\theta,M)$-property} if for all $x\in E$
	\begin{equation*}
	\#\{j\in\mathbb{Z}\ :\ K(x,V,\theta,2^{-j}, 2^{-j+1})\cap E\neq\varnothing \}\le M.
	\end{equation*}
	It is easy to see that if $E$ has the $n$-dimensional $(\theta,0)$-property, then $E$ is contained in a Lipschitz graph with Lipschitz constant bounded by $1/\theta$, see \cite[Remark 1.11]{martikainen2018characterising}.
	
	The main proposition of \cite{martikainen2018characterising} reads as follows.
	\begin{prop}[{\cite[Proposition 1.12]{martikainen2018characterising}}]
		Assume that $E$ is $n$-AD-regular, and assume that $F_1\subset E\cap B(0,1)$ is an $\Hn$-measurable subset with $\Hn(F_1)\approx_C 1$. Suppose further that $F_1$ satisfies the $n$-dimensional $(\theta,M)$-property for some $\theta>0,\ M\ge 0$. Then there exists and $\Hn$-measurable subset $F_2\subset F_1$ with $\Hn(F_2)\approx_{C,\theta,M} 1$ which satisfies the $(\theta/b,0)$-property. Here $b\ge 1$ is a constant depending only on $d$.
	\end{prop}

	\begin{remark}\label{rem:sufficies to construct F1}
		It follows immediately from the proposition above that if we construct $F_1\subset E\cap B(0,1)$ with $\Hn(F_1)\approx \kappa$ satisfying the $n$-dimensional $(\alpha/2,M)$-property, then we will get a Lipschitz graph $\Gamma$ such that \eqref{eq:Gamma for BPLG} holds. Hence, we will be done with the proof of \propref{prop:suff BPLG}.
	\end{remark}
	
	To construct $F_1$ we will use another lemma from \cite{martikainen2018characterising}.
	\begin{lemma}[{\cite[Lemma 2.1]{martikainen2018characterising}}]\label{lem:Fvarepsilon estimate}
		Let $E$ be an $n$-AD-regular set with $\Hn(E)\ge C>0$, let $F\subset E\cap B(0.1)$ be an $\Hn$-measurable subset, and let
		\begin{equation*}
		F_{\varepsilon} = \{ x\in F\ :\ \Hn(F\cap B(x,r_x))\le\varepsilon r_x^n\ \text{for some radius}\ 0<r_x\le 1 \}.
		\end{equation*}
		Then $\Hn(F_{\varepsilon})\lesssim \varepsilon$ with the bound depending only on $C$ and the AD-regularity constant of $E$.
	\end{lemma}
	Note that the set $F\setminus F_{\varepsilon}$ does not have to be AD-regular. Nevertheless, we gain some extra regularity that will prove useful.
	
	Now, let $E$ and $F\subset E\cap B(0,1)$ be as in the assumptions of \propref{prop:suff BPLG}. 	
	We apply \lemref{lem:Fvarepsilon estimate} to conclude that for some $\varepsilon$, depending on $\kappa$ and the AD-regularity constant of $E$, we have
	\begin{equation*}
	\Hn(F\setminus F_{\varepsilon}) \ge \frac{\kappa}{2}.
	\end{equation*}
	Set $F_1 = F\setminus F_{\varepsilon}$. 
	\begin{lemma}
		There exists $M=M(M_0,\varepsilon,\alpha,n)$ such that $F_1$ satisfies the $n$-dimensional $(\alpha/2,M)$-property.
	\end{lemma}
	\begin{proof}
		Denote by $F_{\mathsf{Bad}}\subset F_1$ the set of $x\in F_1$ such that
		\begin{equation}\label{eq:x has no alpha M property}
		\#\{j\in\mathbb{Z}\ :\ K(x,V,\alpha/2,2^{-j}, 2^{-j+1})\cap F_1\neq\varnothing \}> M.
		\end{equation}
		We will show that, if $M$ is chosen big enough, the set $F_{\mathsf{Bad}}$ is empty.
		
		Let $x\in F_{\mathsf{Bad}}$ and $j\in\mathbb{Z}$ be such that there exists $x_j \in K(x,V,\alpha/2,2^{-j}, 2^{-j+1})\cap F_1$. It is easy to see that for some $\lambda=\lambda(\alpha)$, independent of $j$, we have
		\begin{equation*}
		B(x_j,\lambda 2^{-j})\subset K(x,V,\alpha,2^{-j-1}, 2^{-j+2}).
		\end{equation*}
		Since $x_j\in F_1 = F\setminus F_{\varepsilon}$, it follows that 
		\begin{equation*}
		\Hn(F\cap B(x_j,\lambda 2^{-j}))>\varepsilon (\lambda 2^{-j})^n.
		\end{equation*}
		The two observations above give
		\begin{equation*}
		\frac{\Hn(F\cap K(x,V,\alpha,2^{-j+2}))}{(2^{-j+2})^n}\ge \frac{\Hn(F\cap K(x,V,\alpha,2^{-j-1}, 2^{-j+2}))}{(2^{-j+2})^n}\gtrsim_{\alpha,\lambda}\varepsilon.
		\end{equation*}
		By \eqref{eq:x has no alpha M property}, there are more than $M$ different scales (i.e. $j$'s) for which the above holds. Thus, for $x\in F_{\mathsf{Bad}}$ we have
		\begin{equation*}
		\int_0^1 \bigg(\frac{\Hn(K(x,V,\alpha,r)\cap F)}{r^n} \bigg)^p\ \frac{dr}{r}\gtrsim_{\alpha,\lambda} M\varepsilon^p.
		\end{equation*}		
		Taking $M=M(M_0,\varepsilon,\alpha,n,p)$ big enough we get a contradiction with \eqref{eq:BPLG main assumption}. Thus, $F_{\mathsf{Bad}}$ is empty. Now, it follows trivially by the definition of $F_{\mathsf{Bad}}$ that $F_1$ satisfies the $n$-dimensional $(\alpha/2,M)$-property.
	\end{proof}
	By \remref{rem:sufficies to construct F1}, this finishes the proof of \propref{prop:suff BPLG}.
	
	\section{Necessary condition for BPLG}\label{sec:necessary BPLG}
	In this section we prove the ``necessary part'' of \thmref{thm:suff BPLG}. After rescaling, translating, and using the BPLG property, it is clear that it suffices to show the following:
	\begin{prop}\label{prop:necessary BPLG}
		Suppose $E\subset\R^d$ is $n$-AD-regular, and $0\in E$. Let $p\ge 1$. Assume there exists a Lipschitz graph $\Gamma$ such that $\Hn(\Gamma\cap E\cap B(0,1))\ge\kappa$. Then there exists $\alpha=\alpha(\lip(\Gamma))>0,\ V\in G(d,d-n),$ and a set $F\subset \Gamma\cap E\cap B(0,1),$ such that $\Hn(F)\gtrsim\kappa$, and for $x\in F$
		\begin{equation}\label{eq:BPLG necessary assumption}
		\int_0^{1} \bigg(\frac{\Hn(K(x,V,\alpha,r)\cap E)}{r^n} \bigg)^p\ \frac{dr}{r}\le M_0,
		\end{equation}
		where  $M_0>1$ is a constant depending on $p,\, \lip(\Gamma),\ \kappa$ and the AD-regularity constant of $E$.
	\end{prop}
%
	We begin by fixing some additional notation. Set $\mu=\Hr{E}$. We will denote the AD-regularity constant of $E$ by $C_0$, so that for every $x\in E,\ 0<r<\diam(E),$ 
	\begin{equation*}
	C_0^{-1}r^n \le \mu(B(x,r))\le C_0 r^n.
	\end{equation*}
	\begin{remark}
		Since we assume that $E$ is AD-regular, the exponent $p$ in \eqref{eq:BPLG necessary assumption} does not really matter. For any $p>1$ we have
		\begin{equation*}
		\bigg(\frac{\Hn(K(x,V,\alpha,r)\cap E)}{r^n} \bigg)^p \le C_0^{p-1} \frac{\Hn(K(x,V,\alpha,r)\cap E)}{r^n} ,
		\end{equation*}
		and so it is enough to prove \eqref{eq:BPLG necessary assumption} for $p=1$.
	\end{remark}
	Set $L=\lip(\Gamma)$. Let $V\in G(d,d-n)$ be such that $\Gamma$ is an $L$-Lipschitz graph over $V^{\perp}$, and let $\theta=\theta(L)>0$ be such that 
	\begin{equation*}
	K(x,V,\theta)\cap\Gamma = \varnothing\qquad\text{for all $x\in\Gamma$.}
	\end{equation*}
	Set $\alpha=\min(\frac{\theta}{2},0.1, \frac{1}{4L})$.
	
	For every $x\in E\cap B(0,1)\setminus\Gamma$ consider the ball $B_x = B(x,0.01 \dist(x,\Gamma))$.
	We use the $5r$-covering lemma to choose a countable subfamily of pairwise disjoint balls $B_j=B(x_j, r_j),\ r_j=0.01 \dist(x_j,\Gamma),\ j\in \Z,$ such that
	\begin{equation*}
	E\cap B(0,1)\setminus\Gamma\subset \bigcup_{j\in\Z} 5B_j.
	\end{equation*}
	Observe that
	\begin{equation}\label{eq:radii summable}
	\sum_{j\in\Z} r_j^n\le C_0 \sum_{j\in\Z} \mu(B_j) = C_0 \mu\big(\bigcup_{j\in\Z} B_j\big)\le C_0\mu(B(0,2))\lesssim C_0^2.
	\end{equation}
	For each $j\in\Z$ set
	\begin{equation*}
	K_j = \bigcup_{y\in 5B_j} K(y,V,\alpha),\qquad K_j(r) = \bigcup_{y\in 5B_j} K(y,V,\alpha,r).
	\end{equation*}
	\begin{lemma}\label{lem:Kj estimate}
		For each $j\in\Z$ we have
		\begin{equation}\label{eq:Kj estimate}
		\Hn(K_j\cap\Gamma)\lesssim_{L} r_j^n.
		\end{equation}
		Moreover, 
		\begin{equation}\label{eq:Kjr doesnt intersect Gamma}
		K_j(r)\cap\Gamma=\varnothing \qquad\text{for $r<r_j$}.
		\end{equation}
	\end{lemma}
	\begin{proof}
		\eqref{eq:Kjr doesnt intersect Gamma} is very easy -- observe that for $r<r_j$ we have $K_j(r)\subset 6B_j$, and so for $y\in K_j(r)$ 
		\begin{equation*}
		\dist(y,\Gamma)\ge \dist(x_j,\Gamma)- 6r_j = (1-0.06)\dist(x_j,\Gamma)>0.
		\end{equation*}
		
		Concerning \eqref{eq:Kj estimate}, we claim that since $\Gamma = \graph(F)$ for some $L$-Lipschitz function $F:V^{\perp}\to V$, and since $\alpha$ is sufficiently small, for all $x\in\R^d$ we have 
		\begin{equation}\label{eq:cone intersects graph}
		K(x,V,\alpha)\cap\Gamma \subset B(x, C\dist(x,\Gamma)),
		\end{equation}
		where $C=C(L)>1.$ Indeed, if $\dist(x,\Gamma)=0$, then $	K(x,V,\alpha)\cap\Gamma=\varnothing$ and there is nothing to prove. Suppose $\dist(x,\Gamma)>0,$ $y\in K(x,V,\alpha)\cap\Gamma$, and let $z\in\Gamma$ be the image of $x$ under the projection onto $\Gamma$ orthogonal to $V^{\perp}$, i.e. $z = \pi^{\perp}_V(x)+F(\pi^{\perp}_V(x))$. 
		
		Observe that, since $\Gamma$ is a Lipschitz graph,
		\begin{equation*}
		|x-z|\lesssim_L \dist(x,\Gamma),
		\end{equation*}
		and also $\pi_V^{\perp}(x) = \pi_V^{\perp}(z)$. By the definition of a cone, $y\in K(x,V,\alpha)$ gives
		\begin{equation*}
		|\pi_V^{\perp}(z-y)| = |\pi_V^{\perp}(x-y)|<\alpha |x-y|.
		\end{equation*}
		On the other hand, $y\in\Gamma$ and the above imply
		\begin{equation*}
		|\pi_V(z-y)|\le L|\pi_V^{\perp}(z-y)|<L\alpha |x-y|.
		\end{equation*}
		The three estimates above yield
		\begin{multline*}
		|x-y|\le |x-z| + |z-y|\le C(L)\dist(x,\Gamma) + |\pi_V^{\perp}(z-y)| + |\pi_V(z-y)|\\ \le C(L)\dist(x,\Gamma)  + \alpha |x-y|+ L\alpha |x-y|\le C(L)\dist(x,\Gamma) + \frac{1}{2}|x-y|.
		\end{multline*}
		Hence, $|x-y|\lesssim_L \dist(x,\Gamma)$ and \eqref{eq:cone intersects graph} follows.
		
		Now, going back to \eqref{eq:Kj estimate}, note that for $y\in 5B_j$ we have $\dist(y,\Gamma)\approx r_j$, so that $K(y,V,\alpha)\cap\Gamma\subset B(y,Cr_j)$ for some $C=C(L)$. Moreover, $B(y,Cr_j)\subset B(x_j,10Cr_j)$. Therefore, $K_j\cap\Gamma\subset B(x_j,10Cr_j)\cap\Gamma$, and \eqref{eq:Kj estimate} easily follows.
	\end{proof}
	\begin{proof}[Proof of \propref{prop:necessary BPLG}]
		Let $x\in\Gamma\cap B(0,1)$ and $0<r<1$. Since $\{5B_j\}_{j\in\Z}$ cover $E\cap B(0,1)\setminus\Gamma$, and $K(x,V,\alpha,r)\cap\Gamma=\varnothing$, we have
		\begin{equation*}
		\mu(K(x,V,\alpha,r))\le \sum_{j\in\Z\, :\, 5B_j\cap K(x,V,\alpha,r)\neq\varnothing} \mu(5B_j) \lesssim C_0 \sum_{j\in\Z\, :\, 5B_j\cap K(x,V,\alpha,r)\neq\varnothing} r_j^n.
		\end{equation*}
		Notice that $5B_j\cap K(x,V,\alpha,r)\neq\varnothing$ if and only if $x\in K_j(r)$. Hence, using the above and \lemref{lem:Kj estimate} yields
		\begin{multline*}
		\int_{\Gamma\cap B(0,1)}\int_0^{1} \frac{\mu(K(x,V,\alpha,r))}{r^n}\ \frac{dr}{r}d\Hn(x)\lesssim_{C_0} \int_{\Gamma\cap B(0,1)}\int_0^{1} \frac{1}{r^n} \sum_{j\in\Z} r_j^n \one_{K_j(r)}(x) \frac{dr}{r}d\Hn(x)\\
		 = \sum_{j\in\Z} r_j^n \int_{\Gamma\cap B(0,1)}\int_0^{1} \frac{1}{r^n}  \one_{K_j(r)}(x) \frac{dr}{r}d\Hn(x) \overset{\eqref{eq:Kjr doesnt intersect Gamma}}{\le} \sum_{j\in\Z} r_j^n \int_{K_j\cap\Gamma}\int_{r_j}^{1} \frac{1}{r^n}  \frac{dr}{r}d\Hn(x) \\
		 \lesssim \sum_{j\in\Z} r_j^n\ \int_{K_j\cap\Gamma} r_j^{-n} d\Hn(x) \overset{\eqref{eq:Kj estimate}}{\lesssim_{L}}\sum_{j\in\Z} r_j^n \overset{\eqref{eq:radii summable}}{\lesssim_{C_0}}1.
		\end{multline*}
		We know that $\Hn(\Gamma\cap B(0,1)\cap E)\ge\kappa$, and so we can use Chebyshev's inequality to conclude that there exist $M_0=M_0(L,C_0,\kappa)>1$ and $F\subset \Gamma\cap B(0,1)\cap E$ with $\Hn(F)\ge \frac{\kappa}{2}$ such that for all $x\in F$
		\begin{equation*}
		\int_0^{1} \frac{\mu(K(x,V,\alpha,r))}{r^n}\ \frac{dr}{r}\le M_0.
		\end{equation*}
	\end{proof}

	\appendix
	\section{Proof of \lemref{lem:approx planes converge to approx tangent}}\label{sec:appendix}
	For reader's convenience we restate \lemref{lem:approx planes converge to approx tangent} below.
	\begin{lemma}
		Let $\mu$ be a $n$-rectifiable measure. For $x\in\supp\mu$ and $r>0$ let $L_{x,r}$ denote a minimizing plane for $\beta_{\mu,2}(x,r)$, let $W'_x$ be the approximate tangent plane to $\mu$ at $x$, whenever it exists, and let $W_x=W_x'+x$. Then for $\mu$-a.e. $x\in\supp\mu$ we have
		\begin{equation}\label{eq:Lxr converge to approximate tangent 2}
		\frac{\dist_H(L_{x,r}\cap B(x,r), W_x\cap B(x,r))}{r}\xrightarrow{r\to 0 } 0.
		\end{equation}
	\end{lemma}
	\begin{proof}
		Recall that since $\mu$ is $n$-rectifiable, the density $\Theta^n(\mu,x)$ exists and satisfies $0<\Theta^n(\mu,x)<\infty$ for $\mu$-a.e. $x$. Let $M\ge 100$ be some big constant. Define
		\begin{equation*}
		E_M := \{x\in\supp\mu\ :\ M^{-1}\le \Theta^n(\mu,x)\le M\}.
		\end{equation*}
		Note that $\mu(\R^d\setminus \bigcup_{M\ge 100} E_M)=0$, and so it suffices to show that for all $M\ge 100$ \eqref{eq:Lxr converge to approximate tangent 2} holds for $\mu$-a.e. $x\in E_M$. Fix some big $M$, and set $\nu=\restr{\mu}{E_M}$. It is well-known that
		\begin{equation}\label{eq:nu density bdd}
		M^{-1}\le \Theta^n(\nu,x) = \Theta^n(\mu,x)\le M\quad \text{for $\nu$-a.e. $x\in \supp\nu$},
		\end{equation}
		which can be shown e.g. using \cite[Corollary 6.3]{mattila1999geometry} in conjunction with Lebesgue differentiation theorem. For $\nu$-a.e. $x$ the plane $W_x$ is well defined by \thmref{thm:classical tangents}, and also by \thmref{thm:beta necessary}
		\begin{equation}\label{eq:beta square function 2}
		\int_0^1\beta_{\mu,2}(x,r)^2\ \frac{dr}{r}<\infty\quad\text{for $\mu$-a.e. $x\in\R^d$}.
		\end{equation}
		Fix $x\in E_M$ such that \eqref{eq:beta square function 2} and \eqref{eq:nu density bdd} hold, and such that $W_x$ is well-defined. Once we show that \eqref{eq:Lxr converge to approximate tangent 2} holds at $x$, the proof will be finished. From now on we will suppress the subscript $x$, so that $L_{r}=: L_{x,r},\ W:=W_x$. By applying an appropriate translation, we may assume that $x=0$.
		
		Given some small $r>0$, let $A_r(y) = \frac{y}{r}$, so that $A_r(B(0,r))=B(0,1)$. Set $L_r' = A_r(L_r)$. It is easy to see that \eqref{eq:Lxr converge to approximate tangent 2} is equivalent to showing
		\begin{equation*}
		\dist_H(L_r'\cap B(0,1),W\cap B(0,1))\xrightarrow{r\to 0 } 0.
		\end{equation*}
		We will prove that the convergence above holds by contradiction. Suppose it is not true, so that there is $\varepsilon>0$ and a sequence $r_k\to 0$ such that for all $k$ we have
		\begin{equation}\label{eq:Lrk far from L}
		\dist_H(L_{r_k}'\cap B(0,1),W\cap B(0,1))\ge \varepsilon.
		\end{equation}
		
		Let $\eta>0$ be some tiny constant. Observe that by \eqref{eq:beta square function 2} for $k\ge k_0(\eta,M)$ large enough we have 
		\begin{equation}\label{eq:beta rk small}
		\beta_{\mu,2}(0,r_k)^2 \le \frac{\eta^3}{M}.
		\end{equation}
		Indeed, otherwise one could use the simple fact that $\beta_{\mu,2}(0,r)\lesssim \beta_{\mu,2}(0,2r)$ to conclude that $\int_0^1 \beta_{\mu,2}(0,r)^2\ dr/r = \infty$. Moreover, let us remark that for every $0<\delta<1/2$, if $k=k(\delta)$ is large enough, then we have $L_{r_k}'\cap B(0,\delta)\neq\varnothing$. This can be shown easily using the fact that $\Theta^n(\mu,x)\ge M^{-1}$, that $L_{r_k}$ are minimizers of $\beta_{\mu,2}(0,r_k)$, and the fact that $\beta_{\mu,2}(0,r_k)\to 0$. We leave checking the details to the reader.
		
		Now, we use the fact that for $k$ large enough $L_{r_k}'\cap B(0,\delta)\neq\varnothing$ and the compactness properties of the Hausdorff distance to conclude that there exists some subsequence (again denoted by $r_k$) such that $L_{r_k}'\cap \overline{B(0,1)}$ converges in Hausdorff distance to a compact set of the form $V\cap \overline{B(0,1)}$, where $V$ is an $n$-plane intersecting $B(0,\delta)$. Since $\delta>0$ can be chosen arbitrarily small, we get that $V$ passes through $0$. Note that by \eqref{eq:Lrk far from L}
		\begin{equation}\label{eq:L0 far from L}
		\dist_H(V\cap B(0,1),W\cap B(0,1))\ge \varepsilon.
		\end{equation}
		
		Let $B_{\eta r_k}(V)$ denote the $\eta r_k$-neighbourhood of $V$. We will show now that a large portion of measure $\nu$ in $B(0,r_k)$ is concentrated at the intersection of $B_{\eta r_k}(V)$ and $B_{\eta r_k}(W)$. 
		
		Since $V$ passes through 0, for every $r>0$ we have $A_r^{-1}(V)=V$.  Thus,
		\begin{equation}\label{eq:Lrk close to L0}
		\frac{\dist_H(L_{r_k}\cap B(0,r_k), V\cap B(0,r_k))}{r_k}\xrightarrow{k\to \infty } 0.
		\end{equation}		
		Note that for $k$ big enough
		\begin{multline*}
		\frac{1}{\nu(B(0,r_k))}\int_{B(0,r_k)}\bigg(\frac{\dist(y,V)}{r_k}\bigg)^2\ d\nu(y)\\
		\le \frac{1}{\nu(B(0,r_k))}\int_{B(0,r_k)}\bigg(\frac{\dist(y,L_{r_k})}{r_k}\bigg)^2\ d\nu(y) + \bigg(\frac{\dist_H(L_{r_k}\cap B(0,2r_k), V\cap B(0,2r_k))}{r_k}\bigg)^2\\ 
		\overset{\eqref{eq:Lrk close to L0}}{\le} \frac{r_k^n}{\nu(B(0,r_k))} \beta_{\mu,2}(0,r_k)^2 + \eta^3\overset{\eqref{eq:nu density bdd}}{\le} 2M\beta_{\mu,2}(0,r_k)^2 +\eta^3 \overset{\eqref{eq:beta rk small}}{\le} 3\eta^3.
		\end{multline*}
		It follows from Chebyshev's inequality and the estimate above that
		\begin{equation*}
		\nu(B(0,r_k)\setminus B_{\eta r_k}(V))\le \eta^{-2} \int_{B(0,r_k)}\bigg(\frac{\dist(y,V)}{r_k}\bigg)^2\ d\nu(y) \le 3\eta \nu(B(0,r)).
		\end{equation*} 
		Hence, $\nu(B(0,r_k)\cap B_{\eta r_k}(V))\ge (1-3\eta r_k)\nu(B(0,r_k)).$ On the other hand, by the definition of the approximate tangent plane $W$ and \eqref{eq:nu density bdd}, for any $0<\alpha<1$ we have
		\begin{multline*}
		\nu(K(0,W,\alpha,r_k))= \nu(B(0,r_k)) - \nu(K(0,W^{\perp},\sqrt{1-\alpha^2},r_k))\\
		\ge \nu(B(0,r_k)) - \frac{\eta}{2M} r_k^n
		\ge (1-\eta)\nu(B(0,r_k)),
		\end{multline*}
		if $k$ is large enough (depending on $\alpha,\ \eta$ and $M$). Note that $K(0,W,\alpha,r_k)\subset B_{\alpha r_k}(W)\cap B(0,r_k)$. Thus, choosing $\alpha = \eta$, if we define
		\begin{equation*}
		S = S(k,\eta) = B(0,r_k)\cap B_{\eta r_k}(V) \cap B_{\eta r_k}(W),
		\end{equation*}
		then by the two previous estimates we have
		\begin{equation}\label{eq:S large}
		\nu(S)\ge (1-4\eta)\nu(B(0,r_k))\ge \frac{1}{2M} r_k^n,
		\end{equation}
		where in the second inequality we used \eqref{eq:nu density bdd}.
		
		We will show that if $\eta$ is chosen small enough (depending on $\varepsilon$, the constant from \eqref{eq:L0 far from L}), then the estimate above leads to a contradiction. Roughly speaking, \eqref{eq:S large} means that a lot of measure is concentrated in the intersection of $B_{\eta r_k}(V)$ and $B_{\eta r_k}(W)$, but since $V$ and $W$ are somewhat well-separated by \eqref{eq:L0 far from L}, this intersection behaves approximately like an $(n-1)$-dimensional set.
		
		Let us start by exploiting \eqref{eq:L0 far from L}. By the definition of Hausdorff distance and the fact that $V$ and $W$ are $n$-planes, it follows from easy linear algebra that there exists some $w\in W^{\perp}$ with $|w|=1$ and $|\pi_{V}(w)|\ge\varepsilon$. Let $v_1  = \pi_{V}(w)/|\pi_{V}(w)|$, and let $V_0\subset V$ be the orthogonal complement of $\spn(v_1)$ in $V$.
		
		We define $\bbox = \bbox(k,\eta)$ to be a tube-like set defined as
		\begin{equation*}
		\bbox = \bbox(k,\eta)
		= \{z\in\R^d\ :\ |z\cdot v_1|\le 2\eta\varepsilon^{-1} r_k,\ |\pi_{V_0}(z)|\le r_k,\
		|\pi_{V}^{\perp}(z)|\le \eta r_k \}.
		\end{equation*}
		We claim that $S(k,\eta)\subset \bbox(k,\eta)$. Indeed, let $z\in S$. The estimate $|\pi_{V_0}(z)|\le r_k$ is trivial since $S\subset B(0,r_k)$. The estimate $|\pi_{V}^{\perp}(z)|\le \eta r_k$ follows from the fact that $z\in B_{\eta r_k}(V)$. Concerning $|z\cdot v_1|$, note that since $z\in B_{\eta r_k}(W)$ and $w\in W^{\perp}$, we have $|z\cdot w|\le \eta r_k$. We can use our choice of $w$ and $v_1=\pi_{V}(w)/|\pi_{V}(w)|$ to get
		\begin{multline*}
		\eta r_k \ge |z\cdot w| = |z\cdot \pi_{V}(w)+ z\cdot\pi^{\perp}_{V}(w)|\\
		\ge |z\cdot \pi_{V}(w)| - |z\cdot\pi^{\perp}_{V}(w)|
		= {|z\cdot v_1|}{|\pi_{V}(w)|} - |\pi^{\perp}_{V}(z)\cdot\pi^{\perp}_{V}(w)|
		\\ \ge {|z\cdot v_1|}{\varepsilon} - |\pi^{\perp}_{V}(z)| |\pi^{\perp}_{V}(w)| \ge {|z\cdot v_1|}{\varepsilon} - \eta r_k,
		\end{multline*}
		where in the last inequality we used again $z\in B_{\eta r_k}(V)$. Thus, we have $|z\cdot v_1|\le 2\eta \varepsilon ^{-1} r_k$, and the proof of $S(k,\eta)\subset\bbox(k,\eta)$ is finished.
		
		Choose $\eta = \gamma \varepsilon$ for some tiny $\gamma=\gamma(M)>0$, and let $k$ be large enough for \eqref{eq:S large} to hold. It follows from the definition of $\bbox$ that we can cover $\bbox$ with a family of balls $\{B_i\}_{i\in I}$ such that $r(B_i) = \eta r_k$ and $\# I \lesssim \varepsilon^{-1} \eta^{-(n-1)}$. It is well-known that \eqref{eq:nu density bdd} implies that for all $y\in\R^d$ and $r>0$ we have $\nu(B(y,r))\le M r^n$. In particular, for each $i\in I$ we have $\nu(B_i)\le M (\eta r_k)^n$. Thus,
		\begin{equation*}
		\frac{1}{2M} r_k^n \overset{\eqref{eq:S large}}{\le} \nu(S)\le \nu(\bbox)\le \sum_{i\in I}\nu(B_i) \le \#I M (\eta r_k)^n \lesssim \varepsilon^{-1} \eta^{-(n-1)}M (\eta r_k)^n = \varepsilon^{-1} \eta M r_k^n.
		\end{equation*}
		That is,
		\begin{equation*}
		M^{-2}\lesssim \varepsilon^{-1}\eta = \gamma.
		\end{equation*}
		This is a contradiction for $\gamma=\gamma(M)$ small enough. Hence, \eqref{eq:Lrk far from L} is false, and so \eqref{eq:Lxr converge to approximate tangent 2} holds for $\mu$-a.e. $x\in E_M$. Taking $M\to\infty$ finishes the proof.
	\end{proof}



\begin{thebibliography}{CAMT19}
	\expandafter\ifx\csname url\endcsname\relax
	\def\url#1{\texttt{#1}}\fi
	\expandafter\ifx\csname doi\endcsname\relax
	\def\doi#1{\burlalt{doi:#1}{http://dx.doi.org/#1}}\fi
	\expandafter\ifx\csname urlprefix\endcsname\relax\def\urlprefix{URL }\fi
	\expandafter\ifx\csname href\endcsname\relax
	\def\href#1#2{#2}\fi
	\expandafter\ifx\csname burlalt\endcsname\relax
	\def\burlalt#1#2{\href{#2}{#1}}\fi
	
	\bibitem[ADT16]{azzam2016wasserstein}
	J.~Azzam, G.~David, and T.~Toro.
	\newblock {Wasserstein} distance and the rectifiability of doubling measures:
	part {I}.
	\newblock {\em Math. Ann.}, 364(1-2):151--224, 2016.
	\newblock \doi{10.1007/s00208-015-1206-z}.
	
	\bibitem[AM16]{azzam2016characterization}
	J.~Azzam and M.~Mourgoglou.
	\newblock A characterization of $1$-rectifiable doubling measures with
	connected supports.
	\newblock {\em Anal. PDE}, 9(1):99--109, 2016.
	\newblock \doi{10.2140/apde.2016.9.99}.
	
	\bibitem[AT15]{azzam2015characterization}
	J.~Azzam and X.~Tolsa.
	\newblock Characterization of $n$-rectifiability in terms of {Jones'} square
	function: Part {II}.
	\newblock {\em Geom. Funct. Anal.}, 25(5):1371--1412, 2015.
	\newblock \doi{10.1007/s00039-015-0334-7}.
	
	\bibitem[ATT20]{azzam2018characterization}
	J.~{Azzam}, X.~{Tolsa}, and T.~{Toro}.
	\newblock {Characterization of rectifiable measures in terms of
		$\alpha$-numbers}.
	\newblock {\em Trans. Amer. Math. Soc.}, 373(11):7991--8037, 2020.
	\newblock \doi{10.1090/tran/8170}.
	
	\bibitem[Azz21]{azzam2019semi}
	J.~Azzam.
	\newblock Semi-Uniform Domains and the {$A_{\infty}$} Property for Harmonic
	Measure.
	\newblock {\em Int. Math. Res. Not. IMRN}, 2021(9):6717--6771, 2021.
	\newblock \doi{10.1093/imrn/rnz043}.
	
	\bibitem[Bad19]{badger2018generalized}
	M.~Badger.
	\newblock Generalized rectifiability of measures and the identification
	problem.
	\newblock {\em Complex Anal. Synerg.}, 5(1):2, 2019.
	\newblock \doi{10.1007/s40627-019-0027-3}.
	
	\bibitem[BJ94]{bishop1994}
	C.~J. Bishop and P.~W. Jones.
	\newblock {Harmonic measure, $L^2$-estimates and the Schwarzian derivative}.
	\newblock {\em J. Anal. Math.}, 62(1):77--113, 1994.
	\newblock \doi{10.1007/BF02835949}.
	
	\bibitem[BN21]{badger2020radon}
	M.~Badger and L.~Naples.
	\newblock {Radon} measures and {Lipschitz} graphs.
	\newblock {\em Bull. Lond. Math. Soc.}, 53(3):921--936, 2021.
	\newblock \doi{10.1112/blms.12473}.
	
	\bibitem[BS15]{badger2015multiscale}
	M.~Badger and R.~Schul.
	\newblock Multiscale analysis of $1$-rectifiable measures: necessary
	conditions.
	\newblock {\em Math. Ann.}, 361(3-4):1055--1072, 2015.
	\newblock \doi{10.1007/s00208-014-1104-9}.
	
	\bibitem[BS16]{badger2016two}
	M.~Badger and R.~Schul.
	\newblock Two sufficient conditions for rectifiable measures.
	\newblock {\em Proc. Amer. Math. Soc.}, 144(6):2445--2454, 2016.
	\newblock \doi{10.1090/proc/12881}.
	
	\bibitem[BS17]{badger2017multiscale}
	M.~Badger and R.~Schul.
	\newblock Multiscale Analysis of 1-rectifiable Measures {II}:
	Characterizations.
	\newblock {\em Anal. Geom. Metr. Spaces}, 5(1):1--39, 2017.
	\newblock \doi{10.1515/agms-2017-0001}.
	
	\bibitem[CAMT19]{conde-alonso2019failure}
	J.~M. Conde-Alonso, M.~Mourgoglou, and X.~Tolsa.
	\newblock Failure of {$L^2$} boundedness of gradients of single layer
	potentials for measures with zero low density.
	\newblock {\em Math. Ann.}, 373(1):253--285, 2019.
	\newblock \doi{10.1007/s00208-018-1729-1}.
	
	\bibitem[CKRS10]{csornyei2010upper}
	M.~Cs{\"o}rnyei, A.~K{\"a}enm{\"a}ki, T.~Rajala, and V.~Suomala.
	\newblock Upper conical density results for general measures on {$\R^n$}.
	\newblock {\em Proc. Edinb. Math. Soc. (2)}, 53(2):311--331, 2010.
	\newblock \doi{10.1017/S0013091508001156}.
	
	\bibitem[CT20]{chang2017analytic}
	A.~Chang and X.~Tolsa.
	\newblock Analytic capacity and projections.
	\newblock {\em J. Eur. Math. Soc. (JEMS)}, 22(12):4121--4159, 2020.
	\newblock \doi{10.4171/JEMS/1004}.
	
	\bibitem[D{\k{a}}b20a]{dabrowski2019necessary}
	D.~D{\k{a}}browski.
	\newblock Necessary condition for rectifiability involving {Wasserstein}
	distance {$W_2$}.
	\newblock {\em Int. Math. Res. Not. IMRN}, 2020(22):8936–8972, 2020.
	\newblock \doi{10.1093/imrn/rnaa012}.
	
	\bibitem[D{\k{a}}b20b]{dabrowski2020two}
	D.~D{\k{a}}browski.
	\newblock Two examples related to conical energies.
	\newblock {\em To appear in Ann. Acad. Sci. Fenn. Math.}, 2020,
	\burlalt{arXiv:2011.12717}{http://arxiv.org/abs/2011.12717}.
	
	\bibitem[D{\k{a}}b21]{dabrowski2019sufficient}
	D.~D{\k{a}}browski.
	\newblock Sufficient condition for rectifiability involving {Wasserstein}
	distance {$W_2$}.
	\newblock {\em To appear in J. Geom. Anal.}, 2021.
	\newblock \doi{10.1007/s12220-020-00603-y}.
	
	\bibitem[Dav88]{david1988operateurs}
	G.~David.
	\newblock Op\'erateurs d'int\'egrale singuli\`ere sur les surfaces
	r\'eguli\`eres.
	\newblock {\em Ann. Sci. \'Ec. Norm. Sup\'er. (4)}, 21(2):225--258, 1988.
	\newblock \doi{10.24033/asens.1557}.
	
	\bibitem[Dav91]{david1991wavelets}
	G.~David.
	\newblock {\em Wavelets and Singular Integrals on Curves and Surfaces}, volume
	1465 of {\em Lecture Notes in Math.}
	\newblock Springer, Berlin, Heidelberg, 1991.
	\newblock \doi{10.1007/BFb0091544}.
	
	\bibitem[DM00]{david2000removable}
	G.~David and P.~Mattila.
	\newblock Removable sets for {Lipschitz} harmonic functions in the plane.
	\newblock {\em Rev. Mat. Iberoam.}, 16(1):137--215, 2000.
	\newblock \doi{10.4171/RMI/272}.
	
	\bibitem[DNI19]{del2019geometric}
	G.~Del~Nin and K.~O. Idu.
	\newblock Geometric criteria for {$C^{1, \alpha}$} rectifiability.
	\newblock {\em Preprint}, 2019,
	\burlalt{arXiv:1909.10625}{http://arxiv.org/abs/1909.10625}.
	
	\bibitem[DS91]{david1991singular}
	G.~David and S.~Semmes.
	\newblock Singular integrals and rectifiable sets in {$\mathbb{R}^n$}:
	Au-del\`{a} des graphes lipschitziens.
	\newblock {\em Ast{\'e}risque}, 193, 1991.
	\newblock \doi{10.24033/ast.68}.
	
	\bibitem[DS93a]{david1993analysis}
	G.~David and S.~Semmes.
	\newblock {\em Analysis of and on Uniformly Rectifiable Sets}, volume~38 of
	{\em Math. Surveys Monogr.}
	\newblock Amer. Math. Soc., Providence, RI, 1993.
	
	\bibitem[DS93b]{david1993quantitative}
	G.~David and S.~Semmes.
	\newblock {Quantitative rectifiability and Lipschitz mappings}.
	\newblock {\em Trans. Amer. Math. Soc.}, 337(2):855--889, 1993.
	\newblock \doi{10.1090/S0002-9947-1993-1132876-8}.
	
	\bibitem[ENV14]{eiderman2014s}
	V.~Eiderman, F.~Nazarov, and A.~Volberg.
	\newblock The {$s$-Riesz} transform of an $s$-dimensional measure in
	{$\mathbb{R}^2$} is unbounded for $1<s<2$.
	\newblock {\em J. Anal. Math.}, 122(1):1--23, 2014.
	\newblock \doi{10.1007/s11854-014-0001-1}.
	
	\bibitem[ENV16]{edelen2016quantitative}
	N.~Edelen, A.~Naber, and D.~Valtorta.
	\newblock Quantitative Reifenberg theorem for measures.
	\newblock {\em Preprint}, 2016,
	\burlalt{arXiv:1612.08052}{http://arxiv.org/abs/1612.08052}.
	
	\bibitem[Fed47]{federer1947varphi}
	H.~Federer.
	\newblock The ($\varphi$, k) rectifiable subsets of $n$ space.
	\newblock {\em Trans. Amer. Math. Soc.}, 62(1):114--192, 1947.
	\newblock \doi{10.2307/1990632}.
	
	\bibitem[GG20]{ghinassi2020menger}
	S.~Ghinassi and M.~Goering.
	\newblock {Menger} curvatures and {$C^{1,\alpha}$} rectifiability of measures.
	\newblock {\em Arch. Math. (Basel)}, 114(4):419--429, 2020.
	\newblock \doi{10.1007/s00013-019-01414-6}.
	
	\bibitem[Ghi20]{ghinassi2017sufficient}
	S.~Ghinassi.
	\newblock Sufficient conditions for {$C^{1,\alpha}$} parametrization and
	rectifiability.
	\newblock {\em Ann. Acad. Sci. Fenn. Math.}, 45:1065--1094, 2020.
	\newblock \doi{10.5186/aasfm.2020.4557}.
	
	\bibitem[Gra14a]{grafakos2014classical}
	L.~Grafakos.
	\newblock {\em Classical {Fourier} Analysis}, volume 249 of {\em Grad. Texts in
		Math.}
	\newblock Springer, New York, 3rd edition, 2014.
	\newblock \doi{10.1007/978-1-4939-1194-3}.
	
	\bibitem[Gra14b]{grafakos2014modern}
	L.~Grafakos.
	\newblock {\em {Modern Fourier analysis}}, volume 250 of {\em Grad. Texts in
		Math.}
	\newblock Springer, New York, 3rd edition, 2014.
	\newblock \doi{10.1007/978-1-4939-1230-8}.
	
	\bibitem[GS19]{girela-sarrion2018}
	D.~Girela-Sarri{\'o}n.
	\newblock Geometric conditions for the {$L^2$}-boundedness of singular integral
	operators with odd kernels with respect to measures with polynomial growth in
	{$\mathbb{R}^d$}.
	\newblock {\em J. Anal. Math.}, 137(1):339--372, 2019.
	\newblock \doi{10.1007/s11854-018-0075-2}.
	
	\bibitem[JM00]{joyce2000set}
	H.~Joyce and P.~M{\"o}rters.
	\newblock {A Set with Finite Curvature and Projections of Zero Length}.
	\newblock {\em J. Math. Anal. Appl.}, 247(1):126--135, 2000.
	\newblock \doi{10.1006/jmaa.2000.6831}.
	
	\bibitem[Jon90]{jones1990rectifiable}
	P.~W. Jones.
	\newblock Rectifiable sets and the traveling salesman problem.
	\newblock {\em Invent. Math.}, 102(1):1--15, 1990.
	\newblock \doi{10.1007/BF01233418}.
	
	\bibitem[K{\"a}e10]{kaenmaki2010upper}
	A.~K{\"a}enm{\"a}ki.
	\newblock On upper conical density results.
	\newblock In J.~Barral and S.~Seuret, editors, {\em Recent Developments in
		Fractals and Related Fields}, pages 45--54. Birkh\"{a}user, Boston, 2010.
	\newblock \doi{10.1007/978-0-8176-4888-6_4}.
	
	\bibitem[KS08]{kaenmaki2008conical}
	A.~K{\"a}enm{\"a}ki and V.~Suomala.
	\newblock Conical upper density theorems and porosity of measures.
	\newblock {\em Adv. Math.}, 217(3):952--966, 2008.
	\newblock \doi{10.1016/j.aim.2007.07.003}.
	
	\bibitem[KS11]{kaenmaki2011nonsymmetric}
	A.~K{\"a}enm{\"a}ki and V.~Suomala.
	\newblock Nonsymmetric conical upper density and $k$-porosity.
	\newblock {\em Trans. Amer. Math. Soc.}, 363(3):1183--1195, 2011.
	\newblock \doi{10.1090/S0002-9947-2010-04869-X}.
	
	\bibitem[Ler03]{lerman2003quantifying}
	G.~Lerman.
	\newblock Quantifying curvelike structures of measures by using {$L^2$} {Jones}
	quantities.
	\newblock {\em Comm. Pure Appl. Math.}, 56(9):1294--1365, 2003.
	\newblock \doi{10.1002/cpa.10096}.
	
	\bibitem[Mat88]{mattila1988distribution}
	P.~Mattila.
	\newblock Distribution of sets and measures along planes.
	\newblock {\em J. Lond. Math. Soc. (2)}, 2(1):125--132, 1988.
	\newblock \doi{10.1112/jlms/s2-38.1.125}.
	
	\bibitem[Mat95]{mattila1999geometry}
	P.~Mattila.
	\newblock {\em Geometry of sets and measures in {Euclidean} spaces: fractals
		and rectifiability}, volume~44 of {\em Cambridge Stud. Adv. Math.}
	\newblock Cambridge Univ. Press, Cambridge, UK, 1995.
	\newblock \doi{10.1017/CBO9780511623813}.
	
	\bibitem[MMV96]{mattila1996cauchy}
	P.~Mattila, M.~S. Melnikov, and J.~Verdera.
	\newblock The {Cauchy} integral, analytic capacity, and uniform rectifiability.
	\newblock {\em Ann. of Math.}, 144(1):127--136, 1996.
	\newblock \doi{10.2307/2118585}.
	
	\bibitem[MO18a]{martikainen2018boundedness}
	H.~Martikainen and T.~Orponen.
	\newblock Boundedness of the density normalised {Jones}' square function does
	not imply $1$-rectifiability.
	\newblock {\em J. Math. Pures Appl.}, 110:71--92, 2018.
	\newblock \doi{10.1016/j.matpur.2017.07.009}.
	
	\bibitem[MO18b]{martikainen2018characterising}
	H.~Martikainen and T.~Orponen.
	\newblock Characterising the big pieces of {Lipschitz} graphs property using
	projections.
	\newblock {\em J. Eur. Math. Soc. (JEMS)}, 20(5):1055--1073, 2018.
	\newblock \doi{10.4171/JEMS/782}.
	
	\bibitem[MV09]{MAYBORODA2009}
	S.~Mayboroda and A.~Volberg.
	\newblock Boundedness of the square function and rectifiability.
	\newblock {\em C. R. Math. Acad. Sci. Paris}, 347(17):1051 -- 1056, 2009.
	\newblock \doi{10.1016/j.crma.2009.07.007}.
	
	\bibitem[Nap20]{naples2020rectifiability}
	L.~Naples.
	\newblock Rectifiability of pointwise doubling measures in {Hilbert Space}.
	\newblock {\em Preprint}, 2020,
	\burlalt{arXiv:2002.07570}{http://arxiv.org/abs/2002.07570}.
	
	\bibitem[NTV97]{nazarov1997}
	F.~Nazarov, S.~Treil, and A.~Volberg.
	\newblock {Cauchy integral and Calderón-Zygmund operators on nonhomogeneous
		spaces}.
	\newblock {\em Int. Math. Res. Not. IMRN}, 1997(15):703--726, 1997.
	\newblock \doi{10.1155/S1073792897000469}.
	
	\bibitem[NTV14a]{nazarov2014onthe}
	F.~Nazarov, X.~Tolsa, and A.~Volberg.
	\newblock On the uniform rectifiability of AD-regular measures with bounded
	Riesz transform operator: the case of codimension 1.
	\newblock {\em Acta Math.}, 213(2):237--321, 2014.
	\newblock \doi{10.1007/s11511-014-0120-7}.
	
	\bibitem[NTV14b]{nazarov2014}
	F.~Nazarov, X.~Tolsa, and A.~Volberg.
	\newblock {The Riesz transform, rectifiability, and removability for Lipschitz
		harmonic functions}.
	\newblock {\em Publ. Mat.}, 58(2):517--532, 2014.
	\newblock \doi{10.5565/PUBLMAT_58214_26}.
	
	\bibitem[Orp18]{orponen2018absolute}
	T.~Orponen.
	\newblock Absolute continuity and $\alpha$-numbers on the real line.
	\newblock {\em Anal. PDE}, 12(4):969--996, 2018.
	\newblock \doi{10.2140/apde.2019.12.969}.
	
	\bibitem[Orp21]{orponen2020plenty}
	T.~Orponen.
	\newblock Plenty of big projections imply big pieces of {Lipschitz} graphs.
	\newblock {\em To appear in Invent. Math.}, 2021.
	\newblock \doi{10.1007/s00222-021-01055-z}.
	
	\bibitem[Paj97]{pajot1997conditions}
	H.~Pajot.
	\newblock Conditions quantitatives de rectifiabilit{\'e}.
	\newblock {\em Bull. Soc. Math. France}, 125(1):15--53, 1997.
	\newblock \doi{10.24033/bsmf.2298}.
	
	\bibitem[PPT21]{prat2018L2}
	L.~Prat, C.~Puliatti, and X.~Tolsa.
	\newblock {$L^2$}-boundedness of gradients of single layer potentials and
	uniform rectifiability.
	\newblock {\em Anal. PDE}, 14(3):717--791, 2021.
	\newblock \doi{10.2140/apde.2021.14.717}.
	
	\bibitem[Pre87]{preiss1987}
	D.~Preiss.
	\newblock Geometry of Measures in $\textbf{R}^n$: Distribution, Rectifiability,
	and Densities.
	\newblock {\em Ann. of Math.}, 125(3):537--643, 1987.
	\newblock \doi{10.2307/1971410}.
	
	\bibitem[San19]{santilli2019}
	M.~Santilli.
	\newblock Rectifiability and approximate differentiability of higher order for
	sets.
	\newblock {\em Indiana Univ. Math. J.}, 68:1013--1046, 2019.
	\newblock \doi{10.1512/iumj.2019.68.7645}.
	
	\bibitem[Tol05]{tolsa2005bilipschitz}
	X.~Tolsa.
	\newblock {Bilipschitz maps, analytic capacity, and the Cauchy integral}.
	\newblock {\em Ann. of Math. (2)}, 162(3):1243--1304, 2005.
	\newblock \doi{10.4007/annals.2005.162.1241}.
	
	\bibitem[Tol09]{tolsa2008uniform}
	X.~Tolsa.
	\newblock Uniform rectifiability, {Calder{\'o}n}-{Zygmund} operators with odd
	kernel, and quasiorthogonality.
	\newblock {\em Proc. Lond. Math. Soc. (3)}, 98(2):393--426, 2009.
	\newblock \doi{10.1112/plms/pdn035}.
	
	\bibitem[Tol12]{tolsa2012mass}
	X.~Tolsa.
	\newblock Mass transport and uniform rectifiability.
	\newblock {\em Geom. Funct. Anal.}, 22(2):478--527, 2012.
	\newblock \doi{10.1007/s00039-012-0160-0}.
	
	\bibitem[Tol14]{tolsa2014analytic}
	X.~Tolsa.
	\newblock {\em Analytic capacity, the {Cauchy} transform, and non-homogeneous
		{Calder{\'o}n}-{Zygmund} theory}, volume 307 of {\em Progr. Math.}
	\newblock Birkhäuser, Cham, 2014.
	\newblock \doi{10.1007/978-3-319-00596-6}.
	
	\bibitem[Tol15]{tolsa2015characterization}
	X.~Tolsa.
	\newblock Characterization of $n$-rectifiability in terms of {Jones'} square
	function: part {I}.
	\newblock {\em Calc. Var. Partial Differential Equations}, 54(4):3643--3665,
	2015.
	\newblock \doi{10.1007/s00526-015-0917-z}.
	
	\bibitem[Tol17]{tolsa2017rectifiable}
	X.~Tolsa.
	\newblock Rectifiable measures, square functions involving densities, and the
	{Cauchy} transform.
	\newblock {\em Mem. Amer. Math. Soc.}, 245(1158), 2017.
	\newblock \doi{10.1090/memo/1158}.
	
	\bibitem[Tol19]{tolsa2017rectifiability}
	X.~Tolsa.
	\newblock Rectifiability of measures and the $\beta_p$ coefficients.
	\newblock {\em Publ. Mat.}, 63(2):491--519, 2019.
	\newblock \doi{10.5565/PUBLMAT6321904}.
	
	\bibitem[TT15]{tolsa2014rectifiability}
	X.~Tolsa and T.~Toro.
	\newblock Rectifiability via a square function and {Preiss}’ theorem.
	\newblock {\em Int. Math. Res. Not. IMRN}, 2015(13):4638--4662, 2015.
	\newblock \doi{10.1093/imrn/rnu082}.
	
	\bibitem[Vil19]{villa2019square}
	M.~Villa.
	\newblock A square function involving the center of mass and rectifiability.
	\newblock {\em Preprint}, 2019,
	\burlalt{arXiv:1910.13747}{http://arxiv.org/abs/1910.13747}.
	
	\bibitem[Vil20]{villa2018tangent}
	M.~Villa.
	\newblock Tangent points of $d$-lower content regular sets and $\beta$ numbers.
	\newblock {\em J. Lond. Math. Soc. (2)}, 101(2):530--555, 2020.
	\newblock \doi{10.1112/jlms.12275}.
	
\end{thebibliography}
\end{document}